%% file: survey_topological_types_El-Hilany__final_.tex
\documentclass[10pt]{amsart}

\usepackage{latex_base_file}
\usepackage{macros_file}
\title[Topological types of polynomial maps on the plane: A survey]{Around the topological classification problem of polynomial maps: A survey}
\author{Boulos El Hilany}
\begin{document}

\maketitle

\begin{abstract} 
The study of the topology of polynomial maps originates from classical questions in affine geometry, such as the Jacobian Conjecture, as well as from
works of Whitney, Thom, and Mather in the 1950-70s on diffeomorphism types of smooth maps. 
During that period, Thom came up with a famous construction of a one-dimensional family of real polynomial maps all sharing the same degree, but in which every polynomial map has its unique topological type.
According to his convention, the topological type of a map is preserved precisely when it is composed with homeomorphisms on both source and target spaces.

Thom also conjectured that for each pair $(n,d)$, any family of $n$--variate, degree--$d$ (complex, or real) polynomial functions has at most finitely-many topological types. Soon after, a collection of results by several mathematicians throughout the 1970s and 1980s settled this conjecture, and solved its subsequent generalization to polynomial maps. 

In this survey, we outline the historical context and highlight a range of significant works
from the 1950s to the present day that lead to the current state of the art in the study of polynomial maps' topology.
The focal point of this survey is to shed some light on the ensuing classification problem of topological types of polynomial maps.
The presentation is achieved by making a gentle introduction to several other prominent questions in affine geometry, all of which
are recounted through the lens of this classification problem.

\end{abstract}
 \markleft{}
 \markright{}
\setcounter{tocdepth}{1}
\tableofcontents
 
\section{Introduction}\label{sec:intro}
Polynomial maps between affine varieties are foundational in affine geometry. 
As efforts to investigate them increased, a number of open problems and emerging research trends came about, including topics related to their stability, singularities, and  classifications. These involve topological, analytic, bilipschitz, or semi-algebraic perspectives, which find their roots in the pioneering works of Whitney, Thom, Boardman, and Mather on smooth maps between Euclidean spaces~\cite{whitney1955singularities,boardman1967singularities,Tho69,
mather1973stratifications,mather1973thom}. Modern questions on polynomial maps evolved from problems in affine geometry, and can be found in abundance. For instance,~\textit{What would be an algebraic description of the group of (polynomial) automorphisms of $\mathbb{A}^n$?} \textit{Does an isomorphism $Y\times \mathbb{A}^k\simeq \mathbb{A}^{n+k}$ imply that $Y$ is isomorphic to $\mathbb{A}^n$~\cite{kraft1996challenging}?} \textit{How do images $f(\C^n)$ look like~\cite{chevalley1955schemas,arzhantsev2023images,Fer03}?} 
\textit{For which dimensions $n$ and $p$ does $f$ preserve a set of $k$ linearly-independent points~\cite{borsuk1957k}?} etc. 

At the same time, many classical questions in affine geometry prevail. 
It is not clear, for example, whether a polynomial map $\Cntn$ is an isomorphism, provided that the determinant of its Jacobian matrix is nowhere vanishing. 
This problem, known as the \emph{Jacobian Conjecture}, was initially proposed by Keller in the 1930s and remains unresolved for $n\ge 2$. 
A great deal of effort has been invested in this conjecture, which, while yielding many erroneous proofs, also led to a wealth of significant results that have advanced our understanding on the topic~\cite{dEss12}. Another classical question is about a characterization of the set of points in $\C^p$ outside of which a polynomial map $f:\Cntp$ is a locally-trivial $\cC^{\infty}$-- fibration. The smallest such set is called the \emph{bifurcation set} of $f$.  Wallace showed that it is bounded by a Zariski closed set for any $p\geq 1$~\cite{wallace1971linear}. 
As analytic degenerations of fibers ``at infinity'' are difficult to characterize algebraically, 
a method for computing the bifurcation set symbolically, or numerically, remains out of reach.

This survey is aimed at recounting the historical background, main results, and state of the art of various open problems in the study of the topology of complex and real polynomial maps. The materialization of this field of study can be traced back to fundamental questions posed by Whitney in the 1950s regarding singularities and the diffeomorphic stability of smooth maps between manifolds. Shortly thereafter, Thom and Mather achieved groundbreaking developments in the 1960s and 1970s, elucidating the subject~\cite{thom192singularities,mather2006stability}. The focus subsequently shifted from a diffeomorphic characterization of smooth maps to a topological one: 
\textit{In the space of all smooth maps between a fixed pair of manifolds, how large is the largest family containing topologically equivalent maps?}
In this context, one obtains a topologically equivalent map by composing with a homeomorphism on the source space, and one on the target space. 
Thom provided an approach for proving that such a family is dense, under the Whitney topology on the space of smooth maps above~\cite{goreskymemoirs}. 
The details were later published by Mather in~\cite{mather1973stratifications}.

Thom also considered topologically equivalent maps in the space of polynomial maps and map-germs;
he constructed a family of degree-nine polynomial maps $\R^3\longrightarrow\R^3$ having infinitely-many topological types~\cite{thom1962stabilite}.
This motivated the following question: \textit{For which $(d,n,p)\in\N^3$ the space $\PKknp$ of degree-$d$ (complex, or real) polynomial maps $\Kntp$ contains infinitely-many topological types?} 
He also conjectured that there are no more than finitely-many topological types of functions $\Knto$ of any fixed degree $d$. This was proven to be true by Fukuda~\cite{Fuk76}, and thus sparked the problem of generalizing this conjecture to any values of $n$ and $p$. 
Fukuda's results used newly-developed tools in Whitney's stratification theory, expanded by Thom, Mather, and Boardman~\cite{boardman1967singularities,mather1973thom}. Later, Aoki and Noguchi~\cite{aoki1980topological} proved that there are at most finitely-many topological types of maps in $\PCktt$ for any fixed $d$. This result was generalized by Sabbah~\cite{Sab83} to maps in $\Cttp$ for any $p\geq 2$. As for the remaining cases for triples $(n,p,d)$, Nakai~\cite{Nak84} proved that $\PKknp$ has infinitely-many topological types of maps whenever $d\geq 4$. Consequently, the following question emerged as the next step following the resolution of the above finiteness question.

\begin{question}\label{que:main00}
Given a sequence of integers $d,n,p\in\N$, and a field $\Bbbk\in\{\C,\R\}$, how to classify polynomial maps $\Kntp$ of degree $d$ according to their topological type?
\end{question} 
All results that are presented in this survey contribute to a common research direction, which can be summarized by Question~\ref{que:main00}. 
The \emph{degree}, $ \deg f$, of $f$ is 
 defined as the maximum of the degrees of the polynomials forming the coordinates of $f$. Its \emph{topological type}, or, \emph{equivalence} is understood as follows: Two maps $f,g:\Kntp$ are said to be \emph{topologically equivalent} if there are homeomorphisms $\varphi:\Kntn$ and $\psi:\Kptp$ satisfying
\begin{equation*}\label{eq:C-stab}
\psi\circ f = g\circ \varphi.
\end{equation*} For instance, $(x,~y)\mapsto (x+y,~y)$ is topologically equivalent to the identity, but not to $(x,~y)\mapsto (y,~y)$. 
Any significant results pertaining to Question~\ref{que:main00} would yield a more lucid depiction of polynomial map topology, thus producing a new toolkit tailored to addressing open various other questions in affine geometry.

\subsection*{Topological types beyond affine geometry}
Several fields in science and engineering profit from any progress on the study of polynomial maps. We mention computer vision~\cite{DKLT-PLMP-19}, polynomial optimization~\cite{Las15}, robotics~\cite{grasegger2024coupler}, and chemical reaction networks~\cite{craciun2005multiple}, among others. Commonly in those applications a polynomial map \( f : X \rightarrow Y \) models a problem where the space $Y$ corresponds to the collection of all inputs of a given problem, while its preimage $f^{-1}(y)$ describes its solution. In this setting, the ability to distinguish different topological types is key when comparing the sets of solutions to two different inputs. 

In robotics, for example, a \emph{linkage} refers to the assembly of rigid components that are used to transmit motion or force between different parts of a robotic system. Linkages can be represented as a \emph{realization} of a graph $G = (V, E)$ in the real three-dimensional space. That is, a map $V\longrightarrow \R^3$ taking a vertex to its set of coordinates, and determines a set of distances between vertexes connected by edges~\cite{Lam70,schicho2022and,grasegger2024coupler}. If $X=\R^{2\#V-3}$ and $Y = \R^{\# E}$, then the \emph{linkage map} $f:\XtY$, for given coordinates of the vertices, computes the edge lengths via Euclidean distance. Then, for any $y \in Y$, that is for an assignment of edge lengths, the preimage $f^{-1}(y)$ describes all possible motions, and positions this linkage may attain. 
Accordingly, comparing the topological types of linkage maps determines whether two linkages give rise to similar sets of motions.

\subsection*{Challenges and methods towards classifying topological types} 
Answering Question~\ref{que:main00} requires developing a theory, or a collection of methods, that describe all discrete objects that distinguish a polynomial map topologically.
These include, for instance, the topological degree, singularities, fundamental group of the complement of critical loci, etc.
In general, describing topological equivalence directly from the polynomial expressions of the maps proves to be an extremely challenging.
For instance, the identity map has degree one, whereas the map $(x,~y)\longmapsto (y,~y^2-x)$ has degree two, but is also invertible.
Mostly for this reason, Question~\ref{que:main00} remains unresolved except for quadratic maps on the plane~\cite{FJ17,FJM18}. 
Accordingly, one requires an effective, or combinatorial, description of exactly those objects that are topologically relevant to $f$.
These include the set of singularities of a map, its bifurcation set, the set of missing points not in the image, etc. 
One approach to achieve this is to understand their topological invariants.

Due to the finiteness of their topological classes in $\PKknp$, 
a special emphasis is put on polynomial functions, and maps on the plane. 
For polynomial functions, the bifurcation set is finite~\cite{Tho69}, and is exceptionally well understood whenever the domain is two-dimensional~\cite{vui2008critical}. 
As for polynomial maps $f$ on the plane, the bifurcation set forms the union of the \emph{discriminant} and the \emph{non-properness set}. 
The discriminant is the set of images of critical points of a map, whereas
 the non-properness set is formed by points in $f(\Bbbk^2)$ over which the map is not a locally ramified covering. 
Both those sets play a role in distinguishing topological types.
For instance, the map $(x,~y)\longmapsto (x^2,~y)$ is a degree-two ramified cover of $\C^2$, whose discriminant is a vertical line forming its branching locus. 
Hence, it cannot be topologically equivalent to the identity map.
The same can be said about he map $(x,~y)\longmapsto (x,~xy)$ as it has a non-empty non-properness set formed by the same vertical line as above.

As for methods towards Question~\ref{que:main00} in a more general setting, most of them are inherited from classical techniques developed for analytic maps; the quintessential approach is the use of jet spaces developed by Boardman, Malgrange, Mather,  Thom, and Whitney in the 1960s and 1970s~\cite{whitney1955singularities,thom1959singularities,
malgrange1962theoreme,thom1964local,boardman1967singularities,
mather1973thom}. 
Other known approaches to the study of map singularities have tools and methods from differential geometry in common. Their early developments go back to Malgrange and Thom in the 1960s~\cite{thom1959singularities,malgrange1962theoreme,
thom1964local,thom1962stabilite}. This naturally includes the study of flow behavior between manifolds, where integration on through-the-map lifted differential flows is used to create diffeomorphisms that bind together topologically equivalent maps. 
Some of these tools would later see extensions to the space of polynomial maps~\cite{FJR19}.

\subsection*{Organization of the survey} After a section devoted to presenting the main notions, notations, and conventions, we divide the survey into three parts: The first of which is~\S\ref{sec:Thom-conj}, the second consists of~\S\ref{sec:affine_singul}, and~\S\ref{sec:sing_infty}, whereas the third part is made up of~\S\ref{sec:functions} and~\S\ref{sec:top-types_construct}.  

\subsubsection*{Part I} We present in the first part the history of Question~\ref{que:main00}, which originates in the 1950s from questions of Whitney on diffeomorphic stability of smooth maps. The combined effort that ensued culminated in the classification of all triples $d,n,p\in\N$ for which the space $\PKknp$ of polynomial maps $\Kntp$ of degree $d$ has infinitely-many topological types. We make a chronological recount of the breakthroughs above that start from Whitney's early works, and spans up to the present day where new results are still being produced. 
In the course of our description, we elaborate on the main results, the novel ideas, and pertinent tools; it encompasses works of Thom, Mather, Fukuda, Sabbah, Aoki, Noguchi, Nakai, and Jelonek.

\subsubsection*{Part II} 
Studying the topology of polynomial maps implies understanding the behavior of its preimages over points in its range.
Consequently, the bifurcation set of a polynomial map plays a key role in this theory.
The behavior of the fibers in the vicinity of the bifurcation set is dictated by singularities of the maps.
These singularities can either occur inside the domain, or outside it.
We call the former ``affine singularities" and the latter ``singularities at infinity".
We will elaborate on the bifurcation set for polynomial functions later in the survey, as Part II will handle this topic for the other special case of dominant polynomial maps $f:\Kntn$. For such maps, the bifurcation set becomes more tractable as $f$ is a locally trivial fibration with a finite number of points in the fibers over generic points.
Consequently, affine singularities form a set $\cCf$ encoding the ramification points of $f$, that is, points where it is a locally ramified cover, and its discriminant $\cDf:=f(\cCf)$ forms the corresponding branching locus.
As for singularities at infinity, these are captured by the non-properness set.

The study of affine singularities can be traced back to Morse and Whitney in the 1920s and 1930s. 
They were the first to characterize smooth maps on the plane and functions that satisfy certain genericity conditions. 
In contrast, studying the non-properness set is a relatively recent endeavor, and with benefits extending beyond Question~\ref{que:main00}~\cite{Jel99,Jel99a,jelonek2005effective,Jel10}; it was initiated by Jelonek in the 1990s~\cite{Jel93} for the sake of better understanding the Jacobian Conjecture. 
The non-properness set, was first thoroughly investigated only much later by Jelonek in the 1990s, and early 2000s.
In~\S\ref{sec:affine_singul}, we survey the history, state of the art, open problems, and applications that involve singularities of affine polynomial maps. 
We then devote~\S\ref{sec:sing_infty} to laying out the intriguing history, results and applications around the non-properness set.

\subsubsection*{Part III} In this part, we illustrate how to apply some of the known results on the topology of polynomial maps for the sake of constructing complex polynomial functions $\Cnto$ and maps $\Cttt$. 
Even though there is no systematic method for achieving this, a topological classification has been obtained for the quadratic case~\cite{FJ17,FJM18}.

The methods and strategies for Question~\ref{que:main00} differ significantly depending on whether the range has dimension one or not. 
We separate Part III of this paper into sections~\S\ref{sec:functions} and~\S\ref{sec:top-types_construct} accordingly. 
In~\S\ref{sec:functions}, we elaborate on classical results that compute the bifurcation set for polynomial functions, and pertinent to Question~\ref{que:main00}. We then provide a new non-trivial lower bound on the number of topological types of polynomial functions with a given degree $d$, and mention two strategies towards further constructing topological types of functions. 
Finally, in~\S\ref{sec:top-types_construct}, we introduce known results and methods for distinguishing topological types of complex polynomial maps on the plane. Furthermore, we present known methods for constructing topological types of polynomial maps $\Cttt$, and for giving lower bounds on their numbers.

The difference in topologies between real and complex maps presents a challenge when addressing Question~\ref{que:main00}. 
Therefore, our primary focus is on complex polynomial maps, with some discussion of results for real maps.

\section{Preliminaries and Notation}\label{sec:preliminaries}
In this section, we recall some main definitions and notation that will be used in the rest of this survey. 
More nuanced details are introduced later when needed.

\subsection{Basic notions on smooth maps and map-germs}\label{sub:basic_notions}
Let $X$ and $Y$ be smooth or analytic manifolds of dimensions $n$ and $p$ respectively. 
In what follows, we use $\mathcal{C}^\infty(X, Y)$ to denote the space of all smooth (holomorphic) maps $f \colon X \to Y$. 
That is, for each $k\in\mathbb{N}$, the map $f$ is $k$-times differentiable at every $x \in X$, and we use
\begin{align*}
(df)_x \colon\; T_x X \longrightarrow T_{f(x)} Y
\end{align*}
to denote the differential at $x$ of $f$.

\begin{definition}[{{\cite[Chap.~II,~\S 1]{golubitsky2012stable}}}]\label{def:singularities}
Let $f \colon X \to Y$ be a $\mathcal{C}^{\infty}$ map between two manifolds, and let $x \in X$. Then,
\begin{enumerate}
    \item $\corank (df)_x := \min(n, p) - \rank (df)_x$,
    \item $x$ is a \emph{critical point of $f$}, or a \emph{singularity}, if $\corank (df)_x > 0$. Denote by $\cC_f$ the set of critical points of $f$.
    \item The set $f(\cC_f)$ of critical values of $f$, which we denote by $\cD_f$, is called the \emph{discriminant}.
    \item If $x \notin \cC_f$, then $x$ is called \emph{regular} and $f(x)$ is called a \emph{regular value}.
\end{enumerate}
\end{definition}

Let us describe the behavior of $f$ over its critical values and its regular values.
Thanks to Sard's Theorem~\cite[Chap.~II,~\S 1]{golubitsky2012stable}, it follows that the discriminant of a smooth map $f \in \mathcal{C}^\infty(X, Y)$ has Lebesgue measure zero in $Y$.
As for regular values, a variation of the implicit function theorem states that for each regular value $y \in Y$, its preimage $f^{-1}(y)$ is a submanifold of $X$.
Furthermore, if $y \in \operatorname{im}(f)$, then the codimension of $f^{-1}(y)$ in $X$ is equal to the dimension of $Y$, and the tangent space of $f^{-1}(y)$ at $x$ is equal to $\ker ((df)_x)$. In fact, around regular points $x$, the map $f$ is a \emph{submersion} if $n \geq p$, and an \emph{immersion} if $n \leq p$. Namely, the differential $(df)_x$ is a linear map that is surjective in the first case, and injective in the second case.

Whenever $n=p$, over regular values, the map $(df)_x$ is a bijection. Then we call $f$ a \emph{diffeomorphism}. Equivalently, a diffeomorphism is a smooth map that is a homeomorphism and whose inverse is also smooth.

\subsubsection{Fibrations}\label{sss:fibrations} 
In this part, we briefly define the bifurcation set of a smooth map.
\begin{definition}
 Let $\mathcal{G}$ be the graph in $X \times Y$ of a smooth (analytic) map $f \colon X \to Y$ between smooth (analytic) manifolds and consider the map
\begin{align*}
P := \left. \pi \right|_{\mathcal{G}} \colon\; \mathcal{G} \longrightarrow Y,
\end{align*}
where $\pi \colon X \times Y \rightarrow Y$, $(x, y) \mapsto y$. We say that $f$ is a trivial $\mathcal{C}^{\infty}$-fibration (hereafter, simply 'trivial fibration') if for every $y \in \pi(\mathcal{G})$, the set $Y \times f^{-1}(y)$ is diffeomorphic to $P^{-1}(\pi(\mathcal{G}))$. 
The map $f$ is said to be a \emph{locally trivial fibration} if, for each $x \in X$, there exists a neighborhood $U$ of $x$ for which the restricted map $\restr{f}_{U} \colon U \to Y$ is a trivial fibration.
\end{definition}

For the map $f$ in the definition above, there is a set $B \subset Y$ such that the restricted map 
\begin{align}\label{eq:restricted_map}
\restr{f}_{X \setminus f^{-1}(B)} \colon\; X \setminus f^{-1}(B) \longrightarrow Y \setminus B
\end{align}
is a trivial fibration.
The smallest such subset $B \subset Y$ for which this restriction is a trivial fibration is called the \emph{bifurcation set} of $f$ and is denoted by $\mathcal{B}_f$.

A point in the bifurcation set of $f$ can lie either in the discriminant of $f$, or in the \emph{bifurcation set at infinity}, $\mathcal{B}^\infty_f$. Namely, the set $\mathcal{B}^\infty_f$ consists of points $y_0 \in Y$ at which for any neighborhood $\mathcal{U}\ni y_0$, and for any compact subset $K \subset X$, the restricted map
\begin{align*}
\restr{f}_{X \setminus K} \colon\; X \setminus K \longrightarrow \mathcal{U},
\end{align*}
is not a locally trivial $\mathcal{C}^{\infty}$-fibration.

\subsubsection{Polynomial maps}\label{sss:pol_maps}
Given any field $\mathbb{K}$ we consider dominant polynomial maps $f \colon X \to Y$ which are given by restrictions of the form $f := g|_X$ of polynomial maps\\
 $g := (g_1,\ldots, g_p) \colon \mathbb{K}^n \to \mathbb{K}^p$ onto a smooth affine variety $X \subset \mathbb{K}^n$, where $Y \subset \mathbb{K}^p$ is the Zariski closure of $g(X)$.
The variety $X$ is defined as the vanishing locus of the ideal generated by some polynomials $h_1, \ldots, h_k \in \mathbb{K}[x_1, \ldots, x_n]$.

If~$\mathbf{0}$ denotes the point $(0, \ldots, 0)$ in any vector space, then the inverse image of $\mathbf{0}$, $f^{-1}(\mathbf{0})$, is the affine subvariety in $X$, given as the vanishing locus of the ideal $\langle g_1, \ldots, g_p, h_1, \ldots, h_k \rangle$.
We thus write $f^{-1}(\mathbf{0}) = V_{\mathbb{K}}(f)$.
As the torus $(\mathbb{K}^*)^n$ will be a prominent object later, we adopt the notation $V^*_{\mathbb{K}}(f) := f^{-1}(\mathbf{0}) \cap (\mathbb{K}^*)^n$.

Due to a collection of results of Wallace, Varchenko, and Verdier~\cite{Tho69,wallace1971linear,varchenko1972theorems,verdier1976stratifications}, it is known that the bifurcation set of $f$ is a proper algebraic (or semi-algebraic in the real case) subset of $Y$.
This result implies that the fiber of a generic point under $f$ is a smooth algebraic subvariety of $X$.
The notion of generic here is given as follows.

\begin{definition}\label{def:generic}
The statement ``a generic point of a space $S$ has the property $W$'' means that there exists an open dense subset $U \subset S$ such that every element of $U$ has the property $W$.
We take the Zariski topology whenever $S$ is an algebraic variety.
\end{definition}

Several notions of analytic maps can be expressed symbolically using the equations of the polynomial.
An important object associated to the map $f$ is its \emph{Jacobian matrix} in $\mathbb{K}[x_1, \ldots, x_n]^{n \times (k + p)}$, defined as
\[
\operatorname{Jac}_x f :=
\begin{bmatrix}
\frac{\partial h_1}{\partial x_1} & \cdots & \frac{\partial h_k}{\partial x_1} & \frac{\partial g_1}{\partial x_1} & \cdots & \frac{\partial g_p}{\partial x_1} \\
\vdots & & \vdots & \vdots & & \vdots \\
\frac{\partial h_1}{\partial x_n} & \cdots & \frac{\partial h_k}{\partial x_n} & \frac{\partial g_1}{\partial x_n} & \cdots & \frac{\partial g_p}{\partial x_n}
\end{bmatrix},
\]
where the tuple $(x_1, \ldots, x_n)$ denotes the local coordinates of the Euclidean space $\mathbb{K}^n$.
For any $p \in X$, we use $(\operatorname{Jac}_x f)(p)$ to denote the matrix in $\mathbb{K}^{n \times (k+p)}$ obtained by evaluating all entries of $\operatorname{Jac}_x f$ at $p$.

In analogy to Definition~\ref{def:singularities}, we have that  $f$ is dominant whenever $\operatorname{Jac}_x f$ has full rank.
Furthermore, the critical locus $\cC_f$ of $f$ is the set of all $x\in X$ at which the rank of $\operatorname{Jac}_x f$ drops.
Indeed, it is enough to check that the differential of $f$ is obtained as the Jacobian matrix of the map $(g_1, \ldots, g_p, h_1, \ldots, h_k) \colon \mathbb{C}^n \to \mathbb{C}^{p+k}$.
The Zariski closure of $\cD_f = f(\cC_f)$ is a proper algebraic subvariety of $Y$.
Furthermore, the preimage $f^{-1}(y)$ of a regular value $y \notin \cD_f$ is a smooth algebraic subvariety of $X$.
This fact is sometimes referred to as the \emph{Bertini Theorem}~\cite{Jean-Pierre_Jouanolou}.

\subsubsection{Map-germs and singularities}\label{sss:map-germs_sing}
With the notation above, let $x$ be a point in $X$.
Two maps $f$ and $g$, defined on neighborhoods of $x$, are equivalent near $x$ iff $f = g$ on some neighborhood of $x$.
If $f \colon U \to Y$ is defined over a neighborhood $U$ of $x$, 
then the equivalence class of $f$ in the equivalence relation defined above is called the \emph{germ of $f$ at $x$}.
We use $(f, x) \colon (X, x) \longrightarrow (Y, f(x))$ to denote this germ.

For any subset $Z \subset X$, a point $x_0 \in Z$ is called a \emph{regular point} of $X$ (of dimension $n$) if $(Z, x_0)$ is a submanifold of $(X, x_0)$ of dimension $n$.
If $x_0$ is not regular (of any dimension), then $x_0$ is called a \emph{singular point} of $X$.
Next, we define an important invariant of singularities for algebraic varieties.

\begin{definition}\label{def:Milnor-number}
Assume that $V := \{f = 0\}$ is an irreducible algebraic hypersurface in $\mathbb{C}^n$, with an isolated singularity at the origin $\mathbf{0} \in \mathbb{C}^n$.
The \emph{Milnor number} of $V$ at $\mathbf{0}$ is the value
\begin{align*}
\mu_0(V) :=\; & \frac{\dim_{\mathbb{C}} \mathscr{O}_0}{\left\langle \frac{\partial f}{\partial x_1}, \ldots, \frac{\partial f}{\partial x_n} \right\rangle} \in \mathbb{N} \cup \{\infty\}
\end{align*}
where $\mathscr{O}_0$ is the local ring of analytic functions at $\mathbf{0}$. If $V$ has an isolated singular point at $x \neq \mathbf{0}$, then $\mu_x(V)$ is defined to be the Milnor number of the resulting hypersurface after translating the singularity to the origin.
\end{definition}

Germs are used to study smooth and analytic maps, as well as singularities~\cite{golubitsky}. 
For instance, an \emph{ordinary cusp} of $\mathbb{C}^2$ is a singularity of the form $(V, \mathbf{0})$, where $V$ is the curve defined by $u^2 = v^3$. 
It can be obtained as a germ $(\mathbb{C}^2, \mathbf{0}) \longrightarrow (\mathbb{C}^2, \mathbf{0})$, $(x, y) \longmapsto (x,\, y^3 + x y)$~\cite{whitney1955singularities}.

\subsection{Newton tuples and restrictions to coherent faces}\label{sub:prel_newton}
We start by recalling some background on polyhedra. For any two sets $A, B \subset \mathbb{R}^n$, their \emph{Minkowski sum}, $A \oplus B$, is the resulting set by taking the coordinate-wise vector sums of all pairs of points from $(A, B)$:
\begin{align*}
A \oplus B &:= \left\lbrace a + b\,|\, a \in A,\, b \in B \right\rbrace.
\end{align*}

\subsubsection{Tuples of polytopes}
In what follows, we use the abbreviation $[n]$ for $\{1, 2, \dots, n \}$. A convex subset $A \subset \mathbb{R}^n$ is called a \emph{polyhedron} if it is the intersection of finitely many closed half-spaces of $\mathbb{R}^n$. The boundary of one such half-space is called a \emph{supporting hyperplane} of $A$. A set $\Gamma$ is called a \emph{face} of $A$ if $\Gamma$ is the intersection of a supporting hyperplane $H$ of $A$ with its boundary, i.e., $\Gamma = H \cap \partial A$. We use the notation $\Gamma \prec A$ to indicate that $\Gamma$ is a face of $A$.

All polyhedra that are pertinent to this survey are bounded. A bounded polyhedron is referred to as a \emph{polytope}. A \emph{tuple of polytopes}, or a \emph{tuple} for short, is a map $\bm{A}$ from a finite set $K \subset \mathbb{N}$ to the set of all polytopes in $\mathbb{R}^n$. We call $K$ the \emph{support} of $\bm{A}$ and we denote it by $[\bm{A}]$. For any $k \in[\bm{A}]$, the $k$-th polytope in the tuple is denoted by $A_k$.
We call $\bm{A}_K:= \{\left. A_k \right|\, k \in K\}$ the \emph{sub-tuple} (of $\bm{A}$ associated to $K$). If $K = \{k_1, \ldots, k_p\}$, we use $\sum \bm{A}$ to denote the Minkowski sum of all elements in $\bm{A}$:
\begin{align*}
\sum \bm{A} &:=A_{k_1} \oplus \cdots \oplus A_{k_p}.
\end{align*}
Clearly, this results in a polytope in $\mathbb{R}^n$. The \emph{summands} of $\sum \bm{A}$ refers to the collection of polytopes $A_k$ forming the tuple $\bm{A}$.
A tuple $\bm{\Gamma}$ is a \emph{coherent tuple-face} of $\bm{A}$ (or, simply, \emph{face} whenever it is clear from the context) if $[\bm{\Gamma}] = [\bm{A}]$, $\Gamma_k \prec A_k$ for each $k \in [\bm{A}]$, and $\sum\bm{\Gamma} \prec \sum \bm{A}$.

The tuples that we are interested in are formed by polytopes $A \subset \mathbb{R}^n$ contained in the non-negative orthant, $\mathbb{R}_{\geq 0}^n$, and are \emph{lattice} polytopes. That is, all vertices of $A$ have integer coordinates. In what follows, we use $\cInk$ to denote the set of all tuples $\bm{A}$ of polytopes in $\mathbb{R}_{\geq 0}^n$ for which $[\bm{A}] = \{1, \ldots, p\}$.
We call them \emph{$p$-tuples}.

\subsubsection{Newton tuples}\label{sub:Newton_tuples}
A polynomial $f \in \mathbb{K}[x_1, \ldots, x_n]$ is a finite linear combination $\sum c_a x^a$ of monomials, such that $x^a = x_1^{a_1} \cdots x_n^{a_n}$, and $a = (a_1, \ldots, a_n) \in \mathbb{N}^n$. The set of all exponents $a$ appearing in $f$ and satisfying $c_a \neq 0$ is called the \emph{support} of $f$. We denote this set $\operatorname{supp}(f)$. Its convex hull, $\NP(f)$, is an integer polytope in $\mathbb{R}_{\geq 0}^n$ called the \emph{Newton polytope}.

Now, if $f$ is a tuple $f = (f_1, \ldots, f_p)$ of polynomials in $\K[x_1, \ldots, x_n]$, its \emph{Newton tuple}, $\NP(f) \in \cInk$, is given by the collection of Newton polytopes of $f_i$:
\begin{align*}
\NP(f) &:= \left(\NP(f_1), \ldots, \NP(f_p)\right).
\end{align*}
Given a collection $\bm{A} \in \cInk$, we denote by $\KA$ the space of all polynomial tuples $f$ above such that $\NP(f_i) \subset A_i$.
\begin{align*}
\KA := \big\{ (f_1, \ldots, f_n): \mathbb{K}^n \to \mathbb{K}^p \;\big|\; \NP(f_i) \subset A_i,\; i = 1, \ldots, p \big\}.
\end{align*}

The coordinates of each point $f \in \KA$ is given as the list of coefficients (with respect to some fixed ordering) in the polynomials $f_i$. This identifies $\KA$ as a Euclidean space $\mathbb{K}^N$, where $N := \#A_1 \cap \mathbb{N}^n + \cdots + \#A_n \cap \mathbb{N}^n$.

\begin{notation}\label{not:restricted}
For any subset $\sigma \subset \mathbb{R}^n$ and for any polynomial $P \in \mathbb{K}[x_1, \ldots, x_n]$, the \emph{restriction of $P$ to $\sigma$}, denoted by $P_\sigma$, is the polynomial
$
\sum_{\bm{a} \in \sigma \cap \operatorname{supp}(P)} c_{\bm{a}}\, x^{\bm{a}}.
$
\end{notation}

If $\bm{\Gamma} \prec \bm{A}$ is a face, and $f \in\KA$ is a polynomial map,
we define the collection of polynomials $f_{\bm{\Gamma}} := (f_{\Gamma_k})_{k \in [\bm{\Gamma}]}$, where for any $k \in [\bm{\Gamma}]$, the expression $f_{\Gamma_k}$ denotes the restricted polynomial $f_{k, \Gamma_k}$.

\subsubsection{Bernstein theorem}\label{sub:Bernstein}
If $A$ is a convex set in $\mathbb{R}^n$, we use $\operatorname{Vol}_n(A)$ to denote its Euclidean volume.
For any collection of $n$ convex sets $A_1, \ldots, A_n \subset \mathbb{R}^n$, their \emph{mixed volume} is the unique symmetric real-valued function 
satisfying both
\begin{align*}
V(A + B, A_2, \ldots, A_n) = V(A, A_2, \ldots, A_n) + V(B, A_2, \ldots, A_n),
\end{align*}
and $V(A, \ldots, A) = \operatorname{Vol}_n(A)$. We say that the elements of $\bm{A}$ are \emph{independent} if
\[
\dim \left( \sum A_I \right) \geq \# I
\]
for all $I \subset [\bm{A}]$.
A pair is \emph{dependent} if it is not independent.
A classical theorem of Minkowski (see, e.g.,~\cite[pp.~265]{Kho16}) states that dependence of $\bm{A}$ is equivalent to the vanishing of the Mixed volume $V(\bm{A})=0$. 
We have the following well-known version of Bernstein's Theorem~\cite{Ber75} (see, e.g.~\cite{rojas1999toric}).
\begin{theorem}\label{th:Bernstein}
Let $\bm{A} \in \cInn$ be an independent tuple, and let $f \in \KA$ be any collection of polynomials defined over an infinite algebraically closed field $\mathbb{K}$.
Then, the number of isolated points in $V_{\mathbb{K}}^*(f)$, counted with multiplicities, cannot exceed the mixed volume $V(\bm{A})$. Furthermore, we have equality if and only if the system $f_{\bm{\Gamma}} = \mathbf{0}$ has no solutions in $(\mathbb{K}^*)^n$ for any proper face $\bm{\Gamma} \prec \bm{A}$.
\end{theorem}

A consequence of Bernstein's Theorem is that any square polynomial system with a dependent Newton tuple has no isolated solutions in the torus.

\section{Finiteness of topological types: A brief history}\label{sec:Thom-conj}

\S\ref{sub:history_cusps} is devoted to the historical background on the stability of smooth maps $\XtY$ between manifolds. It begins with the early questions on diffeomorphic stability and concludes with topological stability. 
In an effort to make sense of topologically stable maps, Thom introduced his famous Isotopy Lemmas, which played a key role in Mather's proof that these maps form an open dense subset in $\CXY$. 
This, in turn, popularized Thom's Isotopy Lemmas, which have since proven to be effective in further contexts such as stratification and fibration theories.
~\S\ref{sub:Thom-isotopy} is devoted to the original formulation of these results.

Thom conjectured in the 1960s that polynomial functions of a given degree admit only finitely many topological types. 
His conjecture was subsequently validated by Fukuda in the 1970s, and then by Sabbah, Aoki and Noguchi in the 1980s who generalized the result to certain families of polynomial maps. 
Further details on these topics can be found in~\S\ref{sub:Thom-conj}, whereas~\S\ref{sub:uncountably} is dedicated to the complementary families that exhibit infinitely many topological types. 
Finally, in~\S\ref{sub:alternative}, we present some questions regarding potential alternatives to topological equivalence for polynomial maps which have also been a topic of interest in the literature.

\subsection{History of the problem}\label{sub:history_cusps}
A smooth map $f:\XtY$ between two manifolds is called $\cC^{\infty}$-\emph{stable}, if for any sufficiently nearby map $g:X\longrightarrow Y$ there are diffeomorphisms $\varphi: X\longrightarrow X$ and $\psi:Y\longrightarrow Y$ that transform $g$ into $f$. 
That is, the following diagram is commutative
\begin{equation}\label{eq:commutative_differential}
\begin{tikzcd}
X \arrow{r}{f} \arrow[swap]{d}{\phi} & Y \arrow{d}{\psi} \\%
X \arrow{r}{g}&  Y,
\end{tikzcd}
\end{equation} where the term ``sufficiently nearby'' is made rigorous “in what follows. 
Motivated by the smooth classification of map singularities, Whitney first posed the following question in the early 1950s.

\begin{question}[\cite{goreskymemoirs}]\label{que:Whitney-dense}
Do stable maps form a dense subset in $\CXY$?
\end{question}

Let us first introduce the topology most commonly used on the space $\CXY$ (see e.g.,~\cite[Chap. II~\S 2 and~\S 5]{golubitsky2012stable}). 
Let $f,g\in\CXY$ be smooth maps such that $f(x) = g(x) = :y$ for some $x\in X$. We say that $f$ has \emph{first order of contact with $g$} at $x$ if $(df)_x $ and $(dg)_x$, viewed as maps $T_xX\longrightarrow T_yY$ between their respective tangent spaces, are equal.  Accordingly, for any non-negative integer $k$, the map $f$ has \emph{$k$-th order contact with $g$ at $x$} if $(df):TX\longrightarrow TY$ has $(k-1)$--st order contact with $(dg)$ at every point in $T_xX$. This gives rise to an equivalence relation, denoted by $f\sim_k g$ at $x$. We use $J^k(X,Y)_{x,y}$ to denote the set of equivalence classes under $\sim_k $ at $x$ of maps $f:\XtY$ where $f(x) = y$. 

The space $J^k(X,Y)$ of \emph{$k$-jets of maps $\XtY$} is given by the union 
\[
\bigcup_{(x,y)\in X\times Y} J^k(X,Y)_{x,y }.
\] An element $\sigma\in J^k(X,Y)$ is called a \emph{$k$-jet of maps} (or simply a \emph{$k$-jet}) from $X$ to $Y$. Using this description, for each map $f\in\CXY$, we define the \emph{$k$-jet of the map $f$} as the application $j^kf:  X\longrightarrow J^k(X,Y)$, taking a point $x$ onto the equivalence class of $f$ in $J^k(X,Y)$.

With $J^k(X,Y)$ being a smooth manifold (see e.g.,~\cite[Chap. II, Theorem 2.7]{golubitsky2012stable}), we consider the Euclidean topology on it. Then, for any open $U\subset J^k(X,Y)$, we define the set 
\begin{align*}
M(U):= &~\left\lbrace f\in\CXY~|~j^{k}f(X)\subset U \right\rbrace.  
\end{align*} The family of sets $\{M(U)\}$ that runs over open subsets of $J^k(X,Y)$ forms a basis for a topology on $\CXY$. This topology is called the \emph{Whitney $\cC^k$-topology}, and  $W_k$ is the set of its open subsets. Finally, the \emph{Whitney $\cC^{\infty}$-topology} on $\CXY$ is the topology whose basis is $W:=\cup_{k\geq 1} W_k$ (see e.g.,~\cite[Definition II.3.1]{golubitsky2012stable}). 

The term ``dense'' in Question~\ref{que:Whitney-dense} was in terms of the Whitney $\cC^k$-topology; the latter is compatible with a metric $D$ that measures the distance between two maps in $\CXY$ using their partial derivatives (\cite[Chap. II,~\S3]{golubitsky2012stable}). Namely, an open ball around a smooth map $f\in\CXY$, is formed by all smooth $g$ satisfying $D(j^kf(x),~j^kg(x))<\delta(x)$ for all $x\in X$, where $\delta:X\longrightarrow\R^+$ is a continuous map. 
\begin{definition}\label{def:generic_smooth}
Analogously to Definition~\ref{def:generic}, we use the term \emph{generic point in $\CXY$} with respect to the Whitney $\cC^{\infty}$-topology.
\end{definition}

Question~\ref{que:Whitney-dense} has a positive answer in the cases where $\dim Y = 1 $ due to the well-established Morse theory in the 1930s~\cite{morse1925relations,morse1931critical}. Later, in the 1950s, Whitney gave a positive answer to this question whenever $\dim Y\geq 2~ \dim X$ (Whitney Immersion Theorem), and $\dim X = \dim Y = 2$ (see~\S\ref{sub:singularities_analytic} below). The full picture seemed to depend only on the pair $(n,p):=(\dim X,~\dim Y)$. Soon after, Thom showed that whenever $n=p=9$, the subset of stable maps is not dense~\cite{thom192singularities}. His example is thoroughly elaborated upon in~\cite[~\S6]{ruas2022old}, and it stands as a testament to the utility of jet spaces in the examination of the stability and singularity of maps. 

Thom's construction initiated a multitude of works, the goals of which was to distinguish all pairs $(n,p)$ for which Question~\ref{que:Whitney-dense} has a positive answer. Most notable of those were Mather's groundbreaking series of influential papers~\cite{mather1968stability1,mather1969stability2,
mather1970stability3,mather1969stability4,MATHER1970301,
mather2006stability} spanning from the 1960s till the late 1970s  that completely solved the problem.

\begin{theorem}[Mather's nice dimensions]\label{thm:Mather}
Stable maps $\XtY$ form a dense subset of $\CXY$ if and only if $(\dim X,~\dim Y)$ is one of the red dots of Figure~\ref{fig:nice_dimensions}.
\end{theorem}
Mather's classification results, and his overall legacy made a substantial impact on singularity, fibration, and stratification theories. For an outline of the proofs of Theorem~\ref{thm:Mather}, we refer the reader to the recent paper of Ruas~\cite{ruas2022old}, and to Goresky's biographical memoir on Mather's work~\cite{goreskymemoirs}.

\begin{figure}[htb]\label{fig:nice_dimensions}
\center
\includegraphics[scale=.6]{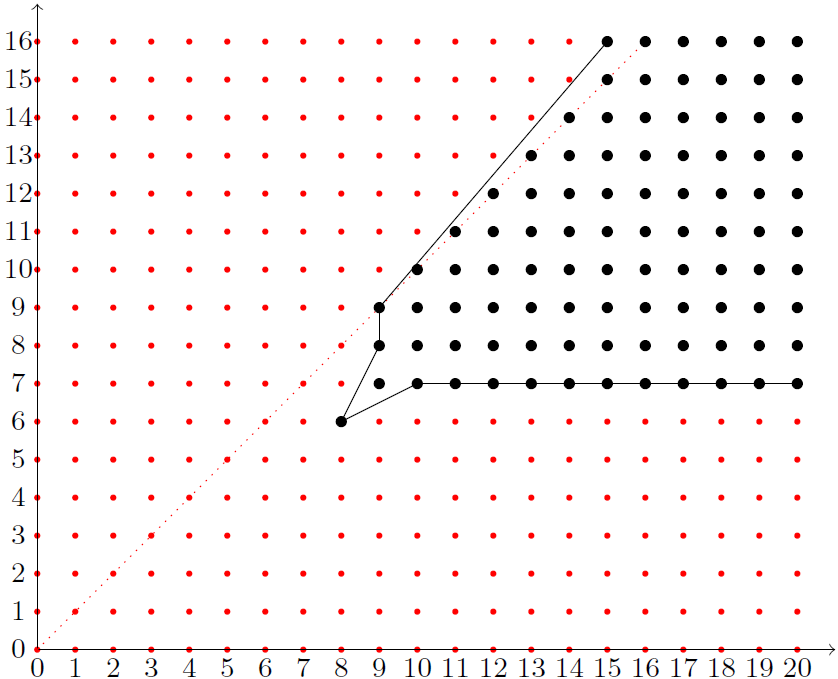}
\caption{Taken from~\cite{goreskymemoirs}: The pair of dimensions $(\dim X,~\dim Y)$ are plotted as dots, where $f:\XtY$ is a smooth map between smooth manifolds. The \emph{nice dimension} is the set of red dots.}
\end{figure}

\subsubsection{Topologically stable maps}\label{sss:diff-homeo}
Thom considered in parallel a weaker version for Question~\ref{que:Whitney-dense}; he outlined in~\cite[~\S4]{Tho69} a novel strategy on how to tackle this question if the diffeomorphisms $\varphi$ and $\psi$ in Diagram~\eqref{eq:commutative_differential} are replaced with homeomorphisms. His approach relies on constructing a stratification of the jet space using orbits under the group $\{\text{homeomorphisms }X\longrightarrow X\} \times \{\text{homeomorphisms }Y\longrightarrow Y\}$ acting on $J(X,Y)$. 
Each orbit thus represents a topological type of a smooth map. 
Inspired by this strategy, Mather later formulated in the 1970s a complete proof that topologically stable maps are dense in $\CXY$~\cite{mather1973stratifications}.

Measuring contact order of two maps at a point is done via comparing their partial derivatives (see e.g.~\cite[Chap. II, Lemma 2.2 ]{golubitsky2012stable}). 
Accordingly, whenever $X=\R^n$ and $Y=\R^p$, the space $J^k(n,p):=J^k(X,Y)$ forms the set $\PRknp$ of polynomial maps $\Rntp$ of degree at most $d:=k$. Consequently, the next logical step is to understand the behavior of polynomial maps and map-germs under topological equivalence. To this end, Thom posed the following question: 

\begin{question}\label{que:how-many}
For exactly which triples $(d,n,p)\in\N^3$ the spaces $\PRknp$ and $\PCknp$ produce infinitely-many topological types?
\end{question} 

In the pursuit of tackling this question, Thom achieved yet another groundbreaking construction;
he discovered  a one-dimensional degree-nine family of maps $\R^3\longrightarrow\R^3\in\Pol_{\R}(9\!:\!3,3)$ that contains infinitely-many topological types\\
 (see~\S\ref{sub:uncountably}). 
 In the same paper, he formulated the following conjecture.

\begin{conjecture}\label{conj:thom}
Let $\Bbbk$ be the field $\C$ or $\R$, and let $d,n\in\N$ be any two positive integers. Then, the space $\PKkno$ gives rise to at most finitely-many topological types.
\end{conjecture}
Whenever the range has dimension one, the smooth maps do not produce fibers having dimension greater than $\dim X - \dim Y$ ($=\dim X - 1$).
Whenever $\dim Y$ increases, however, this threshold may be crossed.
When this happens, the corresponding map having a fiber whose dimension is  greater than $\dim X - \dim Y$ will be referred to as ``blow-up like'' (see~\S\ref{sub:Thom-isotopy}).
Thom believed that any family of maps that are not blow-up like must have finitely-many topological types.
Accordingly, Thom's Conjecture~\ref{conj:thom} was the first step towards verifying this claim, which remains an open problem.

As for Question~\ref{que:how-many}, several mathematicians provided partial results throughout the 1970s and 1980s. 
This culminated in clarifying the vast majority of cases for $(d,n,p)$. 
\begin{theorem}[\cite{Fuk76,aoki1980topological,Sab83,Nak84}]\label{thm:top-types}
Let $d,n,p\in\N$ be any sequence of integers. Then, each of the spaces $\PCkno$, $\PRkno$ and $\PCktp$ contains at most finitely-many topological types. Furthermore, if either $d,n,p\geq 3$ or $d\geq 4$, $n\geq 3$, and $p\geq 2$, then each of the spaces $\PCknp$ and $\PRknp$ contains infinitely-many topological types.
\end{theorem} 
Theorem~\ref{thm:top-types} extends to polynomial map-germs as well, yet leaves out some gaps in the literature. For example, it is not known whether or not there are at most finitely-many topological types of quadratic complexor real polynomial maps $\Kntp$ for high enough $n$ and $p$.

\subsection{Thom's Isotopy Lemmas}\label{sub:Thom-isotopy}  
Ever since their inception in the 1960s, Thom's famous Isotopy Lemmas~\cite[Theorems 1.G.1--2]{Tho69} still play an essential role in statements on the topological types of maps. One such examples concern questions posed by Whitney on the finiteness of homotopy types of proper analytic maps~\cite{lojasiewicz1964triangulation}. For simplicity, we introduce in this section their earlier versions (Theorems~\ref{thm:Thom-lemma-first} and~\ref{thm:Thom-lemma}) as they first appeared in~\cite{thom1962stabilite}.  

\begin{definition}[\cite{thom1962stabilite,Tho69,Fuk76}]\label{def:strat-set}
A closed $k$-dimensional subspace $X\subset\R^n$ is called a \emph{stratified set} if it can be expressed as a sequence of strict inclusions $X^k\supset X^{k-1}\supset\cdots \supset X^0$, such that $X^{i}\setminus X^{i-1}$ is a smooth manifold of pure dimension $j$ in $\R^n$, with finitely-many connected components called \emph{strata}. All together, these strata satisfy the following axioms:
\begin{enumerate}

	\item the closure (in the Euclidean topology) of each stratum is a stratified set, 
	
	\item the boundary $\partial S$ of a stratum $S$ is a stratified subset having dimension smaller than $\dim S$, and

	\item the union and the intersection of stratified subsets are themselves stratified sets. 
\end{enumerate}
\vspace{-.3cm}
\end{definition} 
From Definition~\ref{def:strat-set}, the preimage $f^{-1}(Y)$ of a stratified set under a continuous map is itself a stratified set, where preimages of strata of $Y$ form a stratification of $f^{-1}(Y)$. This motivates the following notion.

\begin{definition}\label{def:stratified_map}
A smooth map $f:\XtY$ between two stratified subspaces is called \emph{stratified} if $f$ maps a stratum $S$ onto a stratum $T$, and its restriction on every stratum is a submersion. 
\end{definition} 

\begin{example}\label{exp:stratification+blow-up}
Let $B$ denote the box $\{|x|\leq 1,~|y|\leq 1\}$, and let $\mathscr{S}:=\{S_\alpha\}$ be the stratification of $B$ formed by $4$ squares, $12$ segments and $9$ points as shown in Figure~\ref{fig:strat-map}. If $f:\RtR$ denotes the map $(x,~y)\longmapsto (x^2y^2,~y)$, then $f(B)$ has a stratification $\mathscr{T}:=\{f(X_\alpha)\}$. Therefore, the restriction $g:=\fr_B$ is a stratified map. 
\end{example}

\begin{figure}[htb]
\center
\input{stratified_map}
\end{figure}

\begin{theorem}[Thom's First Isotopy Lemma]\label{thm:Thom-lemma-first}
For any proper stratified map $f:\XtY$, the restriction $\fr_S:S\longrightarrow T$ of $f$ at the preimage $S:=f^{-1}(T)$ of a stratum $T$ is a locally trivial fibration. 
\end{theorem}

This result would later foresee some generalizations and extensions~\cite{mather2012notes,dhinh2021thom}; following Thom's outline in~\cite{Tho69}, Mather later detailed a complete proof~\cite{mather2012notes}. As a by-product, he obtains the famous refined version of this result~\cite[\S 1.9]{goreskymemoirs}.
 
\begin{theorem}[Thom-Mather First Isotopy Lemma~\cite{mather2012notes}]
Let $W$ be a Whitney stratified subset of some smooth manifold $X$, let $f:\XtY$ be a smooth map whose restriction to $W$ is proper and suppose that the restriction $\fr_S:S\longrightarrow f(S)$ to each stratum $S$ is a submersion. Then $\fr_{W}:W\longrightarrow Y$ is a locally trivial fibration.
\end{theorem}
It should be mentioned that a Whitney stratification is similarly defined as in Definition~\ref{def:strat-set} with key differences where some tangency conditions between adjacent strata are assumed (called Whitney Condition B~\cite[Definition 1.4.3]{pflaum2002analytic}), and with no restriction on the finiteness of strata.

Some strata of a stratified map have a degenerate topological behavior compared to other ones. 
For instance, the map $g$ in Example~\ref{exp:stratification+blow-up} satisfies $\dim g^{-1}(0,0) = 1$, whereas $\dim g^{-1}(w) = 0$, $\forall w\in Y \setminus\{\unze\}$. More generally, we say that a stratified map $f:\XtY$ is \emph{blowup-like}~\cite{thom1962stabilite} if a pair of adjacent strata $(S,~S')$ of $X$ satisfies
\begin{equation}\label{eq:blow-up-like}
 \dim S - \dim f(S)\neq  \dim S' - \dim f(S').
\end{equation} Blow-up like maps play an important role in proving finiteness of topological types for some families of maps. This is achieved through the Second Isotopy Lemma; consider a sequence of stratified maps $X\overset{p}{\longrightarrow} Y\overset{q}{\longrightarrow}  \R^n$, where $\R^n$ is equipped with a trivial stratification. For any given $t\in \R^n$, we take the restriction map $p_t:=\left. p\right|_{X_t}:X_t\longrightarrow Y_t$ where $Y_t:=q^{-1}(t)$ and $X_t:=p^{-1}(Y_t)$.

\begin{theorem}[Thom's Second Isotopy Lemma]\label{thm:Thom-lemma}
If a sequence of stratified maps $X\overset{p}{\longrightarrow} Y\overset{q}{\longrightarrow}  \R^n$ is proper and not blowup-like, then for any $a$ and $b$ in $\R^n$, the maps $p_a$ and $p_b$ are topologically equivalent.
\end{theorem}
The second isotopy lemma was later re-formulated to fit its modern widely-used versions~\cite{Tho69,mather1973stratifications} by making further  technical assumptions. For example, the definition of stratified sets was later extended to include Whitney stratifications, whereas blowup-likeness was replaced by a more elaborate notion, later called Thom's $a_f$ condition~\cite{Tho69}. 

In his paper~\cite[Theorem 3]{thom1962stabilite}, Thom gives an outline for the proof of Theorem~\ref{thm:Thom-lemma}: The goal is to construct homeomorphisms $\Psi_{ab}:Y_a\longrightarrow Y_b$ and $\Phi_{ab}:X_a\longrightarrow X_b$ in which the following diagram commutes
\begin{equation}\label{dia:commutative}
\begin{tikzcd}
X_a \arrow{r}{p_a} \arrow[swap]{d}{\Phi_{ab}} & Y_a \arrow{d}{\Psi_{ab}} \\%
X_b \arrow{r}{p_b}&  Y_b
\end{tikzcd}
\end{equation} Consider a vector field $\partial/\partial t_i$ on $\R^n$.   Since $q:Y\longrightarrow \R^n$ and $p:X\longrightarrow \R^n$ are proper, using the First Isotopy Lemma twice, this vector field is lifted to smooth vector fields on any stratum $Y$ and, subsequently, to its preimage $f^{-1}(Y)$. 
Thanks to $p$ and $q$ being not blowup-like, lifted vector fields can be glued to global smooth vector fields $(H)$ on $Y$ and $(G)$ on $X$. 
Using properness of $p$ and $q$ one more time allows one to integrate the smooth flows corresponding to $(H)$ on $Y$ and to $(G)$ on $X$, which results in homeomorphisms $\Psi_{ab}$ by $\Phi_{ab}$. Here, $(G)$ is sent to $(H)$ via the gradient ``dot'' map $\overset{\bm{\cdot}}{p}$ corresponding to $p$.

\subsection{A conjecture of Thom}\label{sub:Thom-conj} 
Let $\Bbbk$ be the field $\C$ or $\R$. 
We identify a polynomial $f\in\Bbbk[x_1,\ldots,x_n]$ with a point in the space $\Pol_{\Bbbk}(d\!:\!n,1)\simeq \Bbbk^{N}$ ($N\in\N$) of degree-$d$ polynomials over $\Bbbk$ in $n$ variables. 

The theory of map stratifications was used in the context of polynomial maps thanks to the following setup. 
Consider the projection $P:\Bbbk^N\times\P^n_{\Bbbk}\times\Bbbk\longrightarrow \Bbbk^N\times\Bbbk$ and a sequence of two maps
\begin{equation}\label{eq:sequence-1}
\overline{\cG}\overset{\left. P\right|_{\overline{\cG}}}{\longrightarrow} \Bbbk^N\times\Bbbk\overset{\pi}{\longrightarrow}  \Bbbk^N,
\end{equation} where $\overline{\cG}\subset\Bbbk^N\times\P^n_{\Bbbk}\times\Bbbk$ is the closure of the graph $\{(f,x, y): f(x) = y\}$,
 and $\pi$ is the projection $(f,y)\longmapsto f$. The projective closure of the domain guarantees that the maps in~\eqref{eq:sequence-1} are proper.  
Thanks to the Second Isotopy Lemma, each stratum of $\Bbbk^N$ represents a unique topological type~\cite{Fuk76}. 
Furthermore, as polynomial functions are not blowup-like, one should be able to prove that maps appearing in~\eqref{eq:sequence-1} must be stratified morphisms. Fukuda elaborated on these observations, to prove Thom's conjecture in~\cite{Fuk76} by constructing the necessary compatible stratifications of~\eqref{eq:sequence-1}. 

\begin{theorem}[\cite{Fuk76}]\label{thm:Fukuda}
Thom's Conjecture~\ref{conj:thom} is true.
\end{theorem} 

One may ask whether the blowup-like property is sufficient for ensuring no more than finitely-many topological types.
\begin{question}\label{que:blow-up_like}
Let $d,n,p\in\N$ be a sequence of integers. Are there infinitely-many topological types of maps in $\PCknp$ that are not blowup-like?
\end{question}

In an effort to tackle this question, Sabbah showed in~\cite{sabbah1983morphismes} that for any polynomial map $\Cntp$, there is a sequence of blowups $Bl:\tilde{\C^p}\longrightarrow \C^p$ so that, after an appropriate change of bases, the resulting map $\tilde{f}:\C^n\longrightarrow\tilde{\C^p}$ satisfies $f= Bl\circ\tilde{f}$ and is not blowup-like. 
The \emph{center} of $Bl$ is the union of $f(X)$ at which the stratum $X$ has ``atypical'' dimension in the sense of~\eqref{eq:blow-up-like}. 
We say that two polynomial maps $f,g:\Cntp$ are \emph{topologically blowup  equivalent} if the corresponding maps $\tilde{f}$ and $\tilde{g}$ are topologically equivalent. 
\begin{theorem}[\cite{Sab83}] \label{thm:Sabbah-blowup}
The space $\PCknp$ give rise to at most finitely-many topological blowup equivalence classes.
\end{theorem}
In the same paper, Sabbah~\cite{Sab83} addressed the extension Thom's conjecture for polynomial maps on the plane. Namely, consider any smooth family $\sF\subset\PCktp$ of above maps with given degree $d$. Then, thanks to a theorem of Mather~\cite{mather2012notes}, a smooth vector field can be constructed so that its flow is lifted to the closure of the graph 
\begin{align*}
\overline{\cG}:= & \left\{(f,x, y)\in\sF\times\P^{2}\times\P^p~\left|~f_1(x) - y_1=\cdots = f_p(x) - y_p = 0\right\}\right.
\end{align*}through a sequence of maps similar to the one appearing in~\eqref{eq:sequence-1}. As the maps in question are blowup-like, this lifting can only be achieved by passing through respective blowups of $\overline{\cG}$ and of $\sF\times \P^p$. We thus obtain the following commutative  diagram
\begin{equation}\label{eq:diag-Sabbah}
\begin{tikzcd}
\tilde{\overline{\cG}} \arrow[rr, "\tilde{\Proj}"] \arrow[d] &  & \tilde{\sF\times \P^p}  \arrow[r, "\tilde{\pi}"] \arrow[d] & \sF \\
\overline{\cG} \arrow[rr, "\Proj"]           &  & \sF\times \P^p \arrow[ru, "\pi"]          &  
\end{tikzcd}
\end{equation} where the vertical maps are blowups, and $\Proj$ and $\pi$ are projections forgetting $z$ and $w$ respectively, with $\tilde{\Proj}$ and $\tilde{\pi}$ being their respective extensions. 
Unlike polynomial maps defined over spaces of dimension higher than two, the relevant blowups $\tilde{f(\C^2)}\longrightarrow f(\C^2)$, used to construct $\tilde{\sF\times\P^p}\longrightarrow \sF\times\P^p$, have zero-dimensional centers. This observation turned out to be useful for showing that all the morphisms in~\eqref{eq:diag-Sabbah} are stratified~\cite[proof of Theorem 2]{Sab83}. Then, the relevant homeomorphisms are obtained by lifting a smooth flow. This culminated in the following result which was proven by Aoki for the special case of maps $\CtC$.
\begin{theorem}[\cite{aoki1980topological,Sab83}]\label{thm:fin-many-Sab}
Let $d,p\in\N$ be any two positive integers. Then, the space $\PCktp$ gives rise to at most finitely-many topological types.
\end{theorem} 

From Sabbah's above results, it took around 40 years to find another non-trivial family containing finitely-many topological types. 
Note that complex proper polynomial maps are not blowup-like. Recently, Jelonek showed that any affine family of such maps cannot produce infinitely-many topological types.

\begin{theorem}[\cite{jelonek2016semi}]\label{thm:Jelonek-proper-finite}
Let $d,n,p\in\N$ be a sequence of integers satisfying $n\leq p$. Then, the space $\PCknp$ gives rise to at most finitely-many topological types of proper polynomial maps. 
\end{theorem}

\begin{proof}[Sketch of proof]
We present the proof only for the case $n=p$. It is enough to show that any affine family $\sF\subset \PCknn$  of proper polynomial maps contains a Zariski open subset with at most finitely-many topological types. Applied recursively, this claim yields the proof since Zariski closed subsets have finitely-many irreducible affine components. First, using Thom's first isotopy Lemma, one obtains (see~\cite[Proof of Theorem 3.5]{jelonek2016semi}) that $\sF$ contains a Zariski open subset $U$ in which
\begin{enumerate}
	\item every element $f\in U$, representing an analytic cover $\Cntn$ of covering degree $k$, and 
	\item any two maps $f,g\in U$ have homeomorphic discriminants, i.e. there is a homeomorphism $\psi:f(\C^n)\longrightarrow g(\C^n)$ satisfying $\cD_g=\psi(\cDf)$
\end{enumerate} Then, every map $f\in U$ induces an unramified covering $P_f:Q:=\C^n\setminus f^{-1}(\cDf)\longrightarrow f(\C^n)\setminus \cDf =:R$. We obtain this way a homomorphism between the fundamental groups
\[
f_*:\pi_1(Q,a)\longrightarrow \pi_1(R,b),
\] for some $b\in R$ and $a\in f^{-1}(b)$. Hence the subgroup $f_*(\pi_1(Q,a))$ has index $k$ in $\pi_1(R,b)$. Note that the fundamental group of an algebraic variety, including $\pi_1(R,b)$, is finitely generated. Consequently, by a theorem of Hall~\cite{hall1950topology}, the latter has at most finitely-many subgroups with the same index $k$. This implies that at most finitely-many homeomorphism types of such topological covers $P_f$ exist. The proof follows by extending these homeomorphisms between topological covers to the whole maps $\Cntn$. 
This follows from the properness criterion (see~\cite[Proof of Theorem 3.5]{jelonek2016semi} for details). 
\end{proof}
Theorem~\ref{thm:Jelonek-proper-finite} constitutes further evidence that families of blow-up like maps should not contain infinitely-many topological types.

\begin{remark}\label{rem:counting_types_group}
The proof of Theorem~\ref{thm:Jelonek-proper-finite} also provides a two-step recipe for constructing topological types of proper maps $\Cntp$. Namely, for each value $1\leq k\leq d^p$, the first step is to produce a list of homeomorphism types of subsets $S\subset\C^p$, such that there exists $f:\Cntp$ having covering degree $k$, and $S$ is the discriminant of $f$. Then, for each such $S$, one computes all the subgroups $f_*(\pi_1(Q,a))\leq\pi_1(R,b)$ of index $k$. Later, in~\S\ref{sec:top-types_construct} we discuss further methods for constructing topological types. 
\end{remark}

\subsection{Families of infinitely-many topological types}\label{sub:uncountably}  In his famous paper~\cite{thom1962stabilite}, Thom provided an example illustrating that his Conjecture~\ref{conj:thom} does not hold true for arbitrary dimensions; 
consider the following family $\{F_t\}_{t\in\R}$ of polynomial maps $\R_{x,y,z}^3\longrightarrow\R_{u,v,w}^3$,
\begin{equation}\label{eq:map-infinite}
\begin{cases}
	u & = \left(x~\big(x^2+y^2 - 1\big)~-2~y~z\right)^2~\cdot~\left(~(t~y+x)\cdot \big(x^2+y^2 - 1 \big)~-2~z~(y-t~x)\right) \\
v& =x^2+y^2 - 1 \\
  w & =z.
\end{cases}
\end{equation}
Thom showed in~\cite{thom1962stabilite} that the family $\{F_t\}_{t\in\R}$ contains infinitely-many topological types of polynomial maps: 
Each member $F_t$ of this family induces a stratification $\mathscr{S}_t$ of the source space $\R^3$, which includes the set of critical points of $F_t$, denoted by $\cC_t$. The latter contains a circle $S$ given by the equations
\[
S := \begin{cases}
	x^2 + y^2 -1& =0 \\
  z & =0.
\end{cases}
\] For every $t\in\R$, the map $F_t$ is blowup-like since $S\subset F_t^{-1}(\unze)\subset\R^3$. As $S$ is the only subset of $\R^3$ satisfying $\dim S - \dim F_t(S)=1$, a homeomorphism $\varphi:\R^3\longrightarrow\R^3$ relating two different maps $F_t$ to $F_{t'}$ should send $S$ to itself. The map $F_t$ sends its critical locus $F_t(\cC_t)$ onto the plane $\{u=0\}$. It follows that each $p\in\cC_t$ determines the slope $m_p$ of a line $L_p\subset\{u=0\}$ containing the origin $\unze$. In turn, from $\varphi(S) = S$, each $F_t$ determines a rotation $r_t:S\longrightarrow S$, $\theta\longmapsto\theta -\arctan t$. This implies that $\varphi$ should also commute with $r_t$. However, a homeomorphism between circles cannot be represented by a rotation~\cite{van1935topological}. This implies that $F_t$ and $F_{t'}$ are not topologically equivalent if $t\neq t'$.

 The family of polynomial maps~\eqref{eq:map-infinite} above illustrates two observations: First, homeomorphisms are often too rigid to transform one blowup-like polynomial map onto another. 
 The second observation is that the number of topological types of maps $\Kntp$ of a given degree $d$ is not expected to be finite as the dimension $p$ of the range increases.
 
  A similar construction was achieved by Nakai~\cite{Nak84} with the following family $G_{\bm{t}}:\R_{x,y,z}^3\longrightarrow\R_{u,v}^2$,
\begin{align*}\label{eq:map-infinite2}
u:=~&~(t_1~x - y)\cdot (t_2~x - y)\cdot (t_3~x - y)\\
v:=~&~(t_4~x - y)\cdot (t_5~x - y)\cdot (t_6~x - y)\cdot z,
\end{align*} for some vector $\bm{t}:=(t_1,\ldots,t_6)\in\R^6$. There exists a union $\cL({\bm{t}}):=\ell_1({\bm{t}})\cup\cdots\cup\ell_6({\bm{t}})$ of lines in $\{z=0\}$, passing through the origin, and whose slopes depend on ${\bm{t}}$. Some of them belong to the critical locus of $G_{\bm{t}}$, others form the pre-image of $(0,0)$, and all of them are mapped to the line $\{v=0\}$. The homeomorphisms corresponding to two topologically equivalent maps $G_{\bm{t}}$ and $G_{{\bm{t}}'}$ should send $\cL({\bm{t}})$ to $\cL({\bm{t}}')$, while $\{v=0\}$ remains fixed. Nakai shows that there exists a rational function $\Psi:\{t_1<\cdots < t_6\}\longrightarrow\R$ describing a relation on the slopes of the lines in $\cL(\bm{t})$ 
satisfying $\Psi({\bm{t}})\neq\Psi({\bm{t}}')$ if $G_{\bm{t}}$ is not topologically equivalent to $G_{{\bm{t}}'}$. Hence, for any two distinct $\lambda$ and $\lambda'$ outside a finite set of $\R$, we have $\Psi^{-1}(\lambda)\cap \Psi^{-1}(\lambda')=\emptyset$. This shows that $G_{\bm{t}}$ does not have the same topological type as $G_{\bm{t}'}$ if $\bm{t}\in\Psi^{-1}(\lambda)$ and $\bm{t}'\in\Psi^{-1}(\lambda')$.

By adding trivial coordinates, maps $F_t$ and $G_{\bm t}$ can be extended to construct polynomial maps for arbitrary large $n$, $p$, and $d$, that have infinitely-many topological types. This approach was the basis of Nakai's proof of the following result, which also contributed to Theorem~\ref{thm:top-types}.

\begin{theorem}[\cite{Nak84}]
Let $d,n,p\in\N$ be a sequence of integers such that either $d,n,p\geq 3$ or $d\geq 4$, $n\geq 3$, and $p\geq 2$. Then, each of the spaces $\PCknp$ and $\PRknp$ contains infinitely-many topological types.
\end{theorem}

\subsection{Alternative equivalence classes}\label{sub:alternative}  If we replace topological equivalence in Question~\ref{que:how-many} by smooth equivalence (as in Diagram~\ref{dia:commutative}), the problem becomes trivial as Whitney demonstrates~\cite[Example 1]{Fuk76}: The family $\{H_t\}_{t\in\R}:\Rtto$, 
\[
(x,~y)\longmapsto x\cdot (y + t~x)\cdot(x^2 - y^2)
\] gives rise to infinitely-many $\cC^\infty$-types. Indeed, for any $t\in\R$, the preimage $H_t^{-1}(0)$ is the only one having four lines $L(t)$. Then, any pair of diffeomorphisms $\varphi:\RtR$, $\psi:\R\longrightarrow\R$ relating two maps $H_t$ and $H_{t'}$ should send $L(t)$ to $L(t')$. Now, if $t\neq t'$, then the set of ratios of slopes obtained from the lines $L(t)$ is not equal to that from $L(t')$. 
It is shown in~\cite{Fuk76} that this set of slopes should be preserved between diffeomorphic maps. This implies that $H_t$ and $H_{t'}$ are not $\cC^\infty$-equivalent if $t\neq t'$.

Let us briefly mention another type of equivalence that has been considered recently.

\subsubsection{Regular types}\label{sss:aut-type} 
As we will see next, imposing that the homeomorphisms be polynomial results in a condition that is too restrictive for obtaining finitely-many equivalence classes with fixed degrees; two polynomial maps $f,g:\XtY$ are \emph{regularly equivalent} (or \emph{R-equivalent}) if there are polynomial automorphisms $\varphi:X\longrightarrow X$ and $\psi: Y\longrightarrow Y$ such that $\psi\circ f\circ\varphi = g$ is satisfied. 
  Lamy showed in~\cite{lamy2005structure} that a proper polynomial map $\CtC$ of topological degree two is R-equivalent to $(x,~y)\longmapsto (x^2,~y)$. In the pursuit to generalize Lamy's Theorem, Bisi and Polizzi recently showed the following result.

If $\dim X= \dim Y$, then there is a Zariski open subset $V\subset Y$ at which $f^{-1}(y)$ has exactly $k$ simple points for some $k\in\N$. 
The value $k$ is called the \emph{topological degree}, which is denoted by $\tau(f)$.

\begin{theorem}[\cite{bisi2010proper,bisi2011proper}]\label{thm:Bisi}
Let $d,k\in\N$ be such that $d\geq 4$ and $k\geq 3$. Then, there are infinitely-many R-equivalence classes of polynomial maps $\CtC$ of a given topological degree $k$. Furthermore, there are infinitely-many R-equivalence classes of polynomial maps $\CtC$ of a given (algebraic) degree $d$.
\end{theorem}
The first statement is proven by constructing the family $\{A_n\}_{n\in\N}$ of polynomial maps 
\[
(x,~y)\longmapsto (x,~y^k - k~x^ny)
\] having the same topological degree, and with exactly one isolated singularity, whose multiplicity is different for each $n$. 
The arguments for the second statement are similar as above: 
One constructs a similar family $\{B_{n}\}_{n\in \N}$ of polynomial maps of same degree $d$ and whose critical locus is $\{y^d + n~x^{d-1}~y + x^d=0\}$ (see~\cite[Theorem B1]{bisi2010proper}). 
Due to a theorem of Kang~\cite{kang1993analytic}, there is no automorphism sending the germ singularity of $\cC_{B_n}$ onto $\cC_{B_m}$ if $m^d\neq n^d$.

In contrast to Theorem~\ref{thm:Bisi}, the family of polynomial maps $\CtC$ that are Galois covers with finite Galois group $G$ contains only finitely many distinct R-equivalence classes. 
Bisi and Polizzi~\cite[Theorem 3.8]{bisi2010proper} achieved a complete classification, up to R-equivalence, of all above Galois maps. 

Several results appeared later concerning different types of equivalence classes such as \emph{holomorphic}~\cite{jelonek2017finite} and \emph{bilipschitz classes}~\cite[~\S 9]{ruas2022old}. As for the Theorem of Lamy, Jelonek later proved several generalizations. One of them~\cite[Theorem 1.12]{jelonek2017finite} goes as follows.

\begin{theorem}\label{thm:Jel-Lamy}
Let $f:\CtC$ be a proper polynomial map of topological degree $k$. 
If its discriminant $\cDf$ is isomorphic to the complex affine line $\mathbb{A}(\C)$, then $f$ is R-equivalent to the map $(x,y)\longmapsto (x^k,~y)$.
\end{theorem}

\section{Affine singularities of polynomial maps}\label{sec:affine_singul} 

Singularities were among the first objects examined in the study of the topology of a smooth map. 
This direction was mainly initiated by Morse in the 1920s before Whitney continued the subject in the 1950s by introducing the early versions of jet spaces. 
His famous results on the characterization of singularities of stable maps from the plane to the plane were later extended and refined and extended to analytic and polynomial maps. 
We recount in~\S\ref{sub:singularities_analytic} the early results of Morse and Whitney on the singularities of analytic maps.
 We then present in~\S\ref{sub:polynomial_singularities} the various extensions of Whitney's works to the polynomial setting, together with the quantitative results counting isolated singularities.

In~\S\ref{sub:Jacobian_conjecture} we shift our focus to polynomial maps without singularities; this study was initiated by Keller in the 1930s, when he formulated his famous \emph{Jacobian Conjecture}, suggesting that such a complex polynomial map $\Cntn$ should be invertible~\cite{keller1939ganze}. 
Although much work has been done on this topic, the Jacobian Conjecture remains open. 
Surprisingly, this line of research remained separate from the topic of studying topological types of maps. 
Nevertheless, many results discovered for the topological problem turned out to be useful qualitative statements for the Jacobian Conjecture (see e.g.~\S\ref{sec:sing_infty}). \S\ref{sub:Jacobian_conjecture} is devoted to a brief history of the Jacobian Conjecture and its current state of the art. 

\subsection{Singularities of stable maps}\label{sub:singularities_analytic}
One of the main questions in the theory of smooth maps in $\CXY$ is to describe the stability of their singularities. Thanks to classical Morse theory, smooth functions $X\longrightarrow\R$ were among the earliest cases to be well-understood.
\begin{theorem}[Morse~\cite{morse1925relations,morse1931critical}]\label{thm:Morse}
Assume that $\dim X = n$. 
Then, a generic function $f\in\CXR$ has  isolated singularities, and, in local coordinates $z:=(z_1,\ldots,z_n)$, each singularity $x\in X$ of $f$ can be expressed as
\begin{align*}
f(z) = & f(x) - z_1^2 -\cdots - z^2_k +z_{k+1}^2 +\cdots + z_{n}^2
\end{align*} for some $k\in\{1,\ldots,n\}$.
\end{theorem} The next well-understood situation for maps in $\CXY$ was first discovered by Whitney in the 1930s.
\begin{theorem}[Whitney Immersion Theorem~\cite{whitney1992analytic}]\label{thm:Whitney_immersion}
Let $X$ and $Y$ be two manifolds such that $\dim Y\geq 2~\dim X $. Then, a generic map $f\in\CXY$ is an immersion.
\end{theorem}

By the 1970s, the theory behind the space of jets for $\CXY$ was already well established thanks Thom, Mather, and Boardman. In turn, this lead to significant simplifications of proofs for classical theorems on the singularities of maps. One of the key results in this context was the famous Thom Transversality Theorem~\cite{thom1954quelques}. 

\begin{definition}[\cite{golubitsky2012stable}]\label{def:transversality}
Let $X$ and $Y$ be smooth manifolds and $f:\XtY$ be a smooth map. Let $W$ be a submanifold of $Y$ and $x$ a point in $X$. Then $f$ intersects $W$ \emph{transversally at} $x$ if either 
\begin{enumerate}
	\item $f(x)\not\in W$ or
	
	\item $f(x)\in W$ and $T_{f(x)}Y = T_{f(x)}W \oplus (df)_x(T_xX)$. 
\end{enumerate} 
We say that $f$ intersects $W$ \emph{transversally} (and we write $f~\overline{\pitchfork}~W$) if $f$ intersects $W$ transversally at each $x\in X$.
\end{definition}

\begin{theorem}[Thom Transversality Theorem]\label{thm:Thom-transv}
Let $X$ and $Y$ be smooth manifolds and $W$ a submanifold of the jet space $J^k(X,Y)$. 
Then, for each $k\in\N$, the set 
\[
\{f\in\CXY~|~j^kf~\overline{\pitchfork}~W\}
\] is open dense in $\CXY$ in the $\cC^{\infty}$-topology.
\end{theorem}

For any smooth map $f:\XtY$ and any $r\in\N$, we can define the set $\Sigma^r(f)$ of all points $x\in X$ at which the linear map $(df)_x$ drops rank by $r$. 
We may then associate to it the subset $\Sigma^r\subset J^1(X,Y)$ formed by all $j^1f(x)$ such that $x\in \Sigma^r(f)$~\cite{whitney1955singularities}. 
From this setup, if $\Sigma^r(f)$ is generically a manifold, we denote by $\Sigma^{r,s}(f)$ the set of all points $x\in \Sigma^r$ at which the map $\restr{f}_{\Sigma^{r}(f)}$ drops rank by some $s\in\N$. Analogously, we define $\Sigma^{r,s}\subset J^2(X,Y)$ to be the subset of all $j^2f(x)$ such that $x\in \Sigma^r(f)$. One can more generally for each sequence $r_1\geq \cdots\geq r_k$, define these strata $\Sigma^{r_1,\ldots,r_k}\subset J^k(X,Y)$ formed by all $j^kf(x)$ such that $\restr{f}_{\Sigma^{r_1,\ldots,r_{k-1}}(f)}$ drops rank by $r_k$. 

Thom introduced these strata inthe 1960s~\cite{thom}, and conjectured that for generic $f\in\CXY$, subsets $\Sigma^{r_1,\ldots,r_{k-1}}(f)$ are manifolds, and subsequent subsets $\Sigma^{r_1,\ldots,r_{k}}(f)$ are well-defined. Thom's conjecture was later proved by Boardman in the 1960s~\cite{boardman1967singularities}, and, combined with Thom's Transversality Theorem, his result stated that for every sequence of integers $r_1\geq \cdots\geq r_k$, for a generic map $f\in\CXY$, the set $\Sigma^{r_1,\ldots,r_{k}}(f)$ is a well-defined sub-manifold of $X$, and there exists a fiber subbundle $\Sigma^{r_1,\ldots,r_{k}}$ of $J^k(X,Y)$ (relative to the fibration $J^k(X\times Y)\rightarrow X\times Y$), such that
\[
x\in\Sigma^{r_1,\ldots,r_{k}}(f)\Longleftrightarrow j^kf(x)\in\Sigma^{r_1,\ldots,r_{k}}.
\] Altogether, subsets $\Sigma^{r_1,\ldots,r_{k}}$ are called the \emph{Thom-Boardman strata} (see e.g.,~\cite[Chap. VI]{golubitsky2012stable}), and can be used to vastly simplify the proofs of Theorems~\ref{thm:Morse}, and~\ref{thm:Whitney_immersion} that appear in their original papers.
Namely, by the Thom Transversality Theorem, a generic $f\in\CXY$ has to be transversal to all Thom-Boardmann strata.
Accordingly, the existence of certain ``bad'' singularities for generic maps is dictated purely by counting the dimensions of these strata (see~\cite[Theorem 6.2]{golubitsky2012stable} and~\cite[Theorem 5.6]{golubitsky2012stable}).

Theorem~\ref{thm:Thom-transv} can also be used to prove the following famous theorem of Whitney for the case where $\dim X = \dim Y = 2$. In this setting, the singularities of a smooth map $f:\XtY$ arise as the result of projecting its graph onto the target space. This produces geometrical  ``folds'', where two of them join to create a ``crease'' (see Figure~\ref{fig:forlds-cusps}). 
Informally speaking, the projection of folds creates a curve in $\R^2$, called the \emph{discriminant}, and the creases create special singular points of the discriminant, called cusps.

In an open neighborhood $U$ (in the Euclidean topology) of a point $x\in X$, we may express $f$ as $(x_1,~x_2)\longmapsto (f_1(x_1,x_2),~f_2(x_1,x_2))$. Then, the set $\cCf$ of its critical points is a curve, whose equation in $U$ is given by 
\[
\frac{\partial f_1}{\partial x_1}\cdot\frac{\partial f_2}{\partial x_2 }- \frac{\partial f_1}{\partial x_2}\cdot\frac{\partial f_2}{\partial x_1 }.  
\] 
Accordingly, ordinary cusps are isolated singular points of $\cCf$. Intuitively, folds and cusps are the only singularities that are stable under perturbations of the map $f$ (and thus its graph), while all others are expected to split into simple ones or disappear. Whitney used the term \emph{excellent} in reference to maps $\Rttt$, whose only singularities are folds and ordinary cusps.

\begin{figure}[htb]
\center
\input{cusps_nodes}
\end{figure}

\begin{theorem}[\cite{whitney1955singularities}]\label{thm:Whitney-cusps}
Let $X$ and $Y$ be manifolds of dimension two. Then, a generic map in $\CXY$ is excellent.
\end{theorem} 

Clearly, the discriminant $\cDf$ of an excellent map $f\in\CXY$ is a curve in $Y$. Furthermore, its singular points are either \emph{ordinary cusps}or \emph{simple nodes}. The local representation of an ordinary cusp is the singularity of the curve given by $x^2 + y^3$, which has Milnor number equal to two. 
A simple node, on the other hand, has Milnor number equals to one, and is locally given as a smooth intersection of two lines resulting from the discriminants of the multigerm
\begin{equation}
\left\{ \begin{aligned}  (x,y) &\longmapsto (x,~y^2)\\
 (x',y') & \longmapsto (x'^2,~y')
\end{aligned} \right.
\end{equation}

\subsection{Polynomial analogue for Whitney theorem}\label{sub:polynomial_singularities}
Whitney's results extend to map germs $(\C^2,0)\longrightarrow (\C^2,0)$, and complex polynomial maps $\Cttt$. 
Indeed, Gaffney and Mond showed in~\cite{gaffney1991cusps} that, up to generic perturbations, the discriminant of a map germ on the complex plane is a reduced curve which has only ordinary cusps and simple nodes as singularities. 
Later on, Ciliberto and Flamini~\cite{ciliberto2011branch} showed the same result for a generic projection $S\longrightarrow\P^2$ of any algebraic surface $S\subset\P^4$. 
If $S$ is the closure of the graph of a polynomial map, one obtains a description of the singularities of $\cDf$ for generic maps $f\in\PCktt$ for any degree $d$. 
This statement was later refined by Farnik, Jelonek and Ruas in~\cite{farnik2020whitney}: For any collection of positive integers $d_1\geq\cdots\geq d_n$, we use $\PdCnp$ 
to denote the space of polynomial maps $(f_1,\ldots, f_n):\Cntp$ satisfying $d_i\leq \deg f_i$ for $i=1,\ldots,p$. 
The following result is an effective analogue of Whitney's Theorem~\ref{thm:Whitney-cusps} for polynomial maps on the plane. 

\begin{theorem}[\cite{farnik2020whitney}]\label{thm:polynomial-simple-cusps1}
Let $d_1,d_2\in\N$ such that $d_1\geq d_2$, and let $D:=\gcd(d_1,d_2)$.
 Then, every generic polynomial map $f\in \PdCt$ has only ordinary cusps and simple nodes as singularities.
  Furthermore, if $\kappa(f)$ and $\nu(f)$ denote their respective numbers, then
\begin{equation}\label{eq:cusps_nodes}
\begin{array}{lll}
\kappa(f) & = & d_1^2 + d_2^2 + 3~d_1d_2 - 6~d_1 - 6~d_2 + 7\\
2~\nu(f) & = & (d_1d_2 - 4)\big((d_1 + d_2 - 2)^2 - 2\big) - (D - 5)(d_1 + d_2 - 2) - 6.
\end{array}
\end{equation}
\end{theorem}

\begin{proof}[Sketch of proof]
The proof analyses the Thom–Boardman strata in the jet space $J^d(2,2)$ ($d:=\max(d_1,d_2)$). The stratum of largest dimension corresponds to the discriminant $\cDf$, whereas the smaller ones describe the corresponding singularities such as cusps and nodes. 
One approach is to show an analogue of the Thom Transversality Theorem on $J^d(2,2)$ (~\cite[Theorem 2.3]{FJR19}). 
Consequently, a generic map in $\PdCt$ has only ordinary cusps and simple nodes as singularities. Computing $\kappa(f)$ amounts to counting the number of intersection points of $\Ima j^k$ with the largest stratum in the jet space of complementary dimension. This is just an application of B\'ezout Theorem if there is no solutions at infinity. Thanks to $f$ being generic, the projectivization of $\cDf$ in $\P^2$ is smooth at the line at infinity~\cite[Theorem 4.5]{farnik2020whitney}. 
The number $\nu(f)$ can then be deduced from the genus and number of cusps using the Riemann-Roch formula.
\end{proof}

Slightly reformulated, Theorem~\ref{thm:polynomial-simple-cusps1} implies  that the subset of polynomial maps in\\ 
$\PdCtt$, whose cusps and nodes are given as in Equation~\eqref{eq:cusps_nodes}, contains a Zariski open subset. One thus might ask whether this set is itself a Zariski open. Note that this becomes false for $\PdCtp$~\cite{FJM18} for some $p\geq 3$. Very recently, Farnik and Jelonek showed that this is indeed the case for maps $\Cttt$.

\begin{theorem}[{{\cite[Theorem 1.1]{farnik2022generic}}}]\label{thm:generic_top-type_planar}
For every $d_1,d_2\in\N$, the set of polynomial maps $f\in\PdCt$, whose numbers $\kappa(f)$ and $\nu(f)$ satisfy~\eqref{eq:cusps_nodes} form a Zariski open subset $\cZ$. 
Furthermore, for any $f,g\in\PdCt$, if $f\in\cZ$, then $f$ is topologically equivalent to $g$ if and only if $g\in\cZ$.
\end{theorem}
\begin{proof}[Sketch of proof]
For any map $f\in\PdCtt$, the values  $\nu(f)$ and $\kappa(f)$ appearing in~\eqref{eq:cusps_nodes} are upper bounds for the number of its singularities of nodal and cuspidal type, respectively. Furthermore, if these bounds are reached, then singularities of respective types are simple. This follows from Rouch\'e's Theorem~\cite{chirka1997complex}. 
Now, consider the Zariski open $U\subset\PdCtt$ consisting of all maps $f$ whose critical locus $\cC_f$ is irreducible, has maximal degrees, and such that the bounds in~\eqref{eq:cusps_nodes} is reached. If $\Phi:U\times\C^2\longrightarrow U\times\C^2$ is the map $(f,~x)\longmapsto (f,~f(x))$, and $S:=\{Y_1,Y_2,Y_3\}$ is a decomposition of $Y$ given by 
\[
\{Y_1,Y_2,Y_3\} = \{Y\setminus \cD_{\Phi},~\cD_{\Phi}\setminus\Sing(\cD_{\Phi}),~\Sing(\cD_{\Phi})\},
\] then $Y$ is a Whitney stratification. The other non-obvious fact to show is that the projection $\pi:~U\times\C^2\longrightarrow U$, $(f,~f(x))\longmapsto f$ is fibration. The proof uses the newly-established effective method for computing a superset of the bifurcation set of polynomial dominant maps~\cite{dhinh2021thom}. For $d_1=d_2$, the proof follows by applying Thom's Second Isotopy Lemma on an extension of $\pi$ to the projective closure of its domain. 
The case $d_1\neq d_2$ follows similar steps.
\end{proof}

The theorem above has two consequences. First, to determine whether a polynomial map $\Cttt$ has a generic toplogical type, it is enough to compute its singular points. The second consequence is that such polynomial maps satisfying~\eqref{eq:cusps_nodes} are topologically stable after perturbations. Stable singularities are described as well for polynomial maps $\C^3\longrightarrow\C^3$; each such singularity is one of the below three types $\cA_3$, $\cA_2\cA_1$, and $\cA_1^3$ depicted in Figure~\ref{fig:sing-3d}.\\ 
 
{\scriptsize
 
\begin{tabular}[htb]{ccc}
$\cA_3$ & $\cA_2\cA_1$ & $\cA_1^3$ \\
$
(z,y,z)\mapsto (x,y,z^4 + y^2z + xz)
$
&
$
\begin{cases}
	(x_1,y_1,z_1) & \mapsto (x_1,y_1,x_1^3 + y_1x_1)\\
	(x_2,y_2,z_2) & \mapsto (x_2^2,y_2,z_2)
\end{cases}
$
&
$
\begin{cases}
	(x_1,y_1,z_1) & \mapsto (x_1,y_1,z_1^2)\\
	(x_2,y_2,z_2) & \mapsto (x_2,y_2^2,z_2)\\
	(x_3,y_3,z_3) & \mapsto (x_2^2,y_2,z_2)
\end{cases}
$
\end{tabular}
}

\begin{figure}[htb]
\center
\input{sing_3d}
\end{figure}
 Using similar tools as for Theorem~\ref{thm:polynomial-simple-cusps1}, the authors provide the following generalization.
\begin{theorem}[\cite{farnik2021finite}]\label{thm:polynomial-stable-singularities-dim3}
Let $d_1,d_2,d_3\in\N$ such that $\gcd(d_i,d_j)\leq 2$ for $1\leq i\leq j \leq 3$ and $\gcd(d_1,d_2,d_3) = 1$. Then, every generic polynomial map $f\in \PdCth$ has only singularities of types $\cA_3$, $\cA_2\cA_1$, and $\cA_1^3$. Furthermore, their respective numbers $\#\cA_1^3(f)$, $\#\cA_2\cA_1(f)$, and $\#\cA_3(f)$ are computed in terms of $d_1$, $d_2$ and $d_3$ (see~\cite[Theorem 3.1]{farnik2021finite}).
\end{theorem}

\subsubsection{Counting singularities using polyhedra}\label{sss:polyhedral-singularities}  In this part, we present an extension of the problems solved above to some families in $\PdCt$. 

We have the following combinatorial description on the numbers $\kappa(f)$ and $\nu(f)$ of singularities of a generic map $f\in\PdCt$ in terms of the Newton polytope. It is shown in~\cite{farnik2020whitney} that the set of critical points $\cCf$ form an algebraic curve of degree $d_1+d_2-2$, and its discriminant $\cDf$ is a curve of degree $d_1(d_1+d_2 - 2)$. Then, a simple computation shows that the Newton polytopes $\Sigma(f):=\NP(\cCf)$ and $\Delta(f):=\NP(\cDf)$ are triangles satisfying the following
\begin{align*}
\Sigma(f) = & \conv\big\{(0,0),~(d_1+d_2 - 2,~0),~(0,~d_1+d_2 - 2) \big\}, \text{ and}\\ 
\Delta(f) = & \conv\big\{(0,0),~(d_1~(d_1+d_2 - 2),~0),~(0,~d_2~(d_1+d_2 - 2)~) \big\}.
\end{align*} Then, if $D:=\gcd(d_1,d_2)$, we obtain 
\begin{align}\label{eq:interior-degree}
\hcir\Delta(f) - \hcir\Sigma(f) =~&  \big(d_1 +~d_2 - 2\big) \big(d_1^2d_2 + d_1d_2^2 - 2~d_1~d_2 - 2~d_1 - 2~d_2 - D + 5\big)/2,
\end{align}, using Pick's formula. 
Here, we have $\hcir\Pi$ denote the number of integer points of the relative interior of a polytope $\Pi\subset\R^n$,
\[
\hcir\Pi:= \# (\Pi\setminus\partial \Pi)\cap\Z^n.
\] Next, the right hand side of~\eqref{eq:interior-degree} is equal to $\kappa(f) + \nu(f)$ as given in Theorem~\ref{thm:polynomial-simple-cusps1}. 
Hence, we obtain the following equality for generic maps $f\in\PdCt$
\begin{align*}
\hcir\Delta(f) - \hcir\Sigma(f) =~&  \kappa(f) + \nu(f).
\end{align*} It turns out that the above formula can be generalized as follows. 

\begin{theorem}[{{\cite[Theorem 2.19]{hilany2024polyhedral}}}]\label{thm:poly-type-discriminant}
There exists a large family $\mathscr{C}\subset\cInk$ such that for each $\bmA\in\mathscr{C}$, a generic map $f\in\CA$ has only simple nodes and ordinary cusps as singularities outside $\unze$, and their total number equals $\hcir\Delta(f) - \hcir\Sigma(f)$.
\end{theorem}

\begin{proof}[Sketch of proof]
The proof follows closely that for Theorem~\ref{thm:polynomial-simple-cusps1}; we show that the non-degeneracy of the pair $\bmA$ implies that the map $F_{\bmA}:(\C^*)^2\times\CA\longrightarrow J^{d}(\C^2)$ is transversal to the Thom-Boardman strata. The compactification of $\cCf$ and $\cDf$ are taken in the toric varieties $X_\Sigma$ and $X_\Delta$ of the respective Newton polytopes $\NP(\cCf)$ and $\NP(\cDf)$. 
We then show that $\cCf$ is smooth, irreducible, and Newton non-degenerate for generic $f\in\CA$. 
Observing that the projection $(z,~w)\longmapsto w$ extends to a toric morphism $X_\Sigma\times X_\Delta\longrightarrow X_\Delta$, the singularities of $\overline{\cDf}$ are described from $ \overline{\cCf}$. Then, the number of singular points of $\cDf$ can be deduced thanks to Khovanskii's classical formulae for betti numbers of toric subvarieties~\cite{khovanskii1978newton}. 
\end{proof}

Recall from~\S\ref{sub:prel_newton} the definition of $\cInk$ for the $p$-tuples of lattice polytopes in $\R^n_{\geq 0}$.
A generalization of Theorem~\ref{thm:poly-type-discriminant} is still open. 

\begin{problem}\label{prb:singularities_polynomial_general}
Given two integers $n,p\in\N$, let $\bmA\in\cInk$. Describe the singularities of generic $f\in\CA$ in the vein of Theorem~\ref{thm:poly-type-discriminant}.
\end{problem}
The first step towards the above problem is to verify that Theorem~\ref{thm:polynomial-stable-singularities-dim3} admits a re-formulation in terms of Newton polytopes in the sense of Theorem~\ref{thm:poly-type-discriminant}.

\begin{conjecture}\label{con:singularities-lattice-generalized}
There are integers $m_1$, $m_2$, and $m_3$ that depend only on the respective singularity types $\cA_3$, $\cA_2\cA_1$, and $\cA_1^3$,, that satisfy the following. For any triple $\bmA\in \cIthth$ of integer polytpes, a singularity of a generic map $f\in\CA$ can either be of type $\cA_3$, $\cA_2\cA_1$ or $\cA_1^3$. 
Furthermore, if $\delta_0(f)$ measures the Milnor number of the singularity of $f$ at the origin, then we get
\begin{align*}
\hcir\Delta(f)~ -~\hcir\Sigma(f) =~& m_1~\#\cA_3(f)~+~m_2~\#\cA_2\cA_1(f)~+~ m_3~\#\cA_1^3(f) +\delta_0(f).
\end{align*}

\end{conjecture}

\subsection{The Jacobian Conjecture}\label{sub:Jacobian_conjecture} 
The determinant of the Jacobian matrix of a polynomial map $f:\Cntn$ is a complex polynomial in $\C[x_1,\ldots,x_n]$. We denote it by $f'(x):=\det \Jac_x f$, and say that $f$ satisfies the \emph{Jacobian hypothesis} if  $f'(x)\neq 0$ for each $x\in \C^n$. Clearly, if $f$ is invertible, then the Jacobian hypothesis is satisfied for $f$. Keller formulated the following conjecture in the 1930s~\cite{keller1939ganze}.

\begin{conjecture}[Jacobian Conjecture]\label{con:Jacobian}
Let $f:\Cntn$ be a polynomial map such that $f'(x)=1$ for all $x\in\C^n$. Then, the map $f$ is invertible.
\end{conjecture}
Since linear functions are the only univariate polynomials without critical points, the Jacobian Conjecture is true for $n=1$. For higher values of $n$, however, it has proved very difficult to prove, and there is no consensus as to whether it is true or not. Several incorrect proofs were published in the 1950s and 1960s~\cite[\S I.3]{BCW82}. This put work on the Jacobian Conjecture into a lull until some interest was revived later in the 1980s. In particular, a number of important results were obtained which re-motivated the theory and initiated further open questions on the topology of polynomial maps~\cite{BCW82,van2012polynomial}. 

Pertinent to this theory is the following classical result of Bi\l ynicki-Birula and Rosenlicht.
\begin{theorem}[\cite{bailynicki1962injective}]\label{thm:Blyn-Bir-Rosen}
Let $\K$ be an algebraically closed field of characteristic zero. Let $F:\KKntn$ be a polynomial map. If $f$ is injective, then $f$ is surjective and the inverse is a polynomial map, i.e. $f$ is a polynomial automorphism.
\end{theorem} Wang showed in~\cite{wang1980jacobian} that the Jacobian hypothesis guarantees the injectivity of $f$ whenever $\deg f = 2$. Then, it follows from Theorem~\ref{thm:Blyn-Bir-Rosen} that the Jacobian Conjecture holds true for all quadratic polynomial maps $\Cntn$. 

Instead of fixing the degree, one may we consider polynomial maps with prescribed mixed volumes of their Newton polytopes.
This gives an interesting family of polynomial maps satisfying the Jacobian Conjecture. 

\begin{theorem}\label{affber}
If $\bmA\in \cInn,$ and $f\in\CA$, then the number of isolated solutions in $\C^n$ of the system $f=\unze,$ counted with multiplicites, cannot exceed the mixed volume $V(\bmA).$
\end{theorem}

Let $\scIznn\subset\cInn$ denote the subset of all tuples $\bmA:=(A_1,\ldots,A_n)$ such that for each $i\in[n]$, we have $\{\tupze\}\subset A_i\subset\R^n_{\geq 0}$. 
We obtain the following result. 

\begin{corollary}[{{\cite[Theorem 1.2]{EH-dAndrea_affine}}}]\label{cor:MV=1}
Let $n\in\N$ and $\bmA\in\scIznn$ such that $V(\bmA)=1$. Then, the Jacobian Conjecture holds true for every $f\in\CA$.
\end{corollary}

\begin{proof}
Let $f\in\CA$ be any map satisfying the Jacobian hypothesis. Since the Jacobian does not depend on $y$ for any $y\in\C^n$, the tuple $f-y$ also satisfies the Jacobian Hypothesis. 
Then, as the Jacobian is non-zero at each such solution, all solutions to $f-y=\unze$ in $\C^n$ are isolated. 
Hence, Theorem~\ref{affber} implies  that there is at most one isolated solution in $\C^n$, which must have multiplicity $1$. 
Consequently, we get that the map $f:\C^n\to\C^n$ is injective, as otherwise there would be $y_0\in\C^n$ such that $f-y_0=\unze$ has more than one isolated solution. 
 From Theorem \ref{thm:Blyn-Bir-Rosen} we deduce the claim.
\end{proof}

Corollary~\ref{cor:MV=1} may be used to classify larger families of polynomial maps satisfying the Jacobian Conjecture;
one approach is to consider the group $\text{PAut}(\C,n)$ of polynomial automorphisms acting on the space $\PCnn$ of polynomial maps $\Cntn$. 
\textit{How large are the orbits in $\PCnn$ corresponding to the polynomial maps whose mixed volume equals one?} According to Corollary~\ref{cor:MV=1}, such a class of orbits contains maps satisfying the Jacobian conjecture, but might not have mixed volume one as the following example shows.
\begin{example}\label{ex:jacobian_mv1}
The map
\begin{align}\label{eq:mv-map-example_main}
(x,~y)& \longmapsto (1-x-y-x^2,~1 +2x +y +x^2), 
\end{align} has Newton tuple with mixed volume equals to two. 
It is invertible since
\begin{align*}
f^{-1}(a,b) =~ & \left(\frac{a+b}{2},~1-\frac{a+b}{2} - \frac{(a+b)^2}{4} - a \right).
\end{align*} However, the automorphism $\Pol_{\C}(n,n)\longrightarrow\Pol_{\C}(n,n)$, $(f_1,~f_2)\longmapsto (f_1,~f_1+f_2)$ sends~\eqref{eq:mv-map-example_main} to a map whose Newton pair has mixed volume equals to one.
\end{example}
We will not elaborate on further results related to the Jacobian Conjecture, as these can be found in several excellent surveys such as~\cite{BCW82} and~\cite{van2012polynomial}. 
We finish this part by mentioning only a handful of results that are known for the case $n=2$. 
They are summarized in the following theorem.

\begin{theorem}\label{thm:Jac-con_true}
Let $f:\Cttt$ be a polynomial map such that $f'(x)\neq 0$ for all $x\in\C^2$. 
Then, invertibility of $f$ holds true in each of the following cases:

\begin{enumerate}

	\item $\deg f\leq 100$,~\cite{moh1983jacobian}
	
	\item $\gcd(\deg f_1,~\deg f_2)=1$,~\cite{nagata1989some}	and
	
	\item for every line $L\subset\C^2$, the map $\fr_L:L\longrightarrow f(L)$ is injective~\cite{abhyankar1975embeddings,gwozdziewicz1993injectivity}.
	
\end{enumerate}

\end{theorem}

\subsubsection{Jacobian Conjecture for real maps}\label{sub:discr-real}
The \emph{Weak Real Jacobian Conjecture} is a mutatis mutandis analogue of the Jacobian Conjecture~\ref{con:Jacobian}, except that we replace $\C$ with $\R$. 
The two conjectures are related by splitting a complex polynomial map $f:\Cntn$ into its real and imaginary parts to obtain a real polynomial map $\R f:=(\Re f_1,~\Im f_1, \ldots, \Re f_n,~\Im f_n):\R^{2n}\longrightarrow\R^{2n}$. From~\cite{yu1995jacobian}, we have\\
 $f'(x) = (\det \Jac_{(z,w)} \R f)^2$ for all $x=z+i~w\in\C^n$. 
 Hence, the Jacobian Conjecture is true if and only if the Weak Real Jacobian Conjecture is true for all $n\geq 2$. 

Even though there were some progress by Yu in the 1990s~\cite{yu1995jacobian}, the aforementioned weak version remains open.
So is another, topological, version of the Real Jacobian Conjecture is open and implies the Complex Jacobian Conjecture. 
The latter will be discussed in~\S\ref{sss:non-prop_Jacobian} below.
As for now, we introduce yet another variant of the real Jacobian conjecture:
\begin{conjecture}[Strong Real Jacobian Conjecture]\label{con:strong_real_Jacobian}
Let $f:\Rntn$ be a polynomial map such that $f'(x)>0$ for all $x\in\R^n$. Then, the map $f$ is invertible.
\end{conjecture}
In the 1990s~\cite{pinchuk1994counterexamle} Pinchuk has disproven Conjecture~\ref{con:strong_real_Jacobian} by constructing a non-injective polynomial map $(f_1,f_2):\RtR$ satisfying $f'(x)>0$ for all $x\in \R^2$.
This counterexample generated considerable interest in the classification of polynomial maps $\Rttt$ with everywhere positive Jacobian determinant. 
Such maps are called \emph{Pinchuk maps}.
The original Pinchuk map is of degree $40$, and has $\deg f_1 = 10$.
Accordingly, the next task was to find a Pinchuk map $f$, assuming $\deg f_1\leq \deg f_2$, with the smallest possible value for  $\deg f_1$. 

The first lower bound was given by Braun and dos Santos Filho in 2010; they showed in~\cite{braun2010real} that no Pinchuk maps exist whenever $\deg f_1 \leq 3$. 
This was following a result of Gwo\'zdziewicz in~\cite{gwozdziewicz2001real}, stating that the strong Jacobian Conjecture holds true for real polynomial maps on the plane of degree three and less.
Gwo\'zdziewicz's work was also extended by  Braun and Oréfice-Okamoto to maps of degree four~\cite{braun2016polynomial}.
Following Cambell's~\cite{campbell2011asymptotic}, and later, Fernandes'~\cite{fernandes2022new} discovering of Pinchuk maps with smaller values for $\deg q$ ($25$ and $15$ respectively), the construction with the smallest pairs of degrees is a result of Braun and Fernandes~\cite{braun2023very}, with $(\deg f_1,\deg f_2) = (9,15)$. 
The question remains open whether there exists a Pinchuk map with $\deg f_1 \in\{5,6,7,8\}$.

None of the Pinchuk maps above are surjective.
This raises the question on whether non-surjectivity is necessary for non-injectivity.
This was recently addressed by Fernandes and Jelonek in~\cite{fernandes2024pinchuk} by discovering two Pinchuk maps, in which one is surjective, while the other one has non-dense image.

The tools for constructing the original Pinchuk map relied on understanding the set of all Newton pairs $\bmA\in\cItt$ for which there is pair of polynomials $(f_1,f_2)\in\RA$ such that $f'(x)>0$ for all $x\in\R^2$.
Under the latter condition, we say that $f_1$ and $f_2$ are \emph{real Jacobian mates}. 
Even though this term was coined by Gwo\'zdziewicz in 2016,
a partial classification of Newton pairs for real Jacobian mates was initiated by Abhyankar already in the 1970s~\cite[Theorem 19.4]{abhyankar1977lectures}. 
Another partial classification would appear much later in~\cite{Gwozdziewicz2016}.

\begin{theorem}[{{\cite[Theorem 1]{Gwozdziewicz2016}}}]\label{thm:Gwoz}
Assume that the Newton polygon of a polynomial $f_1\in\R[x_1,x_2]$ has an edge that: begins at $(0, 1)$, has a positive inclination, and has no lattice points in its relative interior. 
Then, there does not exist a polynomial $f_2\in\R[x_1,x_2]$ for which $f_1$ and $f_2$ are Jacobian mates.
\end{theorem}
The term \emph{positive inclination} for an edge is used whenever it has an outer normal vector $v:=(v_1,v_2)$ such that $v_1 > 0$ and $v_2 < 0$. 

Analyzing Newton pairs for counting the number of singularities at infinity of the polynomial $f_1$ was the method adopted in~\cite{van1994note} for constructing the above-mentioned Pinchuk map with $(\deg f_1,\deg f_2) = (9,15)$~\cite{braun2023very}.
This signifies the importance of the following more general problem.

\begin{problem}\label{prb:Gwo-n}
Give necessary conditions for all collections $\bmA\in\cInk$, such that for each $f\in\RA$, there exists $x\in\R^n$ satisfying $f'(x)=0$.
\end{problem}

\section{Singularities at infinity of generically-finite maps}\label{sec:sing_infty}
In analogy to affine singularities, the complexity of analytic maps can also be understood via the behavior of their fibers ``at infinity''. 
In this section, we elaborate on this phenomenon for a class of polynomial maps, whose generic fibers are finite.

\begin{definition}\label{def:finite-polynomial}
Let $f:\XtY$ be a dominant polynomial map between irreducible normal affine varieties over an algebraically closed field $K$. 
Then $f$ is said to be \emph{finite} if $K[X]$ is integral over the subring $f^*K[Y]$ of $K[X]$.
That is, if for each polynomial function $h$ in the coordinate ring $f^* K[Y]$, there exists $m\in\N$ and a collection $\{g_1,\ldots,g_m\}\subset  K[X]$ such that the equation $h^m + g_1~h^{m-1} + \cdots +g_m$ is the zero-function~\cite[Ch. I, \S 3]{shafarevich1994basic}.
\end{definition} 
If $f$ is a finite map, then every point $y\in Y$ has only finitely-many inverse images~\cite[pp. 48, Ch. I, \S 3]{shafarevich1994basic}.
Hence, a necessary condition for finiteness is that $\dim X\leq \dim Y$. 
In topological terms, a finite analytic map is a locally analytic branched covering over each point in $Y$ (see e.g.,~\cite[Theorem 2]{gunning1973lectures} or~\cite{milnor2016singular}). 
Then, for analytic maps satisfying $\dim X\leq \dim Y$, we may describe singularities at infinity as those points in $Y$ over which the map is not a local analytic cover. 
As $K[X]$ and $K[Y]$ are integrally closed, given any finite polynomial map, there is a Zariski open $U\subset Y$, such that the restricted mapping
\begin{align}\label{eq:generically}
\fr_{f^{-1}(U)}:= & f^{-1}(U)\longrightarrow U
\end{align} is an unramified covering of degree equal to the index $[K(X):f^*K(Y)]$~\cite[Ch. II, \S 5]{shafarevich1994basic}. This is called the \emph{topological degree of $f$}. Then, the discriminant $\cDf$ is the union of local ramification loci of $f$. 
Accordingly, we refer to $f$ as \emph{generically finite} so to distinguish the case if $f$ is not finite. 
For example, even though we have the equality $\dim \bl^{-1}(\unze)=1$, the blowup map $\bl:(x,~y)\longmapsto(x,~xy)$ is generically finite by taking $U$ to be $\C^2\setminus\{\text{vertical line}\}$.

Let $\ccSf$ denote the set of all $y\in Y$ at which for every neighborhood $V$ containing $y$, the restricted map $\fr_{V}$ is not finite. Jelonek initiated in the 1990s the study of topological characterization of finite polynomial maps over the field of complex numbers~\cite{Jel93,Jel99}, where he later showed that $\ccSf$ is an algebraic hypersurface in $Y$~\cite{jelonek2001topological}. Whenever $K=\C$ or $K=\R$, the set $\ccSf$ coincides with the set of points $y\in Y$ for which there exists an infinite sequence of points $\{x_k\}_{k\in\N}\subset X$ satisfying $\Vert x_k\Vert\longrightarrow \infty$ and $f(x_k)\longrightarrow y$.

As the definition indicates, the non-properness set is an important object associated to $f$ that captures the asymptotic behavior of $f$. 
Its study was first initiated by Jelonek in his two 1990s papers~\cite{Jel93,Jel99} on complex polynomial maps $\Cntn$; he proved structural results describing its topological and algebraic properties, as well as computational methods. 
He also showcased the importance of the non-properness set in solving problems from affine geometry. 

The goal of this section is to present computational and structural results on the non-properness set, as well as generalizations and applications;~\S\ref{sub:non-prop_computational} is devoted to the various methods  for computing the non-properness set of complex and real polynomial maps. 
Then, in~\S\ref{sub:non-prop_struc}, we describe some of the  interesting topological and geometrical properties. 
Finally, in~\S\ref{sub:non-prop_applications}, we highlight some of the many surprising appearances of the non-properness set in different problems from affine geometry, which has lead to useful applications.

\subsection{Computational results}\label{sub:non-prop_computational} Given a polynomial map $f:\Cntn$, consider the maps $F_i:~\C^n\longrightarrow\C^n\times\C$, such that $F_i(x):=(f(x),x_i)$, for $i \in  [n]$. Also let
\[ 
\sum\nolimits_{k=0}^{N_i}A_k^i(f) \, x_i^{N_i -k}~\in \C[f_1,\ldots,f_n,x_i], 
\]
be the defining polynomial for the hypersurface $F_i(\C^n)\subset \C^n\times\C$.
Then, $A_0^i(f)$ is the \emph{$i$-th non-properness polynomial of $f$}. 

\begin{theorem}[{{\cite[Proposition 7]{Jel93}}}]\label{th:Jeln}
 The non-properness set, $\ccSf$, of a dominant polynomial map $f:~\Cntn$,  is $\{\prod_{i=1}^nA_0^i(y)=0\}$, where each $A_0^i$ is the $i$-th non-properness polynomial of $f$.
\end{theorem} 

Results in Theorem~\ref{th:Jeln} were generalized by A. Valette~\cite{Sta02} to polynomial maps $f:\XtY$ between arbitrary affine varieties over an algebraically closed field~\cite{Sta02}.
Despite its difficulty, some computational results for the non-properness set of real maps were presented, albeit only restricted to some special cases~\cite{Jel01c,Sta07}. Recently, a full method has been developed by the author and Tsigaridas for planar maps. 

\begin{theorem}[\cite{EHT21}]\label{thm:ET-real-n=2}
Let $f:\RtR$ be a dominant polynomial map. Then, a decomposition $\cS_f=S_1\cup\cdots\cup S_r$ of the non-properness set into $\R$-uniruled semi-algebraic curves can be computed effectively.
\end{theorem} 

The workings of the method behind Theorem~\ref{thm:ET-real-n=2} relies on toric varieties. 

\begin{proof}[Sketch of proof of Theorem~\ref{thm:ET-real-n=2}]
Let $\bmA\in\cItt$ be a pair of polytopes, let $f\in\CA$, and let $X_{\bmA}$ denote the real toric variety associated to ${\bmA}$~\cite{sottile334toric}. A real solution to the polynomial system $f - w=0$ in the toric compactification of $(\R^*)^2\hookrightarrow X_{\bmA}$ is either at the $(\R^*)^2$-orbit or somewhere at the one-dimensional orbits $\cO_{\bGam}$ of its boundary $\partial X_{\bmA}$. The latter is given by the pair dicritical faces $\bGam\prec \bmA$. Since a solution $z\in(\R^*)^2$ to $f-w=0$ in $(\R^*)^2$ goes to infinity in the vicinity of generic points $w\in\cS_f$, the point $z$ converges to an orbit $\cO_{\bGam}\subset X_{\bmA}$. 
Then, either $w$ gives rise to a new solution at $\cO_{\bGam}$ or it increases the multiplicity of an already existing one.
This way, we get a correspondence between the set of coherent faces $\bGam\prec \bmA$ and the components of $\cS_{\C f}$, where $\C f:\CtC$ is the canonical complexification of $f$. Furthermore, those components can be computed effectively by measuring the multiplicities of solutions at the orbits. In order to distinguish points in the real non-properness set $\cS_f\subset\cS_{\C f}\cap \R^2$, it is enough to test the preimages of points near the intersection $\cS_{\C f}\cap \cD_f$.
\end{proof}

Theorem~\ref{thm:ET-real-n=2} leaves out the following open problem.

\begin{problem}\label{prb:non-prop_computing}
Design a method that computes the non-properness set of any polynomial map $\RtRn$.
\end{problem}

\subsubsection{Non-properness from face-resultants}\label{sss:face-resultants}
More recently in~\cite{EH2022tropinon}, the author generalized the proof of Theorem~\ref{thm:ET-real-n=2} to generic higher-dimensional polynomial maps higher dimensional maps over an algebraically closed field $\K$ of characteristic zero. Notably, the relation between coherent faces of the Newton tuple $\bmA$, and the boundary solutions in the toric compactification $X_{\bmA}$. This is represented in the following result.

A coherent face $\bmG$ of $\bmA$ is called \emph{dicritical} if an  outward normal vector $ (\alpha_1,\ldots,\alpha_n)\in\R^n$ of one of its supporting hyperplanes satisfies $\alpha_i>0$ and $\alpha_j\leq 0$ for some $i,j\in\{1,\ldots,n\}$. 
For a proper coherent face $\bmG$ of $\bmA$, we use $X_{\bmG}$ to denote the set  
\[
\left\{(z,w)\in ( \K^*)^n\times  \K^n~\left|~(f(z)-w)_{\bmG} =0\right\}\right. .
\] The \emph{face-resultant} of the map $f$ at the face ${\bmG}$, denoted by $\cR_{\bmG}(f)$, is the Zariski closure in $\K^n$ of the image of  $X_{\bmG}$ under the projection $\pi:~(z,w)\longmapsto w$.

\begin{theorem}[{{\cite[Theorem 9]{EH2022tropinon}}}]\label{thm:algebraic_correspondence}
Let $\bmA:=(A_1,\ldots,A_n)\in\cInn$ be a tuple of $n$ polytopes in $\R^n_{\geq 0}$ such that $\bm{0}\in A_i$ ($i=1,\ldots,n$). 
Then, a generic map $f\in \K^{\bmA}$ satisfies
\begin{equation}\label{eq:non-properness=crsg}
\cS(f)= \bigcup \cR_{\bmG}(f),
\end{equation} where the union runs over all dicritical coherent faces of $\bmA$.
\end{theorem}
The proof of Theorem~\ref{thm:algebraic_correspondence} uses a refined analysis on methods from multivariate resultants established in~\cite{GKZ08}. An immediate consequence of this result is an effective method for computing the non-properness set of generic maps in $\K^{\bmA}$. Even though the earlier-presented known methods (Theorem~\ref{th:Jeln}) achieve this purpose as well, the sets $\cR_{\bmG}(f)$ usually require significantly fewer input data from $f$.

\subsection{Structural results}\label{sub:non-prop_struc} 

Numerous questions on the topology of the non-properness set remain open, however, several of its characteristics have been uncovered. 
In what follows, we present some of these results.

\subsubsection{Uniruledness}\label{sss:uniruled} An affine variety $X$ over any field $\K$ is said to be $\K$-\emph{uniruled} if for every $x\in X$, there exists a polynomial map $\varphi:\K\longrightarrow X$ satisfying $\varphi(0)= x$.

\begin{theorem}[\cite{Jel93,jelonek2001topological,jelonek2020quantitative}]\label{thm:non-prop-comp}
Assume that $\K$ is algebraically closed, and let $X$ and $Y$ be two affine varieties over $\K$. Then, the non-properness set $\ccSf$ of a generically-finite polynomial map $f:=(f_1,\ldots,f_p):\XtY$ is an algebraic hypersurface in $Y$ of degree at most 
\[
\frac{\deg X\cdot \prod_i \deg f_i - \mu(f) }{\min_i\deg f_i},
\] where $\mu(f)$ is the topological degree of $f$. Furthermore, if $X$ is $\K$-uniruled (e.g. $X=\K^n$), then $\ccSf$ is also $\K$-uniruled.
\end{theorem} 

The idea behind the proof of Theorem~\ref{thm:non-prop-comp} is taking the closure $\overline{G}_{\K}$ in $\P_\K^n\times \K^p$ of the graph $G_{\K}\subset \K^n\times \K^p$ of $f$. If $R:=\overline{G}_{\K}\setminus {G}_{\K}$, then $\cS_f=\pi(R)$, where $\pi:\P_\K^n\times \K^p\longrightarrow \K^p$ is the projection forgetting the first $n$ coordinates. As $R$ is algebraic, then so is $\cS_f$. Since $\dim R<n$, we get $\dim \cS_f<n$. The statement on uniruledness follows by further analyzing the projection $\pi$, whereas the degree is deduced from a classical theorem of O. Perron in the 1930s for finding the degree of the algebraic dependence in Definition~\ref{def:finite-polynomial}~\cite[Theorem 1.5]{ploski1986algebraic}.

Jelonek was also the first to consider the case for real maps.

\begin{theorem}[\cite{Jel02}]\label{thm:real-maps-non-prop}
The non-properness set $\cS_f$ of a real generically-finite polynomial map $f:\Rntp$ is either empty or a closed $\R$-uniruled semi-algebraic subset. Furthermore, we have $1\leq \dim \cS_f \leq n-1$.
\end{theorem} 
\begin{proof}[Sketch of proof]
From the construction of the above set $R$, we deduce the statement about semi-algebraicity and dimension. 
As for the uniruledness, we only present it for $n=2$: 
Extend $f$ to a rational map $\tilde{f}:\P_\R^2\longrightarrow\P^p_\R$, and resolve the points of indeterminacy of $\tilde{f}$ by a sequence of blowups. 
In this way, we obtain a smooth surface $Z$ and a regular map $F:Z\longrightarrow\P^p_\R$ satisfying $\Fr_{\R^2}=f$. If $R$ denotes the set $Z\setminus\R^2$ and $L_\infty$ denotes the hyperplane at infinity in $\P_{\R}^p$, then $\cS_f$ coincides with $R\setminus F^{-1}(L_\infty)$. Consequently, the set $R$ is a union of copies of the projective line $\P_\R$, and the restriction $\Fr_{\P_\R}$ to a copy of each circle gives a polynomial parametrization of $\Ima \Fr_{\P_\R}$~\cite[Proof of Theorem 4.2]{Jel02}.
\end{proof}
The non-properness set is not necessarily algebraic, neither does it have to be a hypersurface. For example, we get $\cS_f $ is the half- line $ \{u=0,~v=0,~w\geq 0\}\subset\R^3$ whenever $f$ is expressed as
\[
(x,~y,~z)\longmapsto (x,~y,~((x^2+y^2)~z^2 + z)^2).
\] In fact, not many restrictions are imposed on an $\R$-uniruled semi-algebraic subsets in $\R^n$ for it to be the non-properness sets of a polynomial map~\cite[Theorem 7.1]{Jel02}, as the following results demonstrate.

\subsubsection{Universality}\label{sss:universality}
The next result is an attempt to answer the following question: \textit{What conditions are sufficient for an algebraic curve $S\subset\R^2$ to be the non-properness set of a planar polynomial map? }

\begin{theorem}[\cite{Jel02}]\label{thm:universality-non-properness}
Let $\K$ be the field $\C$ or $\R$, and let $S_1,\ldots,S_r\subset\K^p$ be any collection of (semi-)algebraic sets, for which there are finite and surjective polynomial maps $\varphi_i:\K^{k_i}\longrightarrow S_i$ ($n>k_i\geq 1$). Then, there exists a generically-finite polynomial map $f:\KKntp$ such that $\ccSf = S_1\cup\cdots\cup S_r$. In particular, every  $\K$-uniruled (semi)-algebraic curve $S$ in $\K^2$, is the non-properness set of a finite polynomial map $\Rttt$. 

\end{theorem} 

\begin{proof}[Sketch of proof]
We only present it for $n=2$. To construct the map $f$ from $S$, we decompose the latter into a collection of irreducible components $S_1,\ldots,S_r\subset\K^2$. Then, for each $i=1,\ldots,r$, the curve $S_i$ is parametrized by a polynomial map $\varphi_i:\K\longrightarrow S_i$. Let $F:~\K^2_{x,y}\longrightarrow\K^2_{u,v}$ be a map given by
\[
(x,y)\longmapsto \left(x,~\prod_{i=1}^r(x - i)~y^2+y\right).
\] It is easy to check that $F$ has finite fibers and satisfies $\cS_F=L:=\{u=1\}\cup\cdots\cup \{u=r\}$. The parametrizations $\varphi_1,\ldots,\varphi_r$ produce finite maps $\phi_i:\{u=i\}\longrightarrow S_i$. Then, there exists a generically-finite finite polynomial map $\Phi:\K^2_{u,v}\longrightarrow\K^2_{r,s}$ such that $\left.\Phi\right|_{\{u=i\}} = \phi_i$~\cite[Proposition 21]{Jel93}. From this setup, the map $f=\Phi\circ F$ is generically finite and $\cS_f=\Phi(S_F)=\Phi(L) = S_1\cup\cdots\cup S_r$.
\end{proof}

The following example illustrates how can one recover a polynomial map from its non-properness set.

\begin{example}\label{ex:universality}
The map $\Phi$ from the above sketch has an explicit expression: For each $i=1,\ldots,r$, we define functions $\psi_i:\C\longrightarrow\C$, $t\longmapsto 1$ if $r=1$, and $t\longmapsto \prod_{k\in [r]\setminus i}(t-k)/(i-k)$ otherwise. Then, we get
\begin{equation}\label{eq:psi_i}
\psi_i(t):=
\begin{cases}
0  ,~\text{ if }t\in [r]\setminus\{i\}, \\
1  ,~\text{ if }t=i, \\
\neq 0  ,~\text{ otherwise.}
\end{cases}
\end{equation}  Recall that for $i=1,\ldots,r$, the set $S_i$ is parametrized by $\varphi_i:=(\varphi_{i,1},~\varphi_{i,2}):\C\longrightarrow S_i$. Thus, we set $\Phi:\CtC$ to be 
\[
(u,~v)\longmapsto \left(\sum_{j=1}^r\varphi_{j,1}(v)\cdot\psi_j(u),~\sum_{j=1}^r\varphi_{j,2}(v)\cdot\psi_j(u)\right).
\] We leave it to the reader to check that $\Phi$ is a finite polynomial map satisfying $\Phi_{|_{\{u=i\}}} = \phi_i$. This demonstrates that if $k=\max(\deg \varphi_1,\ldots,\deg \varphi_r)$, then the degree of the map $f$ can be chosen to be at most $\deg \Phi\cdot\deg F = (k+r-1)\cdot(2+r)$.
\end{example}

\subsubsection{Singularities}
It is known that for any given $d,n,p\in\N$, a generic polynomial map in $\PCknp$ is proper~\cite{farnik2020whitney}. Jelonek proved in~\cite{jelonek2016semi} several qualitative results on the (non-)properness of generic members in affine families of polynomial maps from $\PCknp$. As for singularities of $\ccSf$ for such generic maps $f$ not much is known in general. For planar maps, however, those singularities turn out to be non-trivial, albeit having limited complexity due to the $\C$-uniruledness. 

The following result is an application of Theorem~\ref{thm:algebraic_correspondence} above which uses a careful analysis on the equations of $\ccSf$; the exact formulas can be found in~\cite{hilany2024polyhedral}.

\begin{theorem}[{{\cite[Theorem 2.19]{hilany2024polyhedral}}}]\label{thm:EHRose_nodes_non-properness}
Let $\bmA$ be a pair of polytopes in $\cItt$. Then, the non-properness set $\ccSf$ of a generic map $f\in\CA$ has only simple nodes as singularities outside $\bm{0}$. 
Furthermore, the numbers of those singular points can be computed using only the data of $\bmA$. 
\end{theorem}

The following open problem is  analogous to Problem~\ref{prb:singularities_polynomial_general} from~\S\ref{sec:affine_singul}.
\begin{problem}\label{prb:non-prop_singularites}
Given a tuple $\bmA$ of polytopes in $\cInk$, describe the singularities of the non-properness set for a generic map in $\CA$.
\end{problem}

\subsection{Some applications}\label{sub:non-prop_applications} In his earlier papers on the non-properness set~\cite{Jel99}, Jelonek provided several applications to classical problems from affine geometry. Later, other descriptive and computational results followed, and proved to be useful for tackling problems in affine geometry, real algebraic geometry and optimization~\cite{JK03,JeTi15,Hilany+2022}. In some cases, it lead to answers to some classical questions in those fields. In what follows, we present some of the more recent, less technical applications.

\subsubsection{Counting the number of isolated missing points}\label{sss:isolated_missing} Describing the image $f(X)$ of a dominant polynomial map $f:\XtY$ over an algebraically closed field is a classical problem. This is essentially equivalent to characterizing the non-properness set as it contains the set of \emph{missing points} $Y\setminus f(X)$. 

One of the first non-trivial questions is to understand the number of isolated missing points of $f$ whenever $X=Y=\C^2$.  Jelonek showed in~\cite{Jel99a} that it cannot exceed $d^2$ if $d$ is the degree of $f$. The proof relied on computing the number of intersection points of the curve $f^{-1}(\ccSf)$ with the line at infinity of $\P^2$. Whereas the upper bound is obtained using the degree estimation of $\ccSf$ appearing in Theorem~\ref{thm:non-prop-comp} above. The question of sharpness was not addressed until recently.

\begin{theorem}[\cite{Hilany+2022}]\label{thm:isolated_missing}
For any $\bmA\in\cItt$, a generic map $f\in\CA$ has at most $6~\deg f$  isolated points in $\C^2\setminus f(\C^2)$. Furthermore, there exists $\bmA\in\cItt$, and a map $f\in\CA$ such that the isolated points in $\C^2\setminus f(\C^2)$ equals to $\deg f-2$.
\end{theorem}

\begin{proof}[Sketch of proof]
Each point in $\ccSf$ of a polynomial map $f:\Cttt$ has a preimage with less than than $\mu(f)$ isolated, set-theoretic, points. The same goes for the discriminant $\cDf$. Hence, a self-intersection (or a singularity) $w\in \ccSf$ or a non-transversal intersection of $\ccSf\cap \cDf$ satisfies satisfy $\dim f^{-1}(w)>0$ or $\#f^{-1}(w)\leq \mu(f)-2$. Hence, the isolated points in $\C^2\setminus f(\C^2)$ are contained in those special points of $\ccSf$. The latter are described using the polyhedral methods obtained in Theorem~\ref{thm:algebraic_correspondence} above.
\end{proof}

\subsubsection{Jacobian Conjecture and \'etale maps}\label{sss:non-prop_Jacobian}
A polynomial map $f:\Cntn$ is called \emph{\'etale map} if the determinant $f'(z)$ of its Jacobian is nowhere vanishing in $\C^n$. Recall that the Jacobian Conjecture states that a polynomial map $f:\Cntn$ is invertible if and only if it is \'etale. 
Points at which an \'etale map is non-proper are those over which the fiber has strictly less points (counted \emph{without} multiplicities) than a generic fiber.
In particular, the map an \'etale map $f$ is an unramified covering over $\C^n\setminus\ccSf$. 
Since $\C^n$ is connected, the emptyness of $\ccSf$ is equivalent to invertibility of $f$. 
Consequently, studying the non-properness set of \'etale maps is a viable strategy to tackle the Jacobian Conjecture.

There are several results in this vein. 
For example, employing surgery theory methods, Nollet, Taylor, and  Xavier showed that the fundamental group $\pi_1(\C^n\setminus\ccSf)$ is not Abelian if $f$ is an \'etale map~\cite{nollet2009birationality}. 
More recently, Jelonek in~\cite{jelonek2022note} showed that for \'etale polynomial maps on the plane, the non-properness set, if non-empty, is a curve with self-intersections. 

We end this discussion by elaborating on the version of the Real Jacobian Conjecture alluded to in~\S\ref{sub:discr-real}.
Jelonek showed in~\cite[Theorem 8.2]{Jel02} that a real polynomial map, whose Jacobian has nowhere vanishing determinant, is invertible whenever its non-properness set has codimension $\geq 3$.
\begin{conjecture}[Real Topological Jacobian Conjecture]\label{con:topological_real_Jacobian}
Let $f:\Rntn$ be a real polynomial map satisfying $f'(x)>0$ for all $x\in\R^n$.
If $\cdim S_f\geq 2$, then $f$ is a bijection (and consequently $\cS_f = \emptyset$).
\end{conjecture}
One checks that the classical Jacobian Conjecture holds true whenever Conjecture~\ref{con:topological_real_Jacobian} is true~\cite[Proposition 8.3]{Jel02}.
Indeed, let $f:\Cntn$ be a non-invertible polynomial map.
Then, as remarked above, we have $\cS_f$ is non-empty.
By Theorem~\ref{thm:non-prop-comp}, we get $\cdim\cS_f =1$.
Hence, the set $\cS_f$ has real codimension two and is the non-properness set of the real polynomial map $\R f:\R^{2n}\longrightarrow\R^{2n}$.
If Conjecture~\ref{con:topological_real_Jacobian} holds true, we get $\R  f'(x) = 0$ for some $x\in\R^{2n}$.
This implies that $f$ is not \'etale. 
Hence the Jacobian Conjecture follows.

\subsubsection{Computing the set of atypical values of functions}\label{sss:appl_atypical} Recall that the bifurcation set $\cBf$ of a complex polynomial $f:\Cnto$ is the smallest subset $S\subset\C$ for which the function
\begin{align*}
\fr_{f^{-1}(\C\setminus S)}: & f^{-1}(\C\setminus S)\longrightarrow \C
\end{align*} is an analytical fibration. Unlike critical values, an effective method for computing the bifurcation values at infinity does not exist. Nevertheless, it can be approximated using the set $\cKfi$ of \emph{atypical values at infinity}
\begin{multline}\label{eq:Malgrange}
\cKfi:= \{ w\in\C~|~\exists \{x_k\}_{k\in\N}\subset\C^n,~\Vert x_k\Vert\longrightarrow\infty,\\ 
\Vert x_k\Vert\cdot \Vert\grad f(x_k)\Vert\longrightarrow 0,~f(x_k)\longrightarrow w \}
\end{multline}
The following inclusion is a consequence of Rabier's famous work on fibration theory~\cite{rabier1997ehresmann}, 
\begin{equation}\label{eq:Rabier}
\cBfi\subset\cKfi.
\end{equation} This relation was already known for various special cases by Fedoryuk~\cite{Fed76}, Kouchnirenko \cite{Kou76}, Pham \cite{Pha1983}, Boroughton~\cite{Bo83,broughton1988milnor}, and N\'emethi~\cite{Nem88}, and would later foresee numerous generalizations in a variety of settings~\cite{Par95,Tib99,kurdyka2000semialgebraic,jelonek2003generalized,NZ90}. 
Even though $\cKfi$ is finite~\cite[Corollary 2.12]{Tib99}, the inclusion~\eqref{eq:Rabier} can be strict (see e.g.,~\cite{puaunescu2000remarks}). It then remains uncertain whether equality in~\eqref{eq:Rabier} can hold if $\Vert x_k\Vert$ in~\eqref{eq:Malgrange} is replaced by $\Vert x_k\Vert^{\mathscr{L}}$ for some $0<\mathscr{L} < 1$. The value $\mathscr{L}$ is called the \emph{\L ojasiewicz exponent at infinity}~\cite{paunescu1997lojasiewicz}. 

Both the non-properness and the infinity bifurcation arise as degenerate behavior of fibers at infinity. Then, by carefully adapting the right definitions, the set $\cKfi$ is then expressed  using the non-properness set $\cS_F$ of a larger polynomial map $F$, constructed from $f$ and its partial derivatives. Consequently, Theorem~\ref{th:Jeln} computes $\cKfi$. The exact equations can be found in a paper of Jelonek and Kurdyka~\cite{JK03}.

\begin{theorem}[\cite{JK03}]\label{thm:Je-Kur-compute}
For any polynomial $f:\Cnto$, the set $\cKfi$ can be computed effectively.
\end{theorem}

\subsubsection{Tarski-Seidenberg in reverse}\label{sub:tarski-Seidenberg}
The famous Tarski–Seidenberg theorem implies that the image of a real Euclidean space, under a polynomial map $f:\RtRnp$, is semi-algebraic. 

\begin{problem}\label{prb:Tarski-Seid}
Characterize semi-algebraic subsets that are images of $\R^n$ under a polynomial map.
\end{problem} This problem took different forms~\cite{Jel02,Fer03}, and tackling it is beneficial to classification questions from topological types of real polynomial maps. 
Problem~\ref{prb:Tarski-Seid} turned out to be challenging, even for polynomial maps $f:\RtR$. 
Nevertheless, the set $K:=f(\R^2)$ cannot be arbitrarily complicated.
For instance, the boundary $\partial K:=\overline{K}\setminus \Int K $ is always contained in the bifurcation set $\cBf$.
 Thus, subsets like $R:=\{x>0,~xy>0\}$ are excluded.
 Indeed, as $R$ is open, its boundary, $\partial R$, has to lie in the non-properness set of $f$.
 However, this boundary is not $\R$-uniruled.
 This implies that $\partial R\not\subset\cS_f$ as otherwise it contradicts Theorem~\ref{thm:real-maps-non-prop} (c.f.~\cite[Example 7.3]{Jel02}). 
 In a similar fashion, if $R$ denotes the set $\{x^2+y^2-1>0\}$, and is the image of $\R^2$ under a polynomial map $f$, then we get $\partial R\subset\cS_f$. 
 This is also impossible since the boundary of $R$ is unbounded, and thus not $\R$-uniruled. 

A substantial amount of results have been proven regarding Problem~\ref{prb:Tarski-Seid}~\cite{BFG21,FerGam06,FerGam11,ueno2012convex,FerUeno14,Fer14,Fer16,Fer03}.
Let us present some of them. 
A \emph{layer} is a polyhedron in $\R^n$ that is affinely equivalent to $[-a, a]\times \R^{n-1}$ with $a>0$. Using a clever composition of polynomial maps, and an analysis of the non-properness set, Fernando, Gamboa, and Ueno showed the following result.

\begin{theorem}[\cite{Fer03,FerUeno14}]\label{thm:semi-alg-sets}
Let $K$ be a semi-algebraic subset of $\R^p$. Then, $K$ is the image of $\R^n$ under a polynomial map $\R^n\longrightarrow\R^p$ if $K$ is in one of the following cases

\begin{enumerate}

	\item\label{it:finite-compl} 
		$K = \R^{p}\setminus S\times\{0\}$, where $S\subset\R^{p-1}$ is a proper basic semi-algebraic subset,\\ 
	
	\item\label{it:forms-image} $K=\{\ell_1>0,\ldots,\ell_r>0\}$ for any independent linear forms $\ell_1,\ldots,\ell_r:\R^p\longrightarrow\R$, and\\
	
	\item\label{it:polytope-complement} $K=\R^p\setminus P$ or $\R^p\setminus \Int P$, where $P$ is a bounded polyhedron, and $P$ is not a hyperplane nor a layer.
\end{enumerate}
\end{theorem}

Theorem~\ref{thm:semi-alg-sets} tackles Problem~\ref{prb:Tarski-Seid} for semi-algebraic subsets of degree one, and those belonging to a hyperplane. To the author's knowledge, an analogue of Theorem~\ref{thm:semi-alg-sets} for families of full-dimensional semi-algebraic subsets of higher degrees has not yet been found.

\section{On classifying polynomial functions}\label{sec:functions}

In this section, we present some of the classical results that characterize the topology of polynomial functions $f:\Cnto$ using combinatorial objects. 
Furthermore, we demonstrate how these techniques can be employed to construct a multitude of non-trivial topological types and to estimate their number. 
The idea behind most topological characterization results often relies on analyzing the behavior of the function around its infinity bifurcation values $\cBfi$ (see e.g.,~\S\ref{sss:fibrations}). 
Consequently, results that describe the bifurcation set are instrumental in addressing Question~\ref{que:main00} for polynomial functions. Even though the bifurcation set is generally unyielding when it comes to computation, there has been significant advancement in this area~\cite{Fed76,Kou76,Nem86,Nem88,Bro88,ZaNe90,Par95,Par97,Zah96,Tib99,
ABLMH00,ALM00,JK03,DRT12,JelKur14,DT15,JelTib17}. 
However, in most cases, the existing methods impose certain limitations on the functions considered.

In contrast to the discriminant, the set of infinity bifurcation values is empty for many large families of polynomials. 
Kouchnirenko~\cite{Kou76} identified one such family using Newton polytopes. Broughton~\cite{Bro88}  and Nem\'ethi~\cite{Nem86,Nem88} distinguished other such families using the sizes of a polynomial's gradient, extending known methods for testing irregularities at infinity. 
Subsequently, Parusi\'nski demonstrated that whenever a function possesses only isolated singularities at ``infinity'', the change in the Euler characteristic of preimages can serve as a means of identifying bifurcation points~\cite{Par95}. In the 1990s, more systematic methods for computing exclusively non-trivial bifurcation points were presented. Alexandru and N\'emethi~\cite{ZaNe90} proposed an estimation for the bifurcation set based on gradients of the polynomial and a method to compute it using (in some cases) the faces of the Newton polytope. Subsequently, Zaharia demonstrated that this superset provides a reasonable approximation of the bifurcation set~\cite{Zah96}.

In the bivariate setting, i.e., for functions $\Ctto$, there are specialized tools for computing non-trivial bifurcation points~\cite{Dur98}, as well as results that provide tight upper bounds on their number~\cite{VO94,JK03,JeTi15,Gwo13}. 
For instance, the problem can be approached as a combinatorial one by studying the jumps in the  \emph{Euler-Poincar\'e characteristic} characteristics between different fibers, which are discrete topological invariants of curves~\cite{Suz74,HL84}. In a similar vein, colored trees and monodromy groups can be employed to identify non-trivial bifurcation values and, furthermore, to classify the various topological types of polynomial functions. 
These techniques have been utilized in the context of detecting bifurcation values and classifying topological types of polynomial functions, as well as Milnor numbers and level curves~\cite{Fou96,Dur98,Tib99}.

In~\S\ref{sub:Euler} and~\S\ref{sub:Newton-non-deg}, we elaborate on two of the results mentioned above for general $n$. The first result uses the Newton polytope to compute the bifurcation set effectively, while the second result uses Euler characteristic calculations instead. Subsequently, in~\S\ref{sub:planar_functions}, a result of Fourrier for the bivariate setting is described, which represents topological types using certain colored graphs.

In the course of this section, we describe how the aforementioned results can be combined to produce an effective method for constructing topological types of polynomial functions $\Ctto$ of a given degree $d$. Consequently, this provides lower bounds on these types. Finally, in~\S\ref{sub:univar}, we present one such lower bound.

For the sake of brevity, we will omit numerous important results in the theory of the bifurcation set, as we will focus only on those that give rise to methods for constructing topological types of complex polynomial functions.
We refer the reader to the monograph~\cite{Tib07} for an extensive survey on the topic.

\subsection{Bifurcation set from Euler characteristic}\label{sub:Euler} 
Let $f:\C^n\longrightarrow\C$ be a polynomial of degree $d$. 
We consider the decomposition $f = f_0 + f_1 + \cdots + f_d$ of $f$ into homogeneous components.
Consider the homogenization of $f$, given by
\[
\tilde{f} = z_0^df(z_1/z_0,\cdots,z_n/z_0),
\] and the hyperplane $H_\infty:=\{z_0=0\}\subset\P^n$ at infinity.  We define $A_f$ to be the set
\begin{equation}\label{eq:sing-infty}
\{z\in H_\infty~|~\partial f_d/\partial z_1 = \cdots = \partial f_d/\partial z_n = f_{d-1}=0\}.
\end{equation}

We say that $f$ has \emph{isolated singularities at infinity} if $A_f$ is either empty or a finite set. 

\begin{theorem}[\cite{Par95}]\label{thm:Parusinski}
Let $f:\Cnto$ be a polynomial with isolated singularities at infinity, and let $t_0$ be a value in $\cBf\setminus\cBfz$. 
Then, for every $t\in\C\setminus \cBf$ the following relation of Euler characteristics holds
\[
\chi(f^{-1}(t)) \neq \chi(f^{-1}(t_0)).
\] Furthermore, one obtains $f^{-1}(t)$ from $f^{-1}(t_0)$, up to homotopy, by attaching finitely many $(n-1)$-handles.
\end{theorem}

\begin{proof}[Sketch of proof]
Let $Y$ and $Z$ be two projective hypersurfaces that are equivalent as divisors (i.e. they are two zero-sets of sections of the same line bundle), and assume that they have isolated singularities. Then, the classical Euler-Milnor relation holds~\cite[Corollary 1.7]{parusinski1988generalization}
\begin{equation}\label{eq:Milnor-Euler}
 (-1)^n(\chi(Y) - \chi(Z)) = \sum_{p\in\Sing Y}\mu_p(Y) - \sum_{q\in\Sing Z}\mu_q(Z).
\end{equation} 
Let $C_\infty$ denote the hypersurface at infinity defined by $f_d$, let $X_t$ denote the fibers $f^{-1}(t)$ in $\C^n$, and let $\overline{X}_t$ denote its projective compactification in $\P^n$.
Then, in $\P^n$, we have $\overline{X}_t = C_\infty \cup X_t$ for each $t$.
Let $A_f$ to denote the singular locus of $\overline{X}_0$ at $C_\infty$. 
Then, we have $X_t$ is smooth for any $t\in\C\setminus\cB_f$, and thus $\Sing(\overline{X}_t) \subset A_f$.
Then, for any $t\not\in\cBf$ and $t_0\in\cBfi$, we get
\begin{equation}\label{eq:charact}
\chi(\overline{X}_t) - \chi(\overline{X}_{t_0}) = (-1)^n\cdot\sum_{p\in A_f}\left(\mu_p(\overline{X}_t) - \mu_p(\overline{X}_{t_0}) \right)
\end{equation}  from~\eqref{eq:Milnor-Euler}. 
Since $\overline{X}_t=X_t\cup C_\infty$ for every $t$, by the additivity of the Euler characteristic, the left-hand-side of~\eqref{eq:charact} equals $\chi(X_t) - \chi(X_{t_0})$.

Proving that the right-hand side of~\eqref{eq:charact} is non-zero is done by describing the polar curves of $V(f- t)\subset \P^n\times\C$ (see~\cite[\S 3.2]{Par95} for more details).
The second statement follows from the Malgrange condition and Morse theory.
\end{proof}

Note that Theorem~\ref{thm:Parusinski} may fail in presence of non-isolated singularities at infinity. 
As an example, consider $f := x + x^2~y~z$~\cite{10.1155/S1073792898000580}~\cite[Example 3.1]{vui2008critical}. Using Equation~\eqref{eq:sing-infty}, the set $A_f$ forms a cone $\{x=0\}\subset H_\infty$. The preimage $f^{-1}(0)$ is a disjoint union of a torus $(\C^*)^2\cong\{xy = 1\}$ and a complex plane $\C^2\cong\{x=0\}$. As for $t\neq 0$, the preimage $f^{-1}(t)$ is a disjoint union of $(\C\setminus\{t,0\})\times\C^*\cong\cup_{x\neq t,0}\{(y,z)~|~ yz = (t-x)/x^2\}$ with a singular conic $\{x\cdot y = 0\}$. Therefore, we get  $\chi(f^{-1}(0)) = \chi(P^{-1}(t)) = 1$. However, the function $f$ is clearly not a fibration around $0$.

Another observation for Theorem~\ref{thm:Parusinski} can be made regarding functions $\C^2\longrightarrow\C$; its singularities at infinity are always isolated. Hence, Theorem~\ref{thm:Parusinski} holds true for bivariate polynomials with no additional assumptions. This fact was already known by H\`a and L\^e in~\cite{HL84}. 

Lastly, Theorem~\ref{thm:Parusinski} is an excellent starting points to address the Question~\ref{que:main00} for polynomial functions. 

\begin{example}\label{ex:Parusinski}
Fix an integer polytope $\Delta\subset(\R_{\geq 0}^n)$. Then, for any generic $f\in\C^{\Delta}$, the Euler characteristic of a generic fiber is determined fully by $\Delta$~\cite{DanKho86}. Then, for each $d\in\N$ thanks to Theorem~\ref{thm:Parusinski}, a classification of all above integer polytopes in a simplex $T_d$ gives rise to a set of topologically distinct polynomial functions in $\PCkno$.
\end{example}

In the following part, we present yet another result that enhances the approach in Example~\ref{ex:Parusinski} for obtaining lower bounds on the number of topological types.

\subsection{Bifurcation set from Newton polytopes}\label{sub:Newton-non-deg}
Let $f:=\sum_{a\in\N^n}c_az^a$ be any multivariate polynomial, and let $\Gamma$ be a face of its Newton polytope. 
We say that $f$ is \emph{non-degenerate at $\Gamma$} if the set of solutions $\Sigma_\Gamma$ in $(\C^*)^n$ to 
\begin{equation}\label{eq:grad}
\frac{\partial f_\Gamma}{\partial z_1} = \cdots = \frac{\partial f_\Gamma}{\partial z_n} = 0
\end{equation} is empty.
It is called \emph{Newton non-degenrate} if $f$ is non-degenerate at every face $\Gamma$ of $\NP(f)$ that does not contain the origin. 
It follows from~\eqref{eq:grad} that Newton non-degenerate maps form a Zariski open subset in $\C^\Delta$~\cite{Kou76}. 
Newton polytopes turned out to be very useful in the study of polynomial maps~\cite{Est10,Est13,chen2014invertible}. Let us mention one of the earlier relevant results that was discovered by N\'emethy and Zaharia in the 1990s.
Recall in~\S\ref{sss:face-resultants} the definition of dicritical coherent faces of a tuple of polytopes.

\begin{theorem}[\cite{NZ90}]\label{thm:bifurcation-polytopes}
Let $\mathscr{D}(\Delta)$ denote the set of dicritical faces of $\Delta$ and let $f$ be a Newton non-degenerate polynomial $\C^n\longrightarrow\C$ satisfying $f(\bm{0})\neq 0$. Then, we have 
\begin{equation}\label{eq:bifur-polytopes}
\cBf\subset \{f(\bm{0})\}\cup \cBfz \cup \bigcup_{\Gamma\in\mathscr{D}(\Delta)}f_{\Gamma}(\Sigma_\Gamma).
\end{equation}
That is, bifurcation points of generic functions depend only on the coefficients at the faces of the corresponding Newton polytope.
\end{theorem}

\begin{proof}[Sketch of proof]
Let $M(f)\subset\C^n$ and $S_f\subset\C$ denote the two subsets 
\begin{align*}
M(f):=&~\{z~|~\exists~\lambda\in\C,~\grad f=\lambda\cdot z\},\text{and}\\ S_f:=&~\{ w~|~\exists \{z_k\}_{k\in\N}\subset M(f),~\Vert z_k\Vert\rightarrow\infty,~f(z_k)\rightarrow w\}
\end{align*} Then, it is enough to show the claim: \textit{if $f$ is a Newton nondegenerate polynomial with $f(\bm{z}) \neq 0$, and if $\{p(t)\}_{t\in ]0,1[}\subset M(f)$ is an analytic curve such that 
\[
\lim_{t\rightarrow 0}\Vert p(t)\Vert = \infty\text{\quad and \quad} \lim_{t\rightarrow 0}f(p(t))\in\C,
\] then $\lim_{t\rightarrow 0}f(p(t))$ belongs to the right-hand side of~\eqref{eq:bifur-polytopes}.} 

Under these assumptions, there is an analytic curve $\{\lambda(t)\}_{t\in]0,1[}\subset\C$ for which $p(t)$ and $\lambda (t)$ satisfy 
\begin{equation}\label{eq:identity-grad}
\frac{\partial f(p(t))}{\partial t} = \left\langle\frac{\partial p(t)}{\partial t},~\grad f(p(t))\right\rangle,\quad\text{and}\quad \grad f(p(t)) = \lambda(t)\cdot p(t).
\end{equation}  Thanks to the Curve Selection Lemma~\cite{milnor2016singular}, the values $p(t)$, $\lambda(t)$, and $f(p(t))$ are approximated as Puiseux series. Now, let $a~t^\alpha$ be the first order term of $p(t)$, with $a\in(\C^*)^n$ and $\alpha\in\Q^n$. Then, as $t\longrightarrow 0$, thanks to~\eqref{eq:identity-grad} we obtain \textit{(i)} the exponent vector $\alpha$ minimizes a face $\Gamma\prec\Delta$, and \textit{(ii)} the coefficients vector $a$ satisfies $\grad f_{\Gamma}(a) = 0$ (see~\cite[\S 3]{NZ90}). The limit $\Vert p(t)\Vert\rightarrow \infty$ and Newton non-degeneracy give restrictions on the direction of $\alpha$, and on the choice of $\Gamma$. Namely, the latter has to contain the origin. Altogether, this yields that $\Gamma$ is a dicritical face.
\end{proof}

Several generalizations of this result can be found throughout the literature; these include Esterov's results for polynomial maps on the complex torus~\cite{Est13}, and Chen \emph{et. al}'s work on real/mixed polynomial maps~\cite{chen2014invertible}

Those polynomials in which all monomials $z^\alpha$, $\alpha_1+\cdots+\alpha_n\leq d$ appear with a non-zero coefficient form a full-dimensional torus $T$ in the space $\PCdno$ of all polynomials $\C^n\longrightarrow\C$ of degree $d$. Clearly, all polynomials in $T$ have no dicritical faces. 
Accordingly, since Newton non-degenerate polynomials are also dense in $\PCdno$, above Theorem~\ref{thm:bifurcation-polytopes}, yields the following consequence. 

\begin{corollary}[\cite{Kou76}]\label{cor:Zariski-open-isolated}
For any $d,n\in\N$, the bifurcation set at infinity of a generic polynomial in $\PKkno$ is empty.
\end{corollary}

Theorem~\ref{thm:bifurcation-polytopes} can also be used to tackle the classification Question~\ref{que:main00}. 

\begin{example}[Using Theorem~\ref{thm:bifurcation-polytopes}]\label{ex:Newton}
We also consider here an integer polytope $\Delta\subset(\R_{\geq 0}^n)$. 
Then, the bifurcation sets of two functions $f,g\in\PCkno$ (in particular, the number of points therein) distinguish  the topological types of $f$ and $g$. 
On the one hand, Theorem~\ref{thm:bifurcation-polytopes} shows that as long as $f$ is Newton non-degenerate, the bifurcation set can be computed using data from only the dicritical faces. 
On the other hand, similarly to Example~\ref{ex:Parusinski}, the Euler characteristic of a generic fiber of any Newton non-degenerate polynomial depends only on $\Delta$. 
Then, a refined construction method of Example~\ref{ex:Parusinski} can be established that takes into account all possibilities of $\cBfi$ of Newton non-degenerate functions $\PCkno$. Such a method becomes far more practical and effective for the bivariate case as polygons are much simpler than polytopes.
\end{example}

\subsection{Resolution trees for bivariate polynomials}\label{sub:planar_functions} In the 1990s, L. Fourrier has developed a method that distinguishes the topological type of a bivariate polynomial functions using resolution of singularities~\cite{fourrier1994classification,Fou96}; two such functions $f,g:\Ctto$ are \emph{topologically equivalent at infinity} if there exists two compact subsets $K,L\subset\C^2$ and  homeomorphisms $H:\C^2\setminus \overset{\circ}{K}\longrightarrow\C^2\setminus \overset{\circ}{L}$ and $h:\Coto$ making the below diagram commutative
\begin{equation*}
\begin{tikzcd}
\C^2\setminus\overset{\circ}{K} \arrow{r}{H} \arrow[swap]{d}{f} & \C^2\setminus\overset{\circ}{L}  \arrow{d}{g} \\%
\C \arrow{r}{h}&  \C
\end{tikzcd}
\end{equation*} where the interior $\overset{\circ}{X}$, of a subset $X\subset\C^2$ is given using the Euclidean topology. By considering the resolution $\tilde{C}$ of singularities of $f\in\C[x,y]$, where $C:=\{f=0\}$, at the line at infinity in $\P^2$, one defines a weighed, colored graph $\RRTf$, called a \emph{reduced resolution tree of $f$} (see Figure~\ref{fig:resolution_tree}). 
The vertices of $\RRTf$ is given by the divisors, marked by weights reflecting their self-intersection numbers of the lifted curve, and every edge joins two intersecting divisors. 
The tree is furthermore colored by two types of arrows to reflect which the types of divisors they intersect the strict transforms. 
Two reduced resolution trees are said to be \emph{isomorphic} if 
if there is a bijection between the two trees (viewed as graphs), preserving colors and weights, and if their respective resolutions of singularities are equivalent up to blow-ups and blow-downs.

\begin{figure}[htb]\label{fig:resolution_tree}
\includegraphics[scale=.6]{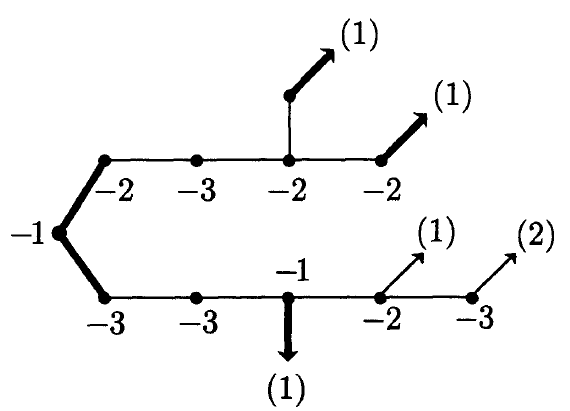}
\caption{Taken from~\cite{Fou96}: An example of a reduced resolution tree of $x^5y^3 - x^2y$.}
\end{figure}

We say that a polynomial $\Ctto$ is \emph{primitive} if a generic fiber is connected. The function $(x,~y)\longmapsto x^2$, for example, is not primitive. Through an analysis on sequences of blowups at points at infinity, the following result was shown.

\begin{theorem}[{{\cite{fourrier1994classification},\cite[Theorem 2]{Fou96}}}]\label{thm:Fourrier}
If two primitive polynomials $f,g:\Ctto$ are topologically equivalent at infinity, their respective reduced resolution trees $\RRTf$, and $\RRTg$ are isomorphic. 
\end{theorem}
In fact, the formulation of Theorem~\ref{thm:Fourrier} in~\cite{Fou96} is an equivalence; Fourrier shows that topological equivalence at infinity occurs precisely when the resolution trees are isomorphic and a monodromy condition at infinity is satisfied for both polynomials~\cite[\S 4.2]{Fou96}. We omit the latter condition as it is too technical to define.

 Fourrier's Theorem~\ref{thm:Fourrier} is one of the best methods to tackle Question~\ref{que:main00}: A classification of resolution trees at infinity, gives rise to a set of topologically distinct polynomial functions in $\PCkno$ for some $n,k\in\N$. This, in turn, imposes lower bounds on the number of those topological types. Let us present a more straightforward method for computing one such bound.

\subsection{Lower bounds from univariate polynomials}\label{sub:univar}

A univariate polynomial $f:\Coto$ of degree $d$ describes the topological type of its canonical extension $\Ctto$ given by $(x,~y)\longmapsto f(x)$. In this part, we show that there are at least $d^{d-3}$ topological types.

The critical values $c_1,\ldots,c_k\in\C$ of $f$ give rise to a collection of corresponding partitions $\bm{\Lambda}:=(\Lambda_1,\ldots,\Lambda_k)$ of $d$, where each $\Lambda_i$ is given by the orders of the roots of $f(z) - c_i$. The Riemann-Hurwitz Theorem implies that
\begin{equation}\label{eq:Riemann-Hurwitz}
\sum_{i=1}^k (d - \ell(\Lambda_i)) = d-1,
\end{equation} where $\ell(\Lambda_i)$ is the number of elements in a partition $\Lambda_i$. As a ramified covering of $\C$ with ramification profiles $\bm{\Lambda}$, the function $f$ gives rise to a labeled graph on the $2$-sphere $\bS^2$ in the following way (see Figure~\ref{fig:dessins}).

Pick any non-self-intersecting oriented loop $\gamma$ inside the complex projective line $\P^1:=\C\cup\{\infty\}$ that contains $\infty$ and the values $c_1,\ldots,c_k\in\C$. 
Then, the polynomial $f$ extends projectively to a meromorphic function $\tilde{f}:\P^1\longrightarrow\P^1$ representing a ramified cover of $\P^1$ with branch locus $\{c_1,\ldots,c_k,\infty\}$ and corresponding ramification profiles $(\Lambda_1,\ldots,\Lambda_k,(d))$. Thus, the pre-image $\Gamma_f:=\tilde{f}^{-1}(\gamma)\subset\P^1\cong\bS^2$ is a labeled, oriented, connected graph called the \emph{polynomial dessin associated to $f$}~\cite{IZ18}.  

Polynomial dessins constitute a particular case of Grothendieck's dessins d'enfant~\cite{grothendieck1997esquisse} which were independently discovered by several other mathematicians~\cite{jones1978theory,voisin1977cartes}. They turned out to be a powerful tool in numerous contexts such as Galois theory~\cite{lando2004graphs} and real enumerative algebraic geometry~\cite{Br06,TopologyofAlgebraicCurves,B15,EH17,IZ18}. We say that two polynomial dessins $\Gamma_1,\Gamma_2\subset\bS^2$ are \emph{homeomorphic} if there exists a homeomorphism $\varphi:\bS^2\longrightarrow\bS^2$ sending $\Gamma_1$ to $\Gamma_2$ and preserving the edge labels and orientations. 
\begin{figure}[htb]
\input{dessins}
\end{figure}

\begin{theorem}[\cite{IZ18}]\label{thm:polynomial-dessins}
Two polynomial functions $\C\longrightarrow\C$ are topologically equivalent if and only if their associated polynomial dessins are homeomorphic.
\end{theorem}

\begin{proof}[Sketch of proof]
Given an isomorphism class of a polynomial dessin $\Gamma\subset\bS^2$, one can construct the corresponding respective polynomial $f:\Coto$. This is explained in~\cite[proof of Proposition 2.7]{IZ18}. 

Notice that homeomorphisms $\Coto$ and homeomorphisms $\bS^2\longrightarrow\bS^2$ are interchangeable by adding/removing the point at infinity since $\bS^2\cong\P$. Then, two isomorphic dessins correspond to two topologically equivalent polynomial functions $\Coto$. The converse follows from the definitions.
\end{proof}

Theorem~\ref{thm:polynomial-dessins} allows to classify  polynomials $\C\longrightarrow\C$ of a given degree and ramification type. 
A lower bound is known for the number of non-homeomorphic polynomials thanks to the work of Looijenga~\cite[Proposition 2.4]{looijenga1974complement} (see also V.I. Arnol'd~\cite{arnol1996topological} and Lyashko~\cite{lyashko1983geometry}).

\begin{theorem}[\cite{looijenga1974complement}]\label{thm:LL-bound}
The number of topological types of polynomials $\Coto$ of degree $d$ is at least $d^{d-3}$. 
\end{theorem}

\begin{proof}[Sketch of proof]
The \emph{Lyashko-Looijenga} map $L\!L:\C^{d-1}\longrightarrow\C^{d-1}$ assigns to a complex \emph{normalized} polynomial 
\begin{align}\label{eq:univar-LL}
p= &~t^{d} + u_{2}t^{d-2}+\cdots + u_{d-1}t + u_{d}
\end{align} the polynomial $q=s^{d-1} + v_1s^{d-2}+\cdots + v_{d-2}s + v_{d-1}$, whose roots are the critical values of $p$~\cite{looijenga1974complement,lyashko1983geometry,arnol1996topological,zvonkine1997multiplicities}. It is well-known that $L\!L$ is a proper polynomial map whose topological degree is $d^{d-2}$. For an elaborate proof of this fact, we refer the reader to~\cite[\S 5.1.3]{lando2004graphs}. Note that the group $G:=\Z/d\Z$ acts on every generic fiber $L\!L^{-1}(v)$. Indeed, two polynomials of the form~\eqref{eq:univar-LL} are isomorphic if and only if they are related by a coordinate change of the form $t\longmapsto\alpha t$ for some $\alpha\in\C^*$ that is the $d$-th root of unity. 
This shows that $G$ is given by the map
\[
(u_2,\cdots,u_{d})\longmapsto (\alpha^{2}u_2,\ldots,~\alpha^{d-1}u_{d-1},~u_{d}).
\] Then, there are $d^{d-2}/d$ degree-$d$ normalized polynomials up to above transformations $t\longmapsto\alpha\cdot t$. Each of them represents one polynomial dessin~\cite{lando2004graphs}. 
This completes the proof by Theorem~\ref{thm:polynomial-dessins}.
\end{proof}

\section{On constructing topological types of polynomial maps on the plane}\label{sec:top-types_construct}

Constructing topological types of polynomial maps requires an effort to extend the existing toolbox used in the literature. 
The best strategy for developing new tools is to combine methods from algebraic geometry with those from polyhedral geometry, commutative algebra, and toric geometry
Ultimately, one of the main goals in this direction 
is to obtain correspondence theorems that yield structural and classification results, primarily via the analysis of combinatorial data such as Newton polytopes.
Another goal is to obtain implementable methods for computing topological invariants, and improve previously-known ones. This becomes evermore essential as existing tools frequently compute more objects than needed, making their implementation too expensive.

Throughout the literature, one finds a plethora of results relating geometrical structures and singularities of varieties to the combinatorial and numerical data obtained from their equations. 
Accordingly, a recurrent type of problem appearing in the literature is to classify all geometric objects sharing the same Newton tuple or other invariants, such as homology groups, Picard numbers, Chern classes, etc.  
On numerous occasions, many relevant topological invariants of a generic member of such families can be determined solely from its combinatorial data. 

One can pose the following question for the space $\KA$ with given tuple $\bmA\in\cInk$:  \textit{What kind of topological invariants are exclusively shared among generic members in $\KA$?} 
\textit{Can we compute those invariants using only $\bmA$?} One can also ask the converse: \textit{How should the tuple $\bmA$ look like if a generic member satisfies a prescribed topological property?} 

Some of these questions have been examined in previous sections, for example in the context of computing singularities of generic maps~\S\ref{sss:polyhedral-singularities}, the bifurcation set of generic polynomial functions~\S\ref{sub:Newton-non-deg}, the Jacobian Conjecture ~\S\ref{sub:Jacobian_conjecture},
computing the non-properness set~\S\ref{sub:non-prop_computational}, etc. These results emanate from theories based from algebraic geometry, where numerous established classical tools utilize the polyhedral structure of polynomials tuples to describe data from topology, singularities, divisors, sheafs, etc. Such results are deep and require fundamental notions from algebraic geometry such as Koszul complexes~\cite{Kou76}, Exact cohomology sequences~\cite{khovanskii1978newton}, Hodge theory~\cite{DanKho86}, toric geometry~\cite{Est10}, tropical geometry~\cite{Est13} etc. Other results that do not come from algebraic geometry have roots in works of singularity theory, such as Reeb graphs~\cite{farley2005discrete}, dessins d'enfant~\cite{Looijenga1999}, blow-up graphs~\cite{Fou96}, etc.

In this section we illustrate how results presented in~\S\ref{sec:functions},~\S\ref{sec:affine_singul}, and~\S\ref{sec:sing_infty} can be applied to classify topological types of quadratic polynomial maps $\Cttp$, to construct topological types of higher degrees, to identify generic ones, and to find lower bounds on their number. 
In~\S\ref{sub:dist-invars}, we elaborate on some of the numerical data of algebraic invariants of maps that distinguish topological types. 
We then illustrate how these invariants were used effectively to classify quadratic polynomial maps $\Kttp$, and to determine topological invariants from combinatorics of Newton pairs. 
We finish this section by presenting in~\ref{sub:some_lower} a lower bound on the number of topological types.

\subsection{Topological invariants from algebra}\label{sub:dist-invars} L\^e showed in~\cite[Theorem 3.3]{le1973topologie} that an ambient homeomorphism around a neighborhood of an isolated singular point at the origin preserves the Milnor number. Using this result, Gaffney and Mond showed in the 1990s that, if an analytic map $f:\CtC$ is excellent (see~\S\ref{sub:history_cusps}), the respective numbers of ordinary cusps and nodes $\kappa(f)$ and $\nu(f)$ are its topological invariants~\cite{gaffney1991cusps}. 
This was an extension of Fukuda and Ishikawa's theorem for stable smooth maps in the plane; it stated that the parity of the number of cusps is a topological invariant~\cite{fukuda1987number}.

The fact that singularities of maps are topological invariants turned out to be useful in some classification results; 
recently Farnik and Jelonek, and later Farnik, Jelonek and Migus~\cite{FJM18} used this approach to classify quadratic topological types of real and complex maps in $\Kttp$. We only state one of those results.

\begin{theorem}[\cite{FJ17}]\label{thm:quadratic-types}
The number of topological types of polynomial maps in $\PCttt$ is equal to $13$.
\end{theorem}

\begin{proof}[Sketch of proof]
We present the sketch for the complex case only. Two maps $f,g:\CtC$ are said to be \emph{linearly equivalent} if there are affine linear isomorphisms $\varphi,\psi:\CtC$ satisfying $g = \psi\circ f\circ\varphi$. 
Let $\cQ$ be the space of quadratic polynomial maps $\CtC$, 
and let $\cG\!A$ be the space of all left-right affine compositions on maps $\CtC$. It makes up for a group action $\cG\!A\times\cQ\longrightarrow\cQ$, $((\psi,\varphi),~f)\longmapsto \psi\circ f\circ\varphi$. Since $\cG\!A$ and $ \cQ$ are two complex Euclidean spaces with the same dimension ($=12$), the orbits of $\cG\!A$ subdivide $ \cQ$ into families of maps having the same linear types. Using the above-mentioned result of Gaffney and Mond~\cite{gaffney1991cusps}, a classification of the orbits of $\cG\!A$ is obtained by carefully analyzing all topological types that the triple $(\cCf,\cDf,\cS_f)$ can have. Since linear equivalence is finer than topological one, this leads to a classification of all topological types in $\cQ$.
\end{proof}

In what follows, we present an extension of Gaffney and Mond's result for polynomial maps $f:\Cttt$. Recall that $\cCf$ and $\cDf$ denote the critical locus and the discriminant of $f$ respectively. Let $C$ be an irreducible component of the critical locus $\cCf$. 
Then, for all but finitely-many $z\in C$, the map $\fr_{U}:U\longrightarrow f(U)$ is a ramified covering of degree $m$, where $U\subset \C^2$ is any small enough neighborhood around $z$.
\begin{definition}\label{def:top-mult-crit}
The \emph{topological multiplicity} of an irreducible component $C\subset\cCf$ of the critical locus is $m-1$, where $m$ is the degree of the covering $\fr_U$. 
The \emph{topological multiplicity} of an irreducible component $D\subset\cDf$ is defined to be the sum of all topological multiplicities of components $C\subset \cCf$ mapping to $D$.
\end{definition}

The line $\{x = 0\}$, for example, is the critical locus of $(x,~y)\longmapsto (x^n,~y)$ and has topological multiplicity $n-1$. 
We analogously define the topological multiplicity for the non-properness set of $\ccSf$: Let $S\subset\ccSf$ be any irreducible component. 
Then, from~\cite{Jel99}, there is a Zariski open subset $\cZ\subset S$, such that for each $w\in\cZ$, the number of isolated points in $f^{-1}(w)$ is independent of $w$. We call any such point $w\in \cZ$ \emph{distinguished}. 

\begin{definition}\label{def:top-mult-non-prop}
The \emph{topological multiplicity} of an irreducible component of $\cS_f$ is the difference $\mu(f)-\#f^{-1}(w)$, where $w$ is any distinguished point of $S$ and $\mu(f)$ is the topological degree of $f$.
\end{definition}

\begin{example}\label{ex:top-mult-non-prop}
The polynomial map $(x,~y)\longmapsto(y^2-y,~x^2y^2)$ has topological degree $4$, and satisfies $\cS_f=\{u=0\}$. The topological multiplicity of $\cS_f$ is equal to $2$ since $\#f^{-1}(0,t) = 2$ for every $t\neq 0$. 
\end{example}

We have the following observation.

\begin{claim}[{{\cite[Proposition 2.7]{hilany2024polyhedral}}}]\label{thm:separator}
Let $f,g:\CtC$ be two topologically equivalent polynomial maps. Then, for every sequence of integers $ m,\iota,s,\mu\in\N$, every pair of curves from the set $\{(\cCf,\cCg),(\cDf,\cDg),(\cS_f,\cS_g)\}$ has the same number of irreducible components that have the same topological multiplicity $m$, geometric genus $\iota$, and $s$ singularities of Milnor number $\mu$.
\end{claim}

\begin{proof}[Sketch of proof]
Recall that the critical loci, discriminants and non-properness sets are preserved under ambient homeomorphisms. By a Theorem of Gau and Lipman~\cite[Lemma A8]{gau1983differential}, irreducible components are sent to irreducible components. 
As the number of preimages has to be preserved under homeomorphisms, and thus topological multiplicities are preserved as well. 
Finally, the last claim of the theorem follows from the result of L\^e~\cite[Theorem 3.3]{le1973topologie} mentioned at the beginning of~\S\ref{sub:dist-invars}.
\end{proof}

\begin{remark}\label{rem:further}
In addition to those appearing in Theorem~\ref{thm:separator}, further properties of $f$ can distinguish its topological type; we mention the fundamental groups of~$\C^2\setminus \cCf$ or of~$\C^2\setminus\cDf$, and the homeomorphism types of~$\cDf\cap\cS_f$ or of~$\C^2\setminus f(\C^2)$. Some progress has been made towards the study of these invariants~\cite{dolgachev2006fundamental}.
\end{remark}

All the invariants in Theorem~\ref{thm:separator} were computed for generic maps in $\PdCt$~\cite{FJR19}; such a map: (i) represents one topological type, (ii) is proper, (iii) its critical locus is smooth, irreducible of topological multiplicity one, and (iv) each of the genus of $\cCf$, the topological degree, and the values $\mu(f)$ and $\kappa(f)$ can be expressed in terms of $d_1$ and $d_2$. 
The next step is to extend this description to affine families $\sF\subset\PdCt$.

Let $\bmA:=(A_1,A_2)\in\cItt$ be a pair of lattice polytopes in $\R^2_{\geq 0}$, where neither $A_1$ nor $A_2$ contains the origin $\bm{0}$, and each of $A_1$ and $A_2$ contains a set $S\subset\N^2$, such that where no four points of $S$ belong to a line. 
Such a pair $\bmA$ is called \emph{conical}.

\begin{theorem}[\cite{hilany2024polyhedral}]\label{thm:EK1}
Let $\bmA\in\cItt$ be a conical pair. Then, for every generic map $f\in\CA$, the formulas for the following invariants are computed in terms of combinatorics of $\bmA$: The topological degree of $f$, the genus, number of irreducible components, and topological multiplicities of $\cCf$, and of $\cDf$, the value $\nu(f)+\kappa(f)$, and the number of singular points of $\ccSf$, together with their Milnor numbers. The exact formulas can be found in~\cite[Theorem 2.19]{hilany2024polyhedral}, and the software implementation in\\ \url{https://mathrepo.mis.mpg.de/PolyhedralTypesOfPlanarMaps}.
\end{theorem} This theorem is a byproduct of the study done for the sake of above Theorems~\ref{thm:poly-type-discriminant} and~\ref{thm:EHRose_nodes_non-properness}. 

\subsection{Lower bounds and asymptotics for topological types}\label{sub:some_lower}

The space $\PdCt$ contains a set of affine families $\{\sF_{\bmA}\}_{\bmA}$ running over conical pairs $\bmA$ for which the maps $f:=(f_1,f_2)\in\sF_{\bmA}\cong\CA$ have degree pairs $(d_1,d_2)$. Thanks to a result of Sabbah~\cite{Sab83}, each $\sF_{\bmA}$ contains a Zariski open set of topologically-equivalent maps. The topological invariants of the latter can be determined from $\bmA$. Therefore, thanks to Theorem~\ref{thm:EK1}, we obtain a list of numerical invariants that characterizes the topological type of a generic $f\in\sF_{\bmA}$, and represent it as a vector $\Psi(A)\in\N^{12}$, which is called the \emph{polyhedral type} of $f$, given by a map $\Psi:\ccCgeq\to\N^{12}$. To this end, lower bounds on topological types of planar maps are obtained by enumerating the number of possible polyhedral types arising from conical pairs. 

For any $d\in\N$, let $\cT(d)$ denote the number of topological types of polynomial maps in $\Cttt$ of degree $d$.

\begin{theorem}[{{\cite[Theorem 1.1]{hilany2024polyhedral}}}]\label{thm:main-bound}
   It holds that $26~\leq~\cT(3)$ and $3~217\ \leq\ \cT(4)$.
\end{theorem} 
		
 The lower bounds obtained by counting the polyhedral types (as done in Theorem~\ref{thm:main-bound}) are not expected to be optimal for higher degrees. 
 This is due to a combination of two observations. 
 First, the number of distinct pairs of integer polytopes in $\cItt$ of degree $d$ is at most polynomial in $d$~\cite{MR2891244}. 
 Second, we have the following lower bound.

\begin{corollary}\label{cor:top_types_unmixed}
For every $d\in\N$, it holds that
\begin{align}\label{eq:lower_1}
		\frac{d^{2(d-3)}}{2} \leq & ~\cT(d,~2).
\end{align}
\end{corollary}
\begin{proof}
Let $f:=(f_1,f_2):\CtC$ be an \emph{unmixed} polynomial map  of bi-degree $(d_1,d_2)$. 
That is, $f_1$ and $f_2$ are univariate polynomials of degrees $d_1$ and $d_2$ respectively. 
Then, the preimage of a generic horizontal line $L$ is a collection of horizontal lines $\cK:=\{K_1,\ldots,K_{d_2}\}$.
Furthermore, for any tubular neighborhood $V\supset L$, we have $f^{-1}(V)$ is a collection of disjoint tubular neighborhoods of $\cK$. 
Now, pick one such neighborhood $U\subset f^{-1}(V)$. 
The ramified covering $\fr_U:U\longrightarrow V$ is determined by the topological type of the function $f_1:\C\longrightarrow\C$. 
Let $g$ is another unmixed polynomial map satisfying $\psi\circ f=g\circ \varphi$ for some homeomorphisms $\varphi,\psi:\CtC$. 
Then, the restricted map $\left.g\right|_{\varphi(U)}:\varphi(U)\longrightarrow\psi(V)$ is also topologically equivalent to $\fr_U$. 
Hence, either $\deg g_2 \leq 1$ or $\varphi(U)$ does not contain points $(a,b)$, where $b$ is a critical value of $g_2$. 
Applying this argument for vertical lines as well, we deduce that $f_1$ and $f_2$ have the same topological types as $g_1$ and $g_2$ respectively (up to index re-enumeration). 
The lower bound~\eqref{eq:lower_1} then follows from Theorem~\ref{thm:LL-bound} applied twice. 
The denominator accounts for the coordinate permutation.
\end{proof}

The bound in Corollary~\ref{cor:top_types_unmixed} has room for improvement by, for example, considering the non-proper maps. However, it is unclear whether it is asymptotically sharp: 

\begin{question}
 Is it true that there exists $c\in\N$ such that for large enough $d\in\N$, we have $\cT(d,~2) \leq d^{2d+c}$? 
\end{question}

\subsection*{Acknowledgments} 
The author thanks Dmitry Kerner, Khazhgali Kozhasov and Timo de Wolff  for helpful discussions and valuable input during the preparation of this survey. 
The author would also like to thank the anonymous referee for their valuable comments and suggestions, which helped improve the manuscript.

\subsection*{Contact}\ \\
  Boulos El Hilany,\\
  Institut f\"ur Analysis und Algebra, \\
TU Braunschweig, Universit\"atsplatz 2.\\
 38106 Braunschweig, Germany.\\
  \href{mailto:b.el-hilany@tu-braunschweig.de}{b.el-hilany@tu-braunschweig.de}, \href{mailto:boulos.hilani@gmail.com}{boulos.hilani@gmail.com},\\
 \href{https://boulos-elhilany.com}{boulos-elhilany.com}.

\input{bibliography.bbl}

\end{document}

%% file: stratified_map.tex
\tikzset{every picture/.style={line width=0.75pt}} 

\begin{tikzpicture}[x=0.75pt,y=0.75pt,yscale=-1,xscale=1]

\draw  [color={rgb, 255:red, 255; green, 255; blue, 255 }  ,draw opacity=1 ][fill={rgb, 255:red, 245; green, 166; blue, 35 }  ,fill opacity=0.3 ] (343.36,148.44) .. controls (367.61,147.96) and (333.88,148.55) .. (390.31,148.46) .. controls (382.97,143.24) and (341.57,131.07) .. (343.23,101.33) .. controls (342.71,157.36) and (342.87,91.79) .. (343.36,148.44) -- cycle ;
\draw  [color={rgb, 255:red, 255; green, 255; blue, 255 }  ,draw opacity=1 ][fill={rgb, 255:red, 80; green, 227; blue, 194 }  ,fill opacity=0.3 ] (343.13,53.99) .. controls (367.39,54.47) and (333.66,53.88) .. (390.09,53.96) .. controls (382.75,59.19) and (341.34,71.36) .. (343.01,101.1) .. controls (342.49,45.07) and (342.65,110.64) .. (343.13,53.99) -- cycle ;
\draw [color={rgb, 255:red, 74; green, 144; blue, 226 }  ,draw opacity=1 ][line width=1.5]    (343.13,53.99) -- (343.2,148.21) ;
\draw  [color={rgb, 255:red, 255; green, 255; blue, 255 }  ,draw opacity=0 ][fill={rgb, 255:red, 245; green, 166; blue, 35 }  ,fill opacity=0.3 ] (198.83,101.1) -- (152.3,102) -- (152.2,148.21) -- (198.7,148.21) -- cycle ;
\draw  [color={rgb, 255:red, 255; green, 255; blue, 255 }  ,draw opacity=0 ][fill={rgb, 255:red, 80; green, 227; blue, 194 }  ,fill opacity=0.3 ] (198.8,53.99) -- (151.89,54.14) -- (152.01,101.1) -- (198.67,101.1) -- cycle ;
\draw  [color={rgb, 255:red, 255; green, 255; blue, 255 }  ,draw opacity=0 ][fill={rgb, 255:red, 245; green, 166; blue, 35 }  ,fill opacity=0.3 ] (247.33,101.1) -- (199.17,101.1) -- (199.2,148.21) -- (246.86,148.21) -- cycle ;
\draw  [color={rgb, 255:red, 255; green, 255; blue, 255 }  ,draw opacity=0 ][fill={rgb, 255:red, 80; green, 227; blue, 194 }  ,fill opacity=0.3 ] (246.12,54.06) -- (199.13,53.99) -- (199.17,101.1) -- (247.33,101.1) -- cycle ;
\draw [color={rgb, 255:red, 208; green, 2; blue, 27 }  ,draw opacity=1 ][fill={rgb, 255:red, 65; green, 117; blue, 5 }  ,fill opacity=1 ][line width=1.5]    (152.01,101.1) -- (247.33,101.1) ;
\draw [color={rgb, 255:red, 0; green, 0; blue, 0 }  ,draw opacity=1 ][fill={rgb, 255:red, 65; green, 117; blue, 5 }  ,fill opacity=1 ][line width=1.5]    (151.81,53.99) -- (152.01,101.1) ;
\draw [color={rgb, 255:red, 0; green, 0; blue, 0 }  ,draw opacity=1 ][fill={rgb, 255:red, 65; green, 117; blue, 5 }  ,fill opacity=1 ][line width=1.5]    (152.01,101.1) -- (152.2,148.21) ;
\draw [color={rgb, 255:red, 0; green, 0; blue, 0 }  ,draw opacity=1 ][fill={rgb, 255:red, 65; green, 117; blue, 5 }  ,fill opacity=1 ][line width=1.5]    (246.46,53.99) -- (246.66,101.1) ;
\draw [color={rgb, 255:red, 0; green, 0; blue, 0 }  ,draw opacity=1 ][fill={rgb, 255:red, 65; green, 117; blue, 5 }  ,fill opacity=1 ][line width=1.5]    (246.66,101.1) -- (246.86,148.21) ;
\draw [color={rgb, 255:red, 65; green, 117; blue, 5 }  ,draw opacity=1 ][fill={rgb, 255:red, 65; green, 117; blue, 5 }  ,fill opacity=1 ][line width=1.5]    (390.09,53.96) -- (343.13,53.99) ;
\draw    (277.01,101.1) -- (315.33,101.1) ;
\draw [shift={(317.33,101.1)}, rotate = 180] [color={rgb, 255:red, 0; green, 0; blue, 0 }  ][line width=0.75]    (4.37,-1.32) .. controls (2.78,-0.56) and (1.32,-0.12) .. (0,0) .. controls (1.32,0.12) and (2.78,0.56) .. (4.37,1.32)   ;
\draw [color={rgb, 255:red, 74; green, 144; blue, 226 }  ,draw opacity=1 ][line width=1.5]    (198.97,53.99) -- (199.04,148.21) ;
\draw  [color={rgb, 255:red, 208; green, 2; blue, 27 }  ,draw opacity=1 ][fill={rgb, 255:red, 255; green, 255; blue, 255 }  ,fill opacity=1 ] (195.92,101.1) .. controls (195.92,99.31) and (197.38,97.85) .. (199.17,97.85) .. controls (200.96,97.85) and (202.42,99.31) .. (202.42,101.1) .. controls (202.42,102.89) and (200.96,104.35) .. (199.17,104.35) .. controls (197.38,104.35) and (195.92,102.89) .. (195.92,101.1) -- cycle ;
\draw    (343.17,101.1) .. controls (342.76,80.25) and (351.76,72.54) .. (390.09,53.96) ;
\draw  [color={rgb, 255:red, 0; green, 0; blue, 0 }  ,draw opacity=1 ][fill={rgb, 255:red, 255; green, 255; blue, 255 }  ,fill opacity=1 ] (387.96,53.96) .. controls (387.96,52.79) and (388.91,51.84) .. (390.09,51.84) .. controls (391.26,51.84) and (392.22,52.79) .. (392.22,53.96) .. controls (392.22,55.14) and (391.26,56.09) .. (390.09,56.09) .. controls (388.91,56.09) and (387.96,55.14) .. (387.96,53.96) -- cycle ;
\draw  [fill={rgb, 255:red, 255; green, 255; blue, 255 }  ,fill opacity=1 ] (341.26,52.12) -- (345.01,52.12) -- (345.01,55.86) -- (341.26,55.86) -- cycle ;
\draw  [fill={rgb, 255:red, 255; green, 255; blue, 255 }  ,fill opacity=1 ] (341.33,150.08) -- (345.08,150.08) -- (345.08,146.34) -- (341.33,146.34) -- cycle ;
\draw [color={rgb, 255:red, 65; green, 117; blue, 5 }  ,draw opacity=1 ][fill={rgb, 255:red, 65; green, 117; blue, 5 }  ,fill opacity=1 ][line width=1.5]    (199.13,53.99) -- (246.46,53.99) ;
\draw [color={rgb, 255:red, 65; green, 117; blue, 5 }  ,draw opacity=1 ][fill={rgb, 255:red, 65; green, 117; blue, 5 }  ,fill opacity=1 ][line width=1.5]    (199.12,148.21) -- (246.86,148.21) ;
\draw [color={rgb, 255:red, 65; green, 117; blue, 5 }  ,draw opacity=1 ][fill={rgb, 255:red, 65; green, 117; blue, 5 }  ,fill opacity=1 ][line width=1.5]    (151.89,54.14) -- (198.8,53.99) ;
\draw [color={rgb, 255:red, 65; green, 117; blue, 5 }  ,draw opacity=1 ][fill={rgb, 255:red, 65; green, 117; blue, 5 }  ,fill opacity=1 ][line width=1.5]    (152.2,148.21) -- (199.2,148.21) ;
\draw  [fill={rgb, 255:red, 255; green, 255; blue, 255 }  ,fill opacity=1 ] (197.26,52.12) -- (201.01,52.12) -- (201.01,55.86) -- (197.26,55.86) -- cycle ;
\draw  [fill={rgb, 255:red, 255; green, 255; blue, 255 }  ,fill opacity=1 ] (196.83,146.34) -- (200.57,146.34) -- (200.57,150.08) -- (196.83,150.08) -- cycle ;
\draw  [color={rgb, 255:red, 0; green, 0; blue, 0 }  ,draw opacity=1 ][fill={rgb, 255:red, 255; green, 255; blue, 255 }  ,fill opacity=1 ] (244.33,53.99) .. controls (244.33,52.81) and (245.29,51.86) .. (246.46,51.86) .. controls (247.64,51.86) and (248.59,52.81) .. (248.59,53.99) .. controls (248.59,55.16) and (247.64,56.12) .. (246.46,56.12) .. controls (245.29,56.12) and (244.33,55.16) .. (244.33,53.99) -- cycle ;
\draw  [color={rgb, 255:red, 0; green, 0; blue, 0 }  ,draw opacity=1 ][fill={rgb, 255:red, 255; green, 255; blue, 255 }  ,fill opacity=1 ] (244.73,148.21) .. controls (244.73,147.04) and (245.69,146.08) .. (246.86,146.08) .. controls (248.04,146.08) and (248.99,147.04) .. (248.99,148.21) .. controls (248.99,149.39) and (248.04,150.34) .. (246.86,150.34) .. controls (245.69,150.34) and (244.73,149.39) .. (244.73,148.21) -- cycle ;
\draw  [color={rgb, 255:red, 0; green, 0; blue, 0 }  ,draw opacity=1 ][fill={rgb, 255:red, 255; green, 255; blue, 255 }  ,fill opacity=1 ] (150.08,148.21) .. controls (150.08,147.04) and (151.03,146.08) .. (152.2,146.08) .. controls (153.38,146.08) and (154.33,147.04) .. (154.33,148.21) .. controls (154.33,149.39) and (153.38,150.34) .. (152.2,150.34) .. controls (151.03,150.34) and (150.08,149.39) .. (150.08,148.21) -- cycle ;
\draw  [color={rgb, 255:red, 0; green, 0; blue, 0 }  ,draw opacity=1 ][fill={rgb, 255:red, 255; green, 255; blue, 255 }  ,fill opacity=1 ] (149.68,53.99) .. controls (149.68,52.81) and (150.63,51.86) .. (151.81,51.86) .. controls (152.98,51.86) and (153.93,52.81) .. (153.93,53.99) .. controls (153.93,55.16) and (152.98,56.12) .. (151.81,56.12) .. controls (150.63,56.12) and (149.68,55.16) .. (149.68,53.99) -- cycle ;
\draw  [color={rgb, 255:red, 208; green, 2; blue, 27 }  ,draw opacity=1 ][fill={rgb, 255:red, 255; green, 255; blue, 255 }  ,fill opacity=1 ] (243.41,101.1) .. controls (243.41,99.31) and (244.87,97.85) .. (246.66,97.85) .. controls (248.45,97.85) and (249.91,99.31) .. (249.91,101.1) .. controls (249.91,102.89) and (248.45,104.35) .. (246.66,104.35) .. controls (244.87,104.35) and (243.41,102.89) .. (243.41,101.1) -- cycle ;
\draw  [color={rgb, 255:red, 208; green, 2; blue, 27 }  ,draw opacity=1 ][fill={rgb, 255:red, 255; green, 255; blue, 255 }  ,fill opacity=1 ] (148.76,101.1) .. controls (148.76,99.31) and (150.21,97.85) .. (152.01,97.85) .. controls (153.8,97.85) and (155.25,99.31) .. (155.25,101.1) .. controls (155.25,102.89) and (153.8,104.35) .. (152.01,104.35) .. controls (150.21,104.35) and (148.76,102.89) .. (148.76,101.1) -- cycle ;
\draw [color={rgb, 255:red, 65; green, 117; blue, 5 }  ,draw opacity=1 ][fill={rgb, 255:red, 65; green, 117; blue, 5 }  ,fill opacity=1 ][line width=1.5]    (390.31,148.46) -- (343.36,148.44) ;
\draw  [fill={rgb, 255:red, 255; green, 255; blue, 255 }  ,fill opacity=1 ] (341.48,150.31) -- (345.23,150.31) -- (345.23,146.57) -- (341.48,146.57) -- cycle ;
\draw    (343.23,101.33) .. controls (342.82,122.18) and (351.98,129.89) .. (390.31,148.46) ;
\draw  [color={rgb, 255:red, 208; green, 2; blue, 27 }  ,draw opacity=1 ][fill={rgb, 255:red, 208; green, 2; blue, 27 }  ,fill opacity=1 ] (340.84,101.1) .. controls (340.84,99.9) and (341.81,98.93) .. (343.01,98.93) .. controls (344.2,98.93) and (345.17,99.9) .. (345.17,101.1) .. controls (345.17,102.3) and (344.2,103.27) .. (343.01,103.27) .. controls (341.81,103.27) and (340.84,102.3) .. (340.84,101.1) -- cycle ;
\draw  [color={rgb, 255:red, 0; green, 0; blue, 0 }  ,draw opacity=1 ][fill={rgb, 255:red, 255; green, 255; blue, 255 }  ,fill opacity=1 ] (388.18,148.46) .. controls (388.18,149.64) and (389.14,150.59) .. (390.31,150.59) .. controls (391.49,150.59) and (392.44,149.64) .. (392.44,148.46) .. controls (392.44,147.29) and (391.49,146.34) .. (390.31,146.34) .. controls (389.14,146.34) and (388.18,147.29) .. (388.18,148.46) -- cycle ;

\draw (347.17,101.1) node [anchor=west] [inner sep=0.75pt]  [font=\tiny]  {$( 0,0)$};
\draw (197.01,104.5) node [anchor=north east] [inner sep=0.75pt]  [font=\tiny]  {$( 0,0)$};
\draw (202.73,44.99) node [anchor=south] [inner sep=0.75pt]  [font=\tiny]  {$y$};
\draw (252.58,101.1) node [anchor=west] [inner sep=0.75pt]  [font=\tiny]  {$x$};
\draw (394.22,53.96) node [anchor=west] [inner sep=0.75pt]  [font=\tiny]  {$( 1,1)$};
\draw (394.44,148.46) node [anchor=west] [inner sep=0.75pt]  [font=\tiny]  {$( 1,-1)$};
\draw (250.59,53.99) node [anchor=west] [inner sep=0.75pt]  [font=\tiny]  {$( 1,1)$};
\draw (148.08,148.21) node [anchor=east] [inner sep=0.75pt]  [font=\tiny]  {$( -1,-1)$};

\end{tikzpicture}

\caption{The stratified map from Example~\ref{exp:stratification+blow-up}}\label{fig:strat-map}

%% file: cusps_nodes.tex
\tikzset{every picture/.style={line width=0.75pt}} 

\tikzset{every picture/.style={line width=0.75pt}} 

\begin{tikzpicture}[x=0.75pt,y=0.75pt,yscale=-1,xscale=1]

\draw [color={rgb, 255:red, 155; green, 155; blue, 155 }  ,draw opacity=0.5 ][line width=1.5]    (94.67,144.97) .. controls (85.24,143.54) and (70.38,136.87) .. (61.83,120.72) ;
\draw    (61.83,120.72) .. controls (61.83,108.72) and (77.5,103.92) .. (115.25,102.17) ;
\draw    (100.17,150.26) .. controls (79.75,146.05) and (65.83,136.22) .. (61.83,120.72) ;
\draw    (137.83,142.72) .. controls (88.83,163.72) and (72.59,138.14) .. (135.33,127.72) ;
\draw    (135.33,127.72) .. controls (123,119.67) and (127.5,111.67) .. (115.25,102.17) ;
\draw    (137.83,142.72) .. controls (135.25,138.42) and (129,135.17) .. (126.5,131.92) ;
\draw [color={rgb, 255:red, 0; green, 0; blue, 0 }  ,draw opacity=1 ]   (80.27,158.97) -- (80.27,173.48) ;
\draw [shift={(80.27,175.48)}, rotate = 270] [color={rgb, 255:red, 0; green, 0; blue, 0 }  ,draw opacity=1 ][line width=0.75]    (4.37,-1.32) .. controls (2.78,-0.56) and (1.32,-0.12) .. (0,0) .. controls (1.32,0.12) and (2.78,0.56) .. (4.37,1.32)   ;
\draw  [color={rgb, 255:red, 155; green, 155; blue, 155 }  ,draw opacity=0.5 ] (61.48,185.08) -- (145.83,185.08) -- (109.68,225.08) -- (25.33,225.08) -- cycle ;
\draw [color={rgb, 255:red, 0; green, 0; blue, 0 }  ,draw opacity=1 ][line width=1.5]    (102.67,219.63) .. controls (93.24,218.2) and (78.38,211.54) .. (69.83,195.38) ;
\draw    (226.17,105.38) .. controls (244.78,109.8) and (255.44,126.8) .. (270.17,129.38) ;
\draw    (226.17,105.38) .. controls (221.11,111.47) and (196.4,113.66) .. (183.19,136.3) ;
\draw    (243.17,170.38) .. controls (197.67,169.63) and (171.44,162.72) .. (183.44,158.47) .. controls (195.44,154.22) and (276.69,141.3) .. (183.19,136.3) ;
\draw    (210.5,140.93) .. controls (207.08,144.18) and (183.24,154.76) .. (181.08,159.68) ;
\draw [color={rgb, 255:red, 0; green, 0; blue, 0 }  ,draw opacity=1 ]   (220.27,175.97) -- (220.27,190.48) ;
\draw [shift={(220.27,192.48)}, rotate = 270] [color={rgb, 255:red, 0; green, 0; blue, 0 }  ,draw opacity=1 ][line width=0.75]    (4.37,-1.32) .. controls (2.78,-0.56) and (1.32,-0.12) .. (0,0) .. controls (1.32,0.12) and (2.78,0.56) .. (4.37,1.32)   ;
\draw [color={rgb, 255:red, 0; green, 0; blue, 0 }  ,draw opacity=1 ][line width=1.5]    (230.33,224.01) .. controls (225.33,218.86) and (221.83,205.01) .. (221.58,200.99) ;
\draw    (228.8,125.77) .. controls (229.8,131.77) and (229.99,139.01) .. (227.83,143.93) ;
\draw    (243.17,170.38) .. controls (251.17,150.38) and (267.67,148.38) .. (270.17,129.38) ;
\draw [color={rgb, 255:red, 0; green, 0; blue, 0 }  ,draw opacity=1 ][line width=1.5]    (192.58,224.26) .. controls (208.98,214.84) and (221.83,205.01) .. (221.58,200.99) ;
\draw  [color={rgb, 255:red, 155; green, 155; blue, 155 }  ,draw opacity=0.5 ] (200.48,185.08) -- (284.83,185.08) -- (248.68,225.08) -- (164.33,225.08) -- cycle ;

\end{tikzpicture}

\caption{Two folds join to make a crease. A cusp is created by projecting a crease.}\label{fig:forlds-cusps}

%% file: sing_3d.tex
\tikzset{every picture/.style={line width=0.75pt}} 

\begin{tikzpicture}[x=0.75pt,y=0.75pt,yscale=-1,xscale=1]

\draw    (92.03,138.08) .. controls (88.29,127.4) and (84.38,118.06) .. (86.04,97.02) ;
\draw    (70.7,133.45) .. controls (77.06,123.12) and (81.18,113.87) .. (86.04,97.02) ;
\draw    (181.83,110.31) .. controls (178.47,103.19) and (175.51,93.49) .. (166.87,87.05) ;
\draw [color={rgb, 255:red, 245; green, 166; blue, 35 }  ,draw opacity=1 ][line width=1.5]    (113.61,121.4) .. controls (173.83,78.38) and (204.83,81.38) .. (86.04,97.02) ;
\draw    (93.27,162.28) .. controls (101.83,153.74) and (105.84,148.22) .. (113.61,121.4) ;
\draw    (122.84,170.11) .. controls (113.94,160.96) and (111.94,142.45) .. (113.61,121.4) ;
\draw    (181.83,110.31) .. controls (188.94,120.99) and (153.85,137.7) .. (122.84,170.11) ;
\draw    (93.27,162.28) -- (111.49,152.79) ;
\draw    (70.7,133.45) -- (85.29,130.24) ;
\draw    (92.03,138.08) -- (106.25,133.8) ;
\draw  [color={rgb, 255:red, 190; green, 35; blue, 35 }  ,draw opacity=1 ][fill={rgb, 255:red, 190; green, 35; blue, 35 }  ,fill opacity=1 ] (164.84,87.05) .. controls (164.84,85.94) and (165.75,85.03) .. (166.87,85.03) .. controls (167.98,85.03) and (168.89,85.94) .. (168.89,87.05) .. controls (168.89,88.17) and (167.98,89.08) .. (166.87,89.08) .. controls (165.75,89.08) and (164.84,88.17) .. (164.84,87.05) -- cycle ;

\draw [color={rgb, 255:red, 245; green, 166; blue, 35 }  ,draw opacity=1 ][line width=1.5]    (459.38,127.94) -- (479.8,113.46) ;
\draw  [fill={rgb, 255:red, 209; green, 205; blue, 205 }  ,fill opacity=0.4 ] (458.56,95.78) -- (459.66,160.11) -- (419.07,160.66) -- (417.97,96.33) -- cycle ;
\draw  [fill={rgb, 255:red, 209; green, 205; blue, 205 }  ,fill opacity=0.4 ] (498.64,95.5) -- (499.74,159.84) -- (458.88,160.39) -- (457.78,96.05) -- cycle ;
\draw  [fill={rgb, 255:red, 74; green, 74; blue, 74 }  ,fill opacity=0.4 ] (435.9,113.46) -- (517.33,113.46) -- (477.82,142.44) -- (396.39,142.44) -- cycle ;
\draw  [fill={rgb, 255:red, 255; green, 255; blue, 255 }  ,fill opacity=0.4 ] (439.07,170.98) -- (438.84,106.24) -- (479.15,85.19) -- (479.38,149.94) -- cycle ;
\draw [color={rgb, 255:red, 245; green, 166; blue, 35 }  ,draw opacity=1 ][line width=1.5]    (416.81,127.95) -- (498.57,127.95) ;
\draw [color={rgb, 255:red, 245; green, 166; blue, 35 }  ,draw opacity=1 ][line width=1.5]    (459.75,160.19) -- (459.01,95.7) ;
\draw [color={rgb, 255:red, 255; green, 203; blue, 117 }  ,draw opacity=1 ][line width=1.5]    (460.12,160.15) -- (459.38,144.09) ;
\draw [color={rgb, 255:red, 255; green, 203; blue, 117 }  ,draw opacity=1 ][line width=1.5]    (416.81,127.95) -- (438.62,127.79) ;
\draw [color={rgb, 255:red, 255; green, 203; blue, 117 }  ,draw opacity=1 ][line width=1.5]    (459.38,127.94) -- (498.13,127.94) ;
\draw [color={rgb, 255:red, 255; green, 203; blue, 117 }  ,draw opacity=1 ][line width=1.5]    (459.38,127.94) -- (458.65,95.74) ;
\draw [color={rgb, 255:red, 255; green, 203; blue, 117 }  ,draw opacity=1 ][line width=1.5]    (438.97,142.43) -- (459.38,127.94) ;
\draw  [color={rgb, 255:red, 190; green, 35; blue, 35 }  ,draw opacity=1 ][fill={rgb, 255:red, 190; green, 35; blue, 35 }  ,fill opacity=1 ] (457.08,128.09) .. controls (457.08,126.97) and (457.99,126.06) .. (459.11,126.06) .. controls (460.23,126.06) and (461.14,126.97) .. (461.14,128.09) .. controls (461.14,129.21) and (460.23,130.11) .. (459.11,130.11) .. controls (457.99,130.11) and (457.08,129.21) .. (457.08,128.09) -- cycle ;

\draw    (254.4,87.47) -- (253.62,153.73) ;
\draw    (292.41,100.37) -- (291.62,167.77) ;
\draw    (334.33,155.17) .. controls (318.78,154.21) and (294.3,156.66) .. (291.62,167.77) .. controls (288.88,155.89) and (269.16,154.69) .. (253.62,153.73) ;
\draw    (334.28,120.58) -- (334.33,155.17) ;
\draw    (256.46,104.65) -- (249.66,136.21) ;
\draw  [fill={rgb, 255:red, 217; green, 217; blue, 217 }  ,fill opacity=1 ] (313.92,149.92) -- (232.41,147.55) -- (262.93,104.04) -- (344.45,106.41) -- cycle ;
\draw [color={rgb, 255:red, 0; green, 0; blue, 0 }  ,draw opacity=1 ][fill={rgb, 255:red, 255; green, 255; blue, 255 }  ,fill opacity=1 ][line width=0.75]    (293.19,101.87) .. controls (299.11,90.57) and (312.25,92.1) .. (334.33,92.51) .. controls (334.15,97.31) and (334.99,117.21) .. (334.28,120.58) .. controls (304.02,121.24) and (298.61,120.3) .. (291.82,130.98) .. controls (292.42,121.95) and (264.39,117.85) .. (254.05,115.96) .. controls (254.01,106.92) and (254.4,105.11) .. (254.4,87.47) .. controls (280.45,90.33) and (290.24,91.48) .. (293.19,101.87) -- cycle ;
\draw    (293.19,101.87) -- (292.41,129.24) ;
\draw [color={rgb, 255:red, 245; green, 166; blue, 35 }  ,draw opacity=1 ][fill={rgb, 255:red, 255; green, 255; blue, 255 }  ,fill opacity=0 ][line width=1.5]    (334.28,120.58) .. controls (317.12,121.21) and (297.01,118.82) .. (291.62,131.39) .. controls (295.43,123.51) and (267.59,117.83) .. (254.05,115.96) ;

\draw  [color={rgb, 255:red, 190; green, 35; blue, 35 }  ,draw opacity=1 ][fill={rgb, 255:red, 190; green, 35; blue, 35 }  ,fill opacity=1 ] (290.38,129.24) .. controls (290.38,128.12) and (291.29,127.22) .. (292.41,127.22) .. controls (293.52,127.22) and (294.43,128.12) .. (294.43,129.24) .. controls (294.43,130.36) and (293.52,131.27) .. (292.41,131.27) .. controls (291.29,131.27) and (290.38,130.36) .. (290.38,129.24) -- cycle ;


\end{tikzpicture}
\caption{The (intersections of) discriminants of the multi-germs in each of the singularity types.}\label{fig:sing-3d}

%% file: dessins.tex
\tikzset{every picture/.style={line width=0.75pt}} 

\begin{tikzpicture}[x=0.75pt,y=0.75pt,yscale=-1,xscale=1]

\draw  [draw opacity=0][dash pattern={on 4.5pt off 4.5pt}] (117.61,75.26) .. controls (117.92,75.16) and (118.23,75.11) .. (118.54,75.11) .. controls (125.15,75.11) and (130.5,98.09) .. (130.5,126.44) .. controls (130.5,154.79) and (125.15,177.78) .. (118.54,177.78) .. controls (118.2,177.78) and (117.87,177.71) .. (117.53,177.59) -- (118.54,126.44) -- cycle ; \draw  [color={rgb, 255:red, 245; green, 166; blue, 35 }  ,draw opacity=1 ][dash pattern={on 4.5pt off 4.5pt}] (117.61,75.26) .. controls (117.92,75.16) and (118.23,75.11) .. (118.54,75.11) .. controls (125.15,75.11) and (130.5,98.09) .. (130.5,126.44) .. controls (130.5,154.79) and (125.15,177.78) .. (118.54,177.78) .. controls (118.2,177.78) and (117.87,177.71) .. (117.53,177.59) ;  
\draw  [draw opacity=0] (377.41,127.26) .. controls (374.83,132.38) and (352.95,136.38) .. (326.33,136.38) .. controls (297.98,136.38) and (275,131.84) .. (275,126.25) .. controls (275,126.1) and (275.02,125.95) .. (275.05,125.81) -- (326.33,126.25) -- cycle ; \draw  [color={rgb, 255:red, 245; green, 166; blue, 35 }  ,draw opacity=1 ] (377.41,127.26) .. controls (374.83,132.38) and (352.95,136.38) .. (326.33,136.38) .. controls (297.98,136.38) and (275,131.84) .. (275,126.25) .. controls (275,126.1) and (275.02,125.95) .. (275.05,125.81) ;  
\draw  [color={rgb, 255:red, 155; green, 155; blue, 155 }  ,draw opacity=0.5 ] (275,125.94) .. controls (275,97.59) and (297.98,74.61) .. (326.33,74.61) .. controls (354.68,74.61) and (377.67,97.59) .. (377.67,125.94) .. controls (377.67,154.29) and (354.68,177.28) .. (326.33,177.28) .. controls (297.98,177.28) and (275,154.29) .. (275,125.94) -- cycle ;
\draw  [draw opacity=0][dash pattern={on 4.5pt off 4.5pt}] (275,125.94) .. controls (277.58,120.82) and (299.47,116.82) .. (326.08,116.82) .. controls (354.43,116.82) and (377.41,121.36) .. (377.41,126.96) .. controls (377.41,127.11) and (377.4,127.25) .. (377.36,127.4) -- (326.08,126.96) -- cycle ; \draw  [color={rgb, 255:red, 245; green, 166; blue, 35 }  ,draw opacity=1 ][dash pattern={on 4.5pt off 4.5pt}] (275,125.94) .. controls (277.58,120.82) and (299.47,116.82) .. (326.08,116.82) .. controls (354.43,116.82) and (377.41,121.36) .. (377.41,126.96) .. controls (377.41,127.11) and (377.4,127.25) .. (377.36,127.4) ;  
\draw  [draw opacity=0] (167.41,127.76) .. controls (164.83,132.88) and (142.95,136.88) .. (116.33,136.88) .. controls (87.98,136.88) and (65,132.34) .. (65,126.75) .. controls (65,126.6) and (65.02,126.45) .. (65.05,126.31) -- (116.33,126.75) -- cycle ; \draw  [color={rgb, 255:red, 245; green, 166; blue, 35 }  ,draw opacity=1 ] (167.41,127.76) .. controls (164.83,132.88) and (142.95,136.88) .. (116.33,136.88) .. controls (87.98,136.88) and (65,132.34) .. (65,126.75) .. controls (65,126.6) and (65.02,126.45) .. (65.05,126.31) ;  
\draw  [color={rgb, 255:red, 155; green, 155; blue, 155 }  ,draw opacity=0.5 ] (65,126.44) .. controls (65,98.09) and (87.98,75.11) .. (116.33,75.11) .. controls (144.68,75.11) and (167.67,98.09) .. (167.67,126.44) .. controls (167.67,154.79) and (144.68,177.78) .. (116.33,177.78) .. controls (87.98,177.78) and (65,154.79) .. (65,126.44) -- cycle ;
\draw  [draw opacity=0][dash pattern={on 4.5pt off 4.5pt}] (65,126.44) .. controls (67.58,121.32) and (89.47,117.32) .. (116.08,117.32) .. controls (144.43,117.32) and (167.41,121.86) .. (167.41,127.46) .. controls (167.41,127.61) and (167.4,127.75) .. (167.36,127.9) -- (116.08,127.46) -- cycle ; \draw  [color={rgb, 255:red, 245; green, 166; blue, 35 }  ,draw opacity=1 ][dash pattern={on 4.5pt off 4.5pt}] (65,126.44) .. controls (67.58,121.32) and (89.47,117.32) .. (116.08,117.32) .. controls (144.43,117.32) and (167.41,121.86) .. (167.41,127.46) .. controls (167.41,127.61) and (167.4,127.75) .. (167.36,127.9) ;  
\draw  [color={rgb, 255:red, 0; green, 0; blue, 0 }  ,draw opacity=1 ][fill={rgb, 255:red, 255; green, 255; blue, 255 }  ,fill opacity=1 ] (128.33,116.01) -- (132.16,116.01) -- (132.16,119.85) -- (128.33,119.85) -- cycle ;
\draw  [color={rgb, 255:red, 0; green, 0; blue, 0 }  ,draw opacity=1 ][fill={rgb, 255:red, 255; green, 255; blue, 255 }  ,fill opacity=1 ] (336.83,115.18) -- (340.66,115.18) -- (340.66,119.01) -- (336.83,119.01) -- cycle ;
\draw    (199,128.63) -- (246.33,128.55) ;
\draw [shift={(248.33,128.55)}, rotate = 179.9] [color={rgb, 255:red, 0; green, 0; blue, 0 }  ][line width=0.75]    (10.93,-3.29) .. controls (6.95,-1.4) and (3.31,-0.3) .. (0,0) .. controls (3.31,0.3) and (6.95,1.4) .. (10.93,3.29)   ;
\draw  [draw opacity=0] (119.47,177.62) .. controls (119.16,177.72) and (118.85,177.78) .. (118.54,177.78) .. controls (111.94,177.78) and (106.58,154.79) .. (106.58,126.44) .. controls (106.58,98.09) and (111.94,75.11) .. (118.54,75.11) .. controls (118.88,75.11) and (119.22,75.17) .. (119.55,75.29) -- (118.54,126.44) -- cycle ; \draw  [color={rgb, 255:red, 245; green, 166; blue, 35 }  ,draw opacity=1 ] (119.47,177.62) .. controls (119.16,177.72) and (118.85,177.78) .. (118.54,177.78) .. controls (111.94,177.78) and (106.58,154.79) .. (106.58,126.44) .. controls (106.58,98.09) and (111.94,75.11) .. (118.54,75.11) .. controls (118.88,75.11) and (119.22,75.17) .. (119.55,75.29) ;  
\draw  [draw opacity=0][line width=1.5]  (339.21,135.97) .. controls (334.58,136.24) and (329.69,136.38) .. (324.62,136.38) .. controls (315.38,136.38) and (306.71,135.9) .. (299.22,135.06) -- (324.62,126.25) -- cycle ; \draw  [line width=1.5]  (339.21,135.97) .. controls (334.58,136.24) and (329.69,136.38) .. (324.62,136.38) .. controls (315.38,136.38) and (306.71,135.9) .. (299.22,135.06) ;  
\draw  [color={rgb, 255:red, 0; green, 0; blue, 0 }  ,draw opacity=1 ][fill={rgb, 255:red, 255; green, 255; blue, 255 }  ,fill opacity=1 ] (296.65,135.06) .. controls (296.65,133.64) and (297.8,132.49) .. (299.22,132.49) .. controls (300.63,132.49) and (301.78,133.64) .. (301.78,135.06) .. controls (301.78,136.47) and (300.63,137.62) .. (299.22,137.62) .. controls (297.8,137.62) and (296.65,136.47) .. (296.65,135.06) -- cycle ;
\draw  [color={rgb, 255:red, 0; green, 0; blue, 0 }  ,draw opacity=1 ][fill={rgb, 255:red, 0; green, 0; blue, 0 }  ,fill opacity=1 ] (336.69,135.97) .. controls (336.69,134.58) and (337.82,133.45) .. (339.21,133.45) .. controls (340.6,133.45) and (341.72,134.58) .. (341.72,135.97) .. controls (341.72,137.36) and (340.6,138.48) .. (339.21,138.48) .. controls (337.82,138.48) and (336.69,137.36) .. (336.69,135.97) -- cycle ;
\draw  [draw opacity=0][line width=1.5]  (109.43,159.68) .. controls (107.65,150.73) and (106.58,139.12) .. (106.58,126.44) .. controls (106.58,112.69) and (107.84,100.19) .. (109.9,90.98) -- (118.54,126.44) -- cycle ; \draw  [color={rgb, 255:red, 0; green, 0; blue, 0 }  ,draw opacity=1 ][line width=1.5]  (109.43,159.68) .. controls (107.65,150.73) and (106.58,139.12) .. (106.58,126.44) .. controls (106.58,112.69) and (107.84,100.19) .. (109.9,90.98) ;  
\draw  [color={rgb, 255:red, 0; green, 0; blue, 0 }  ,draw opacity=1 ][fill={rgb, 255:red, 255; green, 255; blue, 255 }  ,fill opacity=1 ] (104.41,136.61) .. controls (104.41,135.19) and (105.56,134.04) .. (106.98,134.04) .. controls (108.39,134.04) and (109.54,135.19) .. (109.54,136.61) .. controls (109.54,138.03) and (108.39,139.17) .. (106.98,139.17) .. controls (105.56,139.17) and (104.41,138.03) .. (104.41,136.61) -- cycle ;
\draw  [color={rgb, 255:red, 0; green, 0; blue, 0 }  ,draw opacity=1 ][fill={rgb, 255:red, 0; green, 0; blue, 0 }  ,fill opacity=1 ] (107.38,90.98) .. controls (107.38,89.59) and (108.51,88.46) .. (109.9,88.46) .. controls (111.29,88.46) and (112.41,89.59) .. (112.41,90.98) .. controls (112.41,92.37) and (111.29,93.49) .. (109.9,93.49) .. controls (108.51,93.49) and (107.38,92.37) .. (107.38,90.98) -- cycle ;
\draw  [color={rgb, 255:red, 0; green, 0; blue, 0 }  ,draw opacity=1 ][fill={rgb, 255:red, 0; green, 0; blue, 0 }  ,fill opacity=1 ] (106.91,159.68) .. controls (106.91,158.29) and (108.04,157.17) .. (109.43,157.17) .. controls (110.82,157.17) and (111.94,158.29) .. (111.94,159.68) .. controls (111.94,161.07) and (110.82,162.2) .. (109.43,162.2) .. controls (108.04,162.2) and (106.91,161.07) .. (106.91,159.68) -- cycle ;

\draw (299.22,141.02) node [anchor=north] [inner sep=0.75pt]  [font=\scriptsize]  {$0$};
\draw (339.21,141.88) node [anchor=north] [inner sep=0.75pt]  [font=\scriptsize]  {$1$};
\draw (132.83,105.2) node [anchor=north west][inner sep=0.75pt]  [font=\scriptsize]  {$\infty $};
\draw (341.33,104.37) node [anchor=north west][inner sep=0.75pt]  [font=\scriptsize]  {$\infty $};
\draw (386.5,122.53) node [anchor=north west][inner sep=0.75pt]  [font=\scriptsize]  {$\mathbb{R} P^{1}$};
\draw (319.17,55.2) node [anchor=north west][inner sep=0.75pt]  [font=\scriptsize]  {$\mathbb{C} P^{1}$};
\draw (105.83,53.87) node [anchor=north west][inner sep=0.75pt]  [font=\scriptsize]  {$\mathbb{C} P^{1}$};
\draw (44.33,124.03) node [anchor=north west][inner sep=0.75pt]  [font=\scriptsize]  {$\Gamma _{f}$};
\draw (217.33,107.03) node [anchor=north west][inner sep=0.75pt]  [font=\scriptsize]  {$\tilde{f}$};

\end{tikzpicture}\caption{The polynomial dessin of the function $x\mapsto x^2$. Here, the orange graphs in $\P^1$ is sent to the orange graph in $\P_{\R}$, and the black-and-white graph is sent to the interval $[0,1]$.}\label{fig:dessins}

%% file: bibliography.bbl
\def\cprime{$'$}